\documentclass[12pt,a4paper]{amsart}
\usepackage{amssymb}
\usepackage[cmtip,all]{xy}
\usepackage{hyperref}

\calclayout

\newcommand{\+}{\nobreakdash-}
\renewcommand{\:}{\colon}

\newcommand{\rarrow}{\longrightarrow}

\newcommand{\ot}{\otimes}

\newcommand{\bu}{{\text{\smaller\smaller$\scriptstyle\bullet$}}}

\newcommand{\lrarrow}{\mskip.5\thinmuskip\relbar\joinrel\relbar\joinrel
 \rightarrow\mskip.5\thinmuskip\relax}
\newcommand{\llarrow}{\mskip.5\thinmuskip\leftarrow\joinrel\relbar
 \joinrel\relbar\mskip.5\thinmuskip\relax}
 
\DeclareMathOperator{\Hom}{Hom}
\DeclareMathOperator{\Ext}{Ext}

\DeclareMathOperator{\coker}{coker}
\DeclareMathOperator{\Tot}{Tot}

\newcommand{\Modl}{{\operatorname{\mathsf{--Mod}}}}
\newcommand{\Qcoh}{{\operatorname{\mathsf{--Qcoh}}}}
\newcommand{\Comodl}{{\operatorname{\mathsf{--Comod}}}}
\newcommand{\Contra}{{\operatorname{\mathsf{--Contra}}}}

\newcommand{\sA}{\mathsf A}
\newcommand{\sB}{\mathsf B}
\newcommand{\sC}{\mathsf C}
\newcommand{\sD}{\mathsf D}
\newcommand{\sE}{\mathsf E}
\newcommand{\sF}{\mathsf F}
\newcommand{\sG}{\mathsf G}
\newcommand{\sH}{\mathsf H}
\newcommand{\sK}{\mathsf K}

\newcommand{\sS}{\mathsf S}
\newcommand{\sX}{\mathsf X}

\newcommand{\cA}{\mathcal A}
\newcommand{\cF}{\mathcal F}
\newcommand{\cC}{\mathcal C}
\newcommand{\cO}{\mathcal O}
\newcommand{\cE}{\mathcal E}
\newcommand{\cM}{\mathcal M}

\newcommand{\fR}{\mathfrak R}

\newcommand{\VF}{\mathsf{VF}}
\newcommand{\CA}{\mathsf{CA}}
\newcommand{\Flat}{\mathsf{Flat}}
\newcommand{\Cot}{\mathsf{Cot}}

\newcommand{\Hot}{\mathsf{Hot}}
\newcommand{\Com}{\mathsf{Com}}
\newcommand{\Ac}{\mathsf{Ac}}
\newcommand{\DG}{\mathsf{DG}}
\newcommand{\Fil}{\mathsf{Fil}}

\newcommand{\proj}{\mathsf{proj}}
\newcommand{\inj}{\mathsf{inj}}
\renewcommand{\flat}{\mathsf{flat}}
\renewcommand{\cot}{\mathsf{cot}}
\newcommand{\vfl}{\mathsf{vfl}}
\newcommand{\vflp}{\mathsf{vflp}}
\newcommand{\cta}{\mathsf{cta}}
\newcommand{\bco}{\mathsf{bco}}
\newcommand{\bctr}{\mathsf{bctr}}

\newcommand{\dcot}{{\operatorname{\mathsf{-cot}}}}
\newcommand{\dcta}{{\operatorname{\mathsf{-cta}}}}
\newcommand{\dflpr}{{\operatorname{\mathsf{-flpr}}}}
\newcommand{\dvflp}{{\operatorname{\mathsf{-vflp}}}}
\newcommand{\dproj}{{\operatorname{\mathsf{-proj}}}}
\newcommand{\pdo}{{{\operatorname{\mathsf{pd-}}}1}}
\newcommand{\pdn}{{{\operatorname{\mathsf{pd-}}}n}}
\newcommand{\dpdo}{{{\operatorname{\mathsf{-pd-}}}1}}
\newcommand{\dflat}{{\operatorname{\mathsf{-flat}}}}

\newcommand{\id}{\mathrm{id}}

\newcommand{\boZ}{\mathbb Z}

\newcommand{\Section}[1]{\bigskip\section{#1}\medskip}
\theoremstyle{plain}
\newtheorem{thm}{Theorem}[section]
\newtheorem{prop}[thm]{Proposition}
\newtheorem{lem}[thm]{Lemma}
\newtheorem{cor}[thm]{Corollary}
\newtheorem{conj}[thm]{Conjecture}
\theoremstyle{definition}
\newtheorem{rem}[thm]{Remark}

\newtheorem{ex}[thm]{Example}
\newtheorem{exs}[thm]{Examples}

\begin{document}

\title{Coderived and contraderived categories \\
for a cotorsion pair, flat-type cotorsion pairs, \\
and relative periodicity}

\author{Leonid Positselski}

\address{Institute of Mathematics, Czech Academy of Sciences \\
\v Zitn\'a~25, 115~67 Prague~1 \\ Czech Republic} 

\email{positselski@math.cas.cz}

\begin{abstract}
 Given a hereditary complete cotorsion pair $(\sA,\sB)$ generated by
a set of objects in a Grothendieck category $\sK$, we construct
a natural equivalence between the Becker coderived category of
the left-hand class $\sA$ and the Becker contraderived category of
the right-hand class~$\sB$.
 We show that a nested pair of cotorsion pairs $(\sA_1,\sB_1)\le
(\sA_2,\sB_2)$ provides an adjunction between the related
co/contraderived categories, which is induced by a Quillen adjunction
between abelian model structures.
 Then we specialize to the cotorsion pairs $(\sF,\sC)$ sandwiched
between the projective and the flat cotorsion pairs in a module
category, and prove that the related co/contraderived categories for
$(\sF,\sC)$ are the same as for the projective and flat cotorsion pairs
if and only if two periodicity properties hold for $\sF$ and~$\sC$.
 The same applies to the cotorsion pairs sandwiched between the very
flat and the flat cotorsion pairs in the category of quasi-coherent
sheaves over a quasi-compact semi-separated scheme.
 More generally, we define and discuss cotorsion pairs of the very flat
type and of the flat type in Grothendieck categories (as well as exact
categories of the flat type), and work with a cotorsion pair sandwiched
between one of the very flat type and one of the flat type.
 The motivating examples of the classes of flaprojective modules and
relatively cotorsion modules for a ring homomorphism are discussed,
and periodicity conjectures formulated for them.
\end{abstract}

\maketitle

\tableofcontents

\section*{Introduction}
\medskip

 This paper combines three topics in homological algebra: derived
categories of the second kind, cotorsion pairs, and periodicity
properties.
 Let us explain the meaning of these terms one by one.

\subsection{{}} \label{introd-derived-subsecn}
 Given an abelian category $\sK$ with enough injective objects,
the bounded below derived category $\sD^+(\sK)$ can be equivalently
defined in two ways.
 The derived category $\sD^+(\sK)$ is the Verdier quotient category of
the homotopy category $\Hot^+(\sK)$ by the thick subcategory of
acyclic complexes $\Ac^+(\sK)\subset\Hot^+(\sK)$, and it is also 
equivalent to the homotopy category of the additive category
$\sK^\inj$ of injective objects in~$\sK$,
\begin{equation} \label{bounded-below-derived}
 \sD^+(\sK)=\Hot^+(\sK)/\Ac^+(\sK)\simeq\Hot^+(\sK^\inj).
\end{equation}
 Dually, if there are enough projective objects in $\sK$, then
the bounded above derived category $\sD^-(\sK)$ can be equivalently
defined as the Verdier quotient category of the homotopy category
$\Hot^-(\sK)$ by the thick subcategory of acyclic complexes
$\Ac^-(\sK)\subset\Hot^-(\sK)$, or as the homotopy category of
the additive category $\sK_\proj$ of projective objects in~$\sK$,
\begin{equation} \label{bounded-above-derived}
 \sD^-(\sK)=\Hot^-(\sK)/\Ac^-(\sK)\simeq\Hot^-(\sK_\proj).
\end{equation}
 Both the definitions of the derived category as the triangulated
Verdier quotient category and as the full subcategory of 
injective/projective resolutions in the homotopy category are important
for performing various tasks with derived categories, such as
constructing functors between them.
 Both the equivalences~\eqref{bounded-below-derived}
and~\eqref{bounded-above-derived} also hold for an exact category~$\sK$
(in the sense of Quillen) instead of an abelian one.

 For unbounded complexes, the triangulated equivalences
like~\eqref{bounded-below-derived} and~\eqref{bounded-above-derived}
continue to hold under additional assumptions, the most important and
restrictive of which is that the abelian category~$\sK$ should have 
finite homological dimension.
 However, in the general case, the unbounded versions
of~(\ref{bounded-below-derived}\+-\ref{bounded-above-derived})
fail, as the following thematic example illustrates.
 Let $\Lambda=k[\epsilon]/(\epsilon^2)$ be the algebra of dual numbers
over a field~$k$, and let $K^\bu$ be the complex of $\Lambda$\+modules
\begin{equation} \label{thematic-acyclic-complex-over-dual-numbers}
 \dotsb\overset{\epsilon*}\lrarrow\Lambda\overset{\epsilon*}\lrarrow
 \Lambda\overset{\epsilon*}\lrarrow\Lambda\overset{\epsilon*}\lrarrow
 \dotsb
\end{equation}
 Here $\Lambda$ is the free $\Lambda$\+module with one generator
and $\epsilon*\:\Lambda\rarrow\Lambda$ is the map of multiplication
with~$\epsilon$.
 Then $K^\bu$ is an acyclic, noncontractible complex of
projective-injective objects in the abelian category of modules
$\sK=\Lambda\Modl$.
 So $K^\bu$ represents a nonzero object in the unbounded homotopy
category $\Hot(\sK^\inj)=\Hot(\sK_\proj)$, but a zero object in
the unbounded derived category~$\sD(\sK)$.
 See~\cite[Examples~3.3]{Pkoszul}, \cite[Prologue]{Prel},
or~\cite[Section~7.4]{Pksurv} for a further discussion of
the complex~\eqref{thematic-acyclic-complex-over-dual-numbers}.

 The \emph{coderived category} of an exact category $\sK$ is most simply
defined as the homotopy category of unbounded complexes of injective
objects in~$\sK$.
 Similarly, the \emph{contraderived category} of $\sK$ is the homotopy
category of unbounded complexes of projective objects in~$\sK$.
 In the spirit of the discussion above, we prefer the somewhat more
complicated definitions representing the co/contraderived categories
as Verdier quotient categories of the homotopy category $\Hot(\sK)$
rather than the full subcategories of resolutions.
 Then the related triangulated equivalences become theorems which
can be proved in well-behaved cases:
\begin{alignat}{2}
 \sD^\bco(\sK) &= \Hot(\sK)/\Ac^\bco(\sK) &&\simeq\Hot(\sK^\inj),
 \label{introd-becker-coderived} \\
 \sD^\bctr(\sK) &= \Hot(\sK)/\Ac^\bctr(\sK) &&\simeq\Hot(\sK_\proj).
 \label{introd-becker-contraderived}
\end{alignat}

 Let us explain the notation in~(\ref{introd-becker-coderived}\+-%
\ref{introd-becker-contraderived}) and related terminology.
 There are two approaches to the co/contraderived categories:
the one named after the present author (as in~\cite{Psemi,Pkoszul,EP})
and the one named after H.~Becker~\cite{Bec}.
 The latter approach, going back to the papers~\cite{Jor,Kra,Neem},
was further developed in~\cite{Sto2,PS4,PS5}.
 In this paper, we consider the coderived and contraderived
categories in the sense of Becker.
 The letter~``$\mathsf{b}$'' in the superindices $\bco$ and $\bctr$
above stands for Becker.

 The thick subcategory of \emph{Becker-coacyclic complexes}
$\Ac^\bco(\sK)\subset\Hot(\sK)$ consists of all complexes $A^\bu$
in $\sK$ such that, for every complex of injective objects
$J^\bu\in\Hot(\sK^\inj)$, all morphisms of complexes
$A^\bu\rarrow J^\bu$ are homotopic to zero.
 Dually, the thick subcategory of \emph{Becker-contraacyclic complexes}
$\Ac^\bctr(\sK)\subset\Hot(\sK)$ consists of all complexes $B^\bu$
in $\sK$ such that, for every complex of projective objects
$P^\bu\in\Hot(\sK_\proj)$, all morphisms of complexes
$P^\bu\rarrow B^\bu$ are homotopic to zero.
 The \emph{Becker coderived category} and the \emph{Becker contraderived
category} of $\sK$ are constructed as the triangulated Verdier
quotient categories $\sD^\bco(\sK)=\Hot(\sK)/\Ac^\bco(\sK)$ and
$\sD^\bctr(\sK)=\Hot(\sK)/\Ac^\bctr(\sK)$.
 The coderived and contraderived categories are collectively referred
to as the \emph{derived categories of the second kind}.
 We refer to~\cite[Remark~9.2]{PS4} and~\cite[Section~7]{Pksurv}
for discussions of the history and philosophy of derived categories
of the second kind.

\subsection{{}}
 A \emph{cotorsion pair} in an abelian category $\sK$ can be
defined informally as a pair of classes of ``somewhat projective'' and
``somewhat injective'' objects that fit together in a suitable sense.
 More precisely, a pair of classes of objects $\sA$ and $\sB\subset\sK$
is said to be a cotorsion pair if $\Ext^1_\sK(A,B)=0$ for all objects
$A\in\sA$ and $B\in\sB$, and both the classes $\sA$ and $\sB$ are
maximal with this property with respect to each other.
 For example, $\sA=\sK$ and $\sB=\sK^\inj$ is a cotorsion pair in $\sK$,
called the \emph{injective cotorsion pair}.
 Dually, $\sA=\sK_\proj$ and $\sB=\sK$ is also a cotorsion pair in
$\sK$, called the \emph{projective cotorsion pair}.
 These are two trivial examples of cotorsion pairs.

 A nontrivial, and thematic, example of a cotorsion pair is the pair
of classes (flat modules, cotorsion modules) in the abelian category
$\sK=R\Modl$ of modules over an associative ring~$R$.
 Here a left $R$\+module $C$ is said to be \emph{cotorsion}
(in the sense of Enochs~\cite{En2}) if $\Ext^1_R(F,C)=0$ for all flat
left $R$\+modules~$F$.
 We denote the class of flat $R$\+modules by $R\Modl_\flat\subset
R\Modl$ and the class of cotorsion $R$\+modules by
$R\Modl^\cot\subset R\Modl$.
 The pair of classes $\sA=R\Modl_\flat$ and $\sB=R\Modl^\cot$ is called
the \emph{flat cotorsion pair} in $R\Modl$.

 The notion of a cotorsion pair (originally called ``cotorsion
theory'') was introduced by Salce~\cite{Sal}.
 The main result about cotorsion pairs, defining the modern shape of
the theory, is the theorem of Eklof and Trlifaj~\cite[Theorems~2
and~10]{ET} claiming that the cotorsion pair in $R\Modl$
generated by any \emph{set} (rather than a proper class)
of $R$\+modules is \emph{complete}.
 Let us explain what this means.

 In the context of injective or projective objects in an abelian
category, it is important to have enough of them.
 This was already mentioned in Section~\ref{introd-derived-subsecn}.
 What does it mean that there are enough objects in the classes $\sA$
and $\sB$ for a cotorsion pair $(\sA,\sB)$ in an abelian
category~$\sK$\,?
 The definition of a \emph{complete cotorsion pair} answers this
question, and the answer is unexpected and nontrivial.
 A cotorsion pair $(\sA,\sB)$ in $\sK$ is said to be complete if,
for every object $K\in\sK$, there exist short exact sequences
\begin{align*}
 & 0\rarrow B'\rarrow A\rarrow K\rarrow 0, \\
 & 0\rarrow K\rarrow B\rarrow A'\rarrow 0
\end{align*}
in $\sK$ with objects $A$, $A'\in\sA$ and $B$, $B'\in\sB$.

 A cotorsion pair $(\sA,\sB)$ in $\sK$ is said to be \emph{generated
by} a class of objects $\sS\subset\sA$ if, for any object $B\in\sK$,
one has $B\in\sB$ if and only if $\Ext^1_\sK(S,B)=0$ for all $S\in\sS$.
 In other words, $(\sA,\sB)$ is generated by $\sS$ if $\sA$ is
the (obviously unique) minimal left-hand class of a cotorsion pair
in $\sK$ such that $\sS\subset\sA$.
 In particular, the flat cotorsion pair in $R\Modl$ is generated
by the set of all flat $R$\+modules of the cardinality not exceeding
$\aleph_0$ plus the cardinality of~$R$.
 It follows that the flat cotorsion pair in $R\Modl$ is complete.
 These are the results of the paper~\cite[Section~2]{BBE}.

 A complete cotorsion pair $(\sA,\sB)$ in $\sK$ is said to be
\emph{hereditary} if $\Ext^n_\sK(A,B)=0$ for all $A\in\sA$, \,$B\in\sB$,
and $n\ge1$.
 Equivalently, $(\sA,\sB)$ is hereditary if and only if the class
$\sA$ is closed under kernels of epimorphisms in $\sK$, and if and only
if the class $\sB$ is closed under cokernels of monomorphisms in~$\sK$.
 It is clear from these definitions that (assuming existence of enough
injective/projective objects in~$\sK$) the injective and projective
cotorsion pairs in $\sK$ are hereditary, and so is the flat cotorsion
pair in $R\Modl$.
 We refer to the introduction to the paper~\cite{Pctrl} for a further
introductory discussion of cotorsion pairs.

\subsection{{}} \label{introd-periodicity-subsecn}
 \emph{Periodicity theorems} are a class of results mitigating
the counterexample~\eqref{thematic-acyclic-complex-over-dual-numbers}.
 In particular, the \emph{flat and projective periodicity theorem}
of Benson--Goodearl~\cite[Theorem~2.5]{BG} (rediscovered by
Neeman~\cite[Remark~2.15]{Neem}) can be restated by saying that
in any acyclic complex of projective modules with flat modules of
cocycles, the modules of cocycles are actually projective.
 The \emph{cotorsion periodicity theorem} of Bazzoni,
Cort\'es-Izurdiaga, and Estrada~\cite[Theorem~1.2(2),
Proposition~4.8(2), or Theorem~5.1(2)]{BCE} claims that in any
acyclic complex of cotorsion modules, the modules of cocycles
are cotorsion.
 A survey of the most important periodicity theorems for modules over
associative rings can be found in the introduction to
the paper~\cite{BHP}.

 In this paper, we are interested in certain relative versions of
periodicity properties.
 The word ``relative'' means that the classes of modules involved
arise in the context of a ring homomorphism $R\rarrow A$.
 Following~\cite[Section~1.6]{Pdomc} and~\cite{PBas}, we say
that a left $A$\+module $F$ is \emph{$A/R$\+flaprojective} if
$\Ext^1_A(F,C)=0$ for every left $A$\+module $C$ that is \emph{cotorsion
as a left $R$\+module}.
 Assuming that the left $R$\+module $A$ is projective, we also say
that a left $A$\+module $C$ is \emph{$A/R$\+cotorsion} if
$\Ext^1_A(F,C)=0$ for every flat left $A$\+module $F$ that is
\emph{projective as a left $R$\+module}.
 The aim of this paper is to discuss the following two conjectures.

\begin{conj} \label{flaprojective-conjecture}
 Let $R\rarrow A$ be a homomorphism of associative rings, and let
$F^\bu$ be a complex of $A/R$\+flaprojective left $A$\+modules.
 Assume that the complex $F^\bu$ is acyclic with flat $A$\+modules
of cocycles.
 Then the $A$\+modules of cocycles of $F^\bu$ are actually
$A/R$\+flaprojective.
\end{conj}

\begin{conj} \label{relatively-cotorsion-conjecture}
 Let $R\rarrow A$ be a homomorphism of associative rings such that
$A$ is a projective left $R$\+module.
 Let $C^\bu$ be an acyclic complex of $A/R$\+cotorsion left
$A$\+modules.
 Then the $A$\+modules of cocycles of the complex $C^\bu$ are
also $A/R$\+co\-tor\-sion.
\end{conj}

 Conjecture~\ref{flaprojective-conjecture} is relevant in the context
of the discussion of the \emph{contraderived categories of
$X$\+cotorsion contraherent $\mathcal A$\+modules over quasi-compact
semi-separated schemes~$X$} in the paper~\cite{Pdomc}.
 That is where our original motivation for writing the present paper
comes from.
 We state Conjecture~\ref{relatively-cotorsion-conjecture} in
this paper mostly for the purposes of comparison with
Conjecture~\ref{flaprojective-conjecture}.

 In order to give some context and present equivalent reformulations
of Conjectures~\ref{flaprojective-conjecture}
and~\ref{relatively-cotorsion-conjecture}, we discuss the general
setting of a hereditary complete cotorsion pair $(\sF,\sC)$ in
$A\Modl$ that is sandwiched between the projective and flat cotorsion
pairs, i.~e., $A\Modl_\proj\subset\sF\subset A\Modl_\flat$.
 We also consider cotorsion pairs $(\sF,\sC)$ in the abelian category
of quasi-coherent sheaves $X\Qcoh$ on a quasi-compact semi-separated
scheme $X$ that are similarly sandwiched between the \emph{very flat}
and the flat cotorsion pair in $X\Qcoh$.

\subsection{{}}
 Now let us describe the main results of this paper.
 Let $\sK$ be a Grothendieck category and $(\sA,\sB)$ be a hereditary
complete cotorsion pair generated by a set of objects in~$\sK$.
 We start with discussing certain hereditary complete cotorsion pairs
related to $(\sA,\sB)$ in the Grothendieck abelian category of
complexes $\Com(\sA)$.
 Using these cotorsion pairs, we prove that the triangulated
equivalence~\eqref{introd-becker-coderived} holds for the exact
category~$\sA$,
$$
 \Hot(\sA^\inj)\simeq\Hot(\sA)/\Ac^\bco(\sA)=\sD^\bco(\sA),
$$
while the triangulated equivalence~\eqref{introd-becker-contraderived}
holds for the exact category~$\sB$,
$$
 \Hot(\sB_\proj)\simeq\Hot(\sB)/\Ac^\bctr(\sB)=\sD^\bctr(\sB).
$$
 In fact, one has $\sA^\inj=\sA\cap\sB=\sB_\proj$, so we obtain
natural triangulated equivalences
\begin{equation} \label{introd-co-contra-equiv}
 \sD^\bco(\sA)\simeq\Hot(\sA\cap\sB)\simeq\sD^\bctr(\sB).
\end{equation}

 We also show that every Becker-coacyclic complex in $\sA$ is acyclic
in~$\sA$ (i.~e., it is an acyclic complex in $\sK$ with the objects of
cocycles belonging to~$\sA$).
 Similarly, every Becker-contraacyclic complex in $\sB$ is acyclic
in~$\sB$ (once again, this means that the complex is acyclic in $\sK$
with the objects of cocycles belonging to~$\sB$).
 So we have
$$
 \Ac^\bco(\sA)\subset\Ac(\sA)
 \quad\textup{and}\quad
 \Ac^\bctr(\sB)\subset\Ac(\sB).
$$
 It is an open problem whether these properties hold for exact
categories with enough injective/projective objects in general;
see~\cite[Lemma~B.7.3 and Remark~B.7.4]{Pcosh}.

 Secondly, we consider a nested pair of cotorsion pairs $(\sA_1,\sB_1)$
and $(\sA_2,\sB_2)$ in~$\sK$.
 The word ``nested'' means that $\sA_1\subset\sA_2$, or equivalently,
$\sB_1\supset\sB_2$.
 In this context, we prove that the inclusions of exact categories
$\sA_1\rarrow\sA_2$ and $\sB_2\rarrow\sB_1$ induce well-defined
triangulated functors between the coderived/contraderived categories
$$
 \sD^\bco(\sA_1)\lrarrow\sD^\bco(\sA_2)
 \quad\text{and}\quad
 \sD^\bctr(\sB_2)\lrarrow\sD^\bctr(\sB_1).
$$
 In other words, this means that the inclusions
$\Ac^\bco(\sA_1)\subset\Ac^\bco(\sA_2)$ and $\Ac^\bctr(\sB_2)\subset
\Ac^\bctr(\sB_1)$ hold in $\Hot(\sK)$.
 Generally speaking, it may be difficult to prove the preservation
of Becker co/contraacyclicity by a functor, even if the functor
is exact and preserves infinite direct sums/direct products;
see~\cite[Lemma~B.7.5]{Pcosh} and~\cite[Lemmas~7.16\+-7.17,
Sections~7.9 and~8, and Remarks~11.18 and~11.29]{Pdomc}.

 Using the equivalences~\eqref{introd-co-contra-equiv}, we obtain
a pair of triangulated functors in the opposite directions
$$
 \xymatrix{
  \Hot(\sA_1\cap\sB_1) \ar@<2pt>[r] &
  \Hot(\sA_2\cap\sB_2) \ar@<2pt>[l]
 }
$$
 We prove that the functor $\sD^\bco(\sA_1)\rarrow\sD^\bco(\sA_2)$ is 
left adjoint to the functor $\sD^\bctr(\sB_2)\rarrow\sD^\bctr(\sB_1)$.

\subsection{{}}
 The more special classes of cotorsion pairs studied in this paper are
the cotorsion pairs \emph{of the very flat type} and \emph{of the flat
type} in Grothendieck categories~$\sK$.
 We say that a hereditary complete cotorsion pair $(\VF,\CA)$ generated
by a set of objects in $\sK$ is \emph{of the very flat type} if all
the objects of $\VF$ have finite projective dimensions in~$\sK$.
 Given a cotorsion pair $(\VF,\CA)$ and a hereditary complete cotorsion
pair $(\Flat,\Cot)$ generated by a set of objects in $\sK$ such that
$\VF\subset\Flat$, we say that the cotorsion pair $(\Flat,\Cot)$ is
\emph{of the flat type} if the exact category $\Flat\cap\CA\subset\sK$
is of the flat type.

 Let us explain what the latter condition means.
 Given an exact category $\sE$ with enough injective and projective
objects, we say that $\sE$ is \emph{weakly of the flat type} if
the classes of acyclic, Becker-coacyclic, and Becker-contraacyclic
complexes coincide in~$\sE$, that is
$$
 \Ac^\bco(\sE)=\Ac(\sE)=\Ac^\bctr(\sE).
$$
 An exact category $\sE$ is said to be \emph{of the flat type} if,
in addition to the previous conditions, the Becker-coderived and
the Becker-contraderived categories of $\sE$ are ``well-behaved''
in the sense that the triangulated
equivalences~(\ref{introd-becker-coderived}\+-%
\ref{introd-becker-contraderived}) hold for~$\sE$.

 All exact categories of finite homological dimension having enough
projective and injective objects are of the flat
type~\cite[Proposition~7.5]{Pphil}.
 Surprisingly, however, there are exact categories of much more
general and complicated nature that are still of the flat type.
 For example, the exact category $\sE=R\Modl_\flat$ of flat modules
over an arbitrary associative ring $R$ is of the flat
type~\cite[Theorems~7.14 and~7.18]{Pphil}.
 This result is closely related to the two periodicity theorems
discussed in Section~\ref{introd-periodicity-subsecn}.

 Specifically, the assertion that the classes of acyclic and
Becker-contraacyclic complexes in the exact category $R\Modl_\flat$
coincide is essentially a restatement of Neeman's flat and projective
module theorem~\cite[Theorem~8.6\,(i)\,$\Leftrightarrow$\,(3)]{Neem},
which is a strong version of the flat and projective periodicity
theorem (as explained in~\cite[Remark~2.15]{Neem}).
 The assertion that the classes of acyclic and Becker-coacyclic
complexes in $R\Modl_\flat$ coincide is equivalent to the cotorsion
periodicity theorem of Bazzoni, Cort\'es-Izurdiaga, and
Estrada~\cite{BCE} according to the quite general result which
we prove in the present paper, see
Theorem~\ref{cotorsion-periodity-type-equivalent-conditions}.
{\uchyph=0\par}

 Returning to cotorsion pairs of the very flat and flat type in
Grothendieck categories, there are two main intended examples
in the present paper.
 One example is the module category $\sK=A\Modl$ for an arbitrary
associative ring $\sA$, when we take $\VF=A\Modl_\proj$ and
$\CA=A\Modl$, and also $\Flat=A\Modl_\flat$ and $\Cot=A\Modl^\cot$.

 The other example is the category of quasi-coherent sheaves
$\sK=X\Qcoh$ for a quasi-compact semi-separated scheme~$X$.
 In this case, we take $\VF=X\Qcoh_\vfl$ to be the class of
\emph{very flat} quasi-coherent sheaves on~$X$
\,\cite[Section~1.10]{Pcosh}, \cite[Section~0.5]{PSl1} and
$\CA=X\Qcoh^\cta$ to be the class of \emph{contraadjusted}
quasi-coherent sheaves on~$X$
\,\cite[Sections~2.5 and~4.1]{Pcosh}, \cite[Section~3.5]{PRoos}.
 We also take $\Flat=X\Qcoh_\flat$ to be the class of flat
quasi-coherent sheaves on~$X$, and $\Cot=X\Qcoh^\cot$ to be the class
of cotorsion quasi-coherent sheaves on~$X$ \,\cite{EE}.

 The assertion that the exact category of flat contraadjusted
quasi-coherent sheaves $\sE=X\Qcoh_\flat^\cta$ on a quasi-compact
semi-separated scheme~$X$ is of the flat type can be found in
the book manuscript~\cite[Corollary~5.5.11]{Pcosh}.

 Given a pair of cotorsion pairs $(\VF,\CA)$ of the very flat type
and $(\Flat,\Cot)$ of the flat type such that $\VF\subset\Flat$,
we prove several triangulated equivalences of
the derived/coderived/contraderived categories that were previously
known to hold for flat modules and flat quasi-coherent sheaves.
 Specifically,
\begin{itemize}
\item the classes of acyclic and Becker-coacyclic complexes in
the exact category $\Flat$ coincide, so $\sD^\bco(\Flat)=\sD(\Flat)$;
\item the inclusion functors $\Cot\rarrow\CA\rarrow\sK$ induce
triangulated equivalences of the conventional derived categories
$$
 \sD(\Cot)\simeq\sD(\CA)\simeq\sD(\sK);
$$
\item the inclusion functor $\VF\rarrow\Flat$ induces a triangulated
equivalence of the conventional derived categories
$$
 \sD(\VF)\simeq\sD(\Flat),
$$
which can be also interpreted as a triangulated equivalence of
the Becker coderived categories
$$
 \sD^\bco(\VF)\simeq\sD^\bco(\Flat);
$$
\item the inclusion functor $\Cot\rarrow\CA$ induces a triangulated
equivalence of the Becker contraderived categories
$$
 \sD^\bctr(\Cot)\simeq\sD^\bctr(\CA).
$$
\end{itemize}
 These are the results of our
Theorems~\ref{flat-type-conventional-derived-equivalence-theorem}
and~\ref{flat-type-contraderived-equivalences-theorem}.

\subsection{{}}
 Specializing the discussion of cotorsion pairs $(\sA,\sB)$ above to
the case of cotorsion pairs $(\sF,\sC)$ sandwiched between
a very flat-type and a flat-type cotorsion pair, we consider
a Grothendieck category $\sK$ with three cotorsion pairs $(\VF,\CA)$,
\,$(\sF,\sC)$, and $(\Flat,\Cot)$.
 Here the cotorsion pairs $(\VF,\CA)$ and $(\Flat,\Cot)$ are presumed
to be in some sense well understood, while the intermediate cotorsion
pair $(\sF,\sC)$ is the intended object of study.

 To repeat, we make the following assumptions:
\begin{enumerate}
\renewcommand{\theenumi}{\roman{enumi}}
\item each of the three cotorsion pairs $(\VF,\CA)$, \,$(\sF,\sC)$,
and $(\Flat,\Cot)$ is hereditary, complete, and generated by a set
of objects;
\item one has $\VF\subset\sF\subset\Flat$, or equivalently,
$\CA\supset\sC\supset\Cot$;
\item all the objects of $\VF$ have finite projective dimensions
in~$\sK$ (bounded by a fixed constant);
\item for any given acyclic complex $F^\bu$ in\/ $\sK$ with the terms 
belonging to\/ $\Flat\cap\CA$, the following two conditions are
equivalent: the objects of cocycles of $F^\bu$ belong to\/ $\Flat$
if and only if, for every complex $P^\bu$ in\/ $\sK$ with the terms
belonging to\/ $\VF$, every morphism of complexes $P^\bu\rarrow F^\bu$
is homotopic to zero;
\item in any acyclic complex in\/ $\sK$ with the terms belonging to\/
$\Cot$, the objects of cocycles also belong to\/~$\Cot$.
\end{enumerate}
 These assumptions are known to hold for the cotorsion pairs
$(\VF,\CA)$ and $(\Flat,\Cot)$ in the two examples above.
 In particular, condition~(iv) is a ``strong version of the flat and
very flat periodicity property for $\Flat$ and $\VF$\,'' generalizing
the strong version of the flat and projective periodicity
theorem~\cite[Theorem~8.6\,(iii)\,$\Leftrightarrow$\,(i)]{Neem}.
 Condition~(v) is a ``cotorsion periodicity property for~$\Cot$''
similar to the cotorsion periodicity theorem~\cite[Theorem~5.1(2)]{BCE}.

 Condition~(iii) says that the cotorsion pair $(\VF,\CA)$ is of
the very flat type.
 Conditions~(iv) and~(v), taken together, are equivalent to saying
that the cotorsion pair $(\Flat,\Cot)$ is of the flat type.

 Under the assumptions listed above, we construct a \emph{recollement}
involving various derived, coderived, and contraderived categories
associated with the three cotorsion pairs $(\VF,\CA)$,
\,$(\sF,\sC)$, and $(\Flat,\Cot)$.
 See Theorem~\ref{sandwiched-recollement-theorem}.

 We also prove the following result (the full version of this
theorem, with two additional equivalent conditions, is
Theorem~\ref{sandwiched-equivalent-conditions-theorem}).

\begin{thm} \label{introd-main-theorem}
 The following eight conditions are equivalent:
\begin{enumerate}
\item the cotorsion pair $(\sF,\sC)$ is of the flat type;
\item the triangulated functor\/ $\sD^\bco(\VF)\rarrow\sD^\bco(\sF)$
induced by the inclusion\/ $\VF\rarrow\sF$ is a triangulated
equivalence;
\item the triangulated functor\/ $\sD^\bco(\sF)\rarrow\sD^\bco(\Flat)$
induced by the inclusion\/ $\sF\rarrow\Flat$ is a triangulated
equivalence;
\item any complex in\/ $\sF$ that is Becker-coacyclic in\/ $\Flat$ is
also Becker-coacyclic in\/~$\sF$;
\item the triangulated functor\/ $\sD^\bctr(\Cot)\rarrow\sD^\bctr(\sC)$
induced by the inclusion\/ $\Cot\rarrow\sC$ is a triangulated
equivalence;
\item the triangulated functor\/ $\sD^\bctr(\sC)\rarrow\sD^\bctr(\CA)$
induced by the inclusion\/ $\sC\rarrow\CA$ is a triangulated
equivalence;
\item any complex in\/ $\sC$ that is Becker-contraacyclic in\/ $\CA$ is
also Becker-con\-tra\-acyclic in\/~$\sC$;
\item two periodicity properties hold simultaneously:
\begin{enumerate}
\renewcommand{\theenumii}{\alph{enumii}}
\item in any acyclic complex in\/ $\sK$ with the terms belonging to\/
$\sF$ and the objects of cocycles belonging to\/ $\Flat$, the objects
of cocycles actually belong to\/~$\sF$;
\item in any acyclic complex in\/ $\sK$ with the terms belonging to\/
$\sC$, the objects of cocycles also belong to\/~$\sC$.
\end{enumerate}
\end{enumerate}
\end{thm}

 In particular, Theorem~\ref{introd-main-theorem} is applicable to
both the settings of Conjectures~\ref{flaprojective-conjecture}
and~\ref{relatively-cotorsion-conjecture}.
 In both cases, the theorem provides several equivalent reformulations
of the respective conjecture.

 In the context of Conjecture~\ref{flaprojective-conjecture}, one
takes $\sC$ to be the class of all left $A$\+modules that are
cotorsion over~$R$; then $\sF$ is the class of $A/R$\+flaprojective
left $A$\+modules.
 Then property~(8b) follows from the cotorsion periodicity theorem of
Bazzoni, Cort\'es-Izurdiaga, and Estrada~\cite{BCE} applied to
the ring~$R$.
 So condition~(8) reduces to~(8a), which is the assertion of
Conjecture~\ref{flaprojective-conjecture}.
 Therefore, every one of the conditions~(1\+-7) is an equivalent 
reformulation of the conjecture.

 In the context of Conjecture~\ref{relatively-cotorsion-conjecture},
one takes $\sF$ to be the class of all flat left $A$\+modules that are
projective over~$R$; then $\sC$ is the class of $A/R$\+cotorsion
left $A$\+modules.
 Then property~(8a) follows from the flat and projective periodicity
theorem of Benson--Goodearl and Neeman~\cite{BG,Neem} applied to
the ring~$R$.
 So condition~(8) reduces to~(8b), which is the assertion of
Conjecture~\ref{relatively-cotorsion-conjecture}.
 Thus, every one of the conditions~(1\+-7) is an equivalent 
reformulation of the conjecture.

\subsection*{Acknowledgement}
 I~am grateful to Jan \v St\!'ov\'\i\v cek and Jan Trlifaj for helpful
conversations.
 I~also wish to thank two anonymous referees for several helpful
suggestions.
 The author is supported by the GA\v CR project 23-05148S
and the Institute of Mathematics, Czech Academy of Sciences
(research plan RVO:~67985840).

\Section{Preliminaries}  \label{preliminaries-secn}

 We suggest the survey paper~\cite{Bueh} as the background reference
source on exact categories (in the sense of Quillen).
 In the context of the present paper, one can safely assume that all
the exact categories under consideration are idempotent-complete.

 Given an additive category $\sK$, we denote by $\Com(\sK)$ the additive
category of (unbounded) complexes in~$\sK$.
 The morphisms in $\Com(\sK)$ are the usual (closed) morphisms of
complexes, i.~e., homogeneous morphisms of degree~$0$ commuting with
the differentials.
 The homotopy category of complexes in $\sK$ is denoted by $\Hot(\sK)$;
so $\Hot(\sK)$ is the additive quotient category of $\Com(\sK)$ by
the ideal of morphisms that are cochain homotopic to zero.

 Given a complex $C^\bu$ in $\sK$ and an integer $n\in\boZ$, we denote
by $C^\bu[n]$ the cohomological degree shift of the complex~$C^\bu$;
so $C^\bu[n]^i=C^{n+i}$ for all $i\in\boZ$.
 For any two complexes $A^\bu$ and $B^\bu$ in $\sK$, we denote by
$\Hom_\sK(A^\bu,B^\bu)$ the total complex of the bicomplex of abelian
groups $\Hom_\sK(A^i,B^j)$, \ $i$, $j\in\boZ$, constructed by taking
infinite direct products along the diagonals.
 For any objects $A$ and $B$ in $\sK$, the notation $\Hom_\sK(A,B^\bu)$
and $\Hom_\sK(A^\bu,B)$ has the similar (and even more obvious) meaning.
 By the definition, we have $\Hom_{\Hot(\sK)}(A^\bu,B^\bu[n])=
H^n\Hom_\sK(A^\bu,B^\bu)$.

 Let $\sK$ be an exact category.
 The (conventional unbounded) derived category $\sD(\sK)$ is defined
as the triangulated Verdier quotient category
$$
 \sD(\sK)=\Hot(\sK)/\Ac(\sK)
$$
of the homotopy category of complexes $\Hot(\sK)$ by the thick
subcategory of acyclic complexes $\Ac(\sK)\subset\Hot(\sK)$
\,\cite{Neem0}, \cite[Section~10]{Bueh}.
 We will use the same notation $\Ac(\sK)\subset\Com(\sK)$ for the full
subcategory of acyclic complexes in $\Com(\sK)$.

 The Yoneda Ext functor in an exact category $\sK$ can be defined
as $\Ext_\sK^n(X,Y)=\Hom_{\sD(\sK)}(X,Y[n])$ for all $X$, $Y\in\sK$
and $n\in\boZ$.
 One can easily check that $\Ext_\sK^n(X,Y)=0$ for $n<0$
and $\Ext_\sK^0(X,Y)=\Hom_\sK(X,Y)$, while $\Ext_\sK^1(X,Y)$ is
the set of all equivalence classes of admissible short exact sequences
$0\rarrow Y\rarrow Z\rarrow X\rarrow0$ in~$\sK$.

 The higher Ext groups $\Ext_\sK^n(X,Y)$ are also produced by
the classical construction of the equivalence classes of exact
sequences $0\rarrow Y\rarrow Z^1\rarrow\dotsb\rarrow Z^n\rarrow X
\rarrow0$ \cite[Proposition~A.13 of the published version or
Proposition in Section~A.7 of the \texttt{arXiv} version]{Partin}.
 This assertion is easily provable using the description of the groups
of morphisms in the derived category $\sD(\sK)$ as fractions (``roofs'')
of morphisms in the homotopy category $\Hot(\sK)$.
 The fact that exact complexes in an exact category have canonical
truncations (which arbitrary complexes in exact categories do not have,
in general) plays an important role here.

 Let $\sA$, $\sB\subset\sK$ be two classes of objects.
 The notation $\sA^{\perp_1}\subset\sK$ stands for the class of all
objects $X\in\sK$ such that $\Ext^1_\sK(A,X)=0$ for all $A\in\sA$.
 Dually, ${}^{\perp_1}\sB\subset\sK$ is the class of all objects
$Y\in\sK$ such that $\Ext^1_\sK(Y,B)=0$ for all $B\in\sB$.

 A pair of classes of objects $(\sA,\sB)$ in $\sK$ is said to be
a \emph{cotorsion pair} if $\sB=\sA^{\perp_1}$ and
$\sA={}^{\perp_1}\sB$.
 The intersection of the two classes $\sA\cap\sB\subset\sK$ is called
the \emph{kernel} (or in a different terminology, the \emph{core})
of a cotorsion pair $(\sA,\sB)$.

 For any class of objects $\sS\subset\sK$, the pair of classes
$\sB=\sS^{\perp_1}$ and $\sA={}^{\perp_1}\sB$ is a cotorsion pair
in~$\sK$.
 The cotorsion pair $(\sA,\sB)$ is said to be \emph{generated by}~$\sS$.

 A cotorsion pair $(\sA,\sB)$ in an exact category $\sK$ is said to be
\emph{complete} if, for every object $K\in\sK$, there exist admissible
short exact sequences
\begin{align}
 & 0\rarrow B'\rarrow A\rarrow K\rarrow 0,
 \label{special-precover-sequence} \\
 & 0\rarrow K\rarrow B\rarrow A'\rarrow 0
 \label{special-preenvelope-sequence}
\end{align}
in $\sK$ with objects $A$, $A'\in\sA$ and $B$, $B'\in\sB$.
 The sequence~\eqref{special-precover-sequence} is called
a \emph{special precover sequence}.
 The sequence~\eqref{special-preenvelope-sequence} is called
a \emph{special preenvelope sequence}.
 Collectively, the sequences~(\ref{special-precover-sequence}\+-%
\ref{special-preenvelope-sequence}) are referred to as
the \emph{approximation sequences}.

 A class of objects $\sA$ in an exact category $\sK$ is said to be
\emph{generating} if, for every object $K\in\sK$, there exists
an object $A\in\sA$ together with an admissible epimorphism $A\rarrow K$
in~$\sK$.
 Dually, a class of objects $\sB\subset\sK$ is said to be
\emph{cogenerating} if, for every object $K\in\sK$, there exists
an object $B\in\sB$ together with an admissible monomorphism
$K\rarrow B$ in~$\sK$.

 Clearly, in any complete cotorsion pair $(\sA,\sB)$, the class $\sA$
is generating and the class $\sB$ is cogenerating.
 For any cotorsion pair $(\sA,\sB)$ in $\sK$, all the projective
objects of $\sK$ belong to $\sA$, and all the injective objects of $\sK$
belong to~$\sB$.
 Hence the class $\sA$ is always generating if there are enough
projective objects in $\sK$, and the class $\sB$ is always cogenerating
if there are enough injective objects in~$\sK$.

 Let $(\sA,\sB)$ be a cotorsion pair in $\sK$ such that the class $\sA$
is generating and the class $\sB$ is cogenerating.
 Then one says that the cotorsion pair $(\sA,\sB)$ is \emph{hereditary}
if any one of the following equivalent conditions holds:
\begin{enumerate}
\item the class $\sA$ is closed under kernels of admissible epimorphisms
in~$\sK$;
\item the class $\sB$ is closed under cokernels of admissible
monomorphisms in~$\sK$;
\item one has $\Ext^2_\sK(A,B)=0$ for all $A\in\sA$ and $B\in\sB$;
\item one has $\Ext^n_\sK(A,B)=0$ for all $A\in\sA$, \,$B\in\sB$,
and $n\ge1$.
\end{enumerate}
 The observation that conditions~(1\+-4) are equivalent goes back
to~\cite[Theorem~1.2.10]{GR} (cf.~\cite[Lemma~5.24]{GT}
and~\cite[Lemma~2.3]{Gil2a}).
 An exposition in the stated generality can be found
in~\cite[Lemma~6.17]{Sto-ICRA}; for a further generalization,
see~\cite[Lemma~7.1]{PS6}.

 Let $\sK$ be an exact category and $\sE\subset\sK$ be a full
additive subcategory closed under extensions.
 Then the category $\sE$ is endowed with the \emph{inherited exact
structure} in which the admissible short exact sequences in $\sE$ are
the admissible short exact sequences in $\sK$ with the terms belonging
to~$\sE$.
 Throughout this paper, we always presume the inherited exact structure
when speaking about exact structures on full subcategories of our
exact/abelian categories.
 One says that a complete cotorsion pair $(\sA,\sB)$ in $\sK$
\emph{restricts to} (a complete cotorsion pair in) $\sE$ if the pair
of classes $(\sE\cap\sA,\>\sE\cap\sB)$ is a complete cotorsion pair
in~$\sE$.

\begin{lem} \label{restricting-cotorsion-pair-lemma}
 Let\/ $\sK$ be an exact category, $\sE\subset\sK$ be a full additive
subcategory, and $(\sA,\sB)$ be a complete cotorsion pair in\/~$\sK$.
 In this context: \par
\textup{(a)} if the full subcategory\/ $\sE$ is closed under extensions
and kernels of admissible epimorphisms in\/ $\sK$, and if\/
$\sA\subset\sE$, then the cotorsion pair $(\sA,\sB)$ restricts to
a complete cotorsion pair $(\sA,\>\sE\cap\sB)$ in\/~$\sE$; \par
\textup{(b)} if the full subcategory\/ $\sE$ is closed under extensions
and cokernels of admissible monomorphisms in\/ $\sK$, and if\/
$\sB\subset\sE$, then the cotorsion pair $(\sA,\sB)$ restricts to
a complete cotorsion pair $(\sE\cap\sA,\>\sB)$ in\/~$\sE$. \par
 In both the cases, if the cotorsion pair $(\sA,\sB)$ in\/ $\sK$ is
hereditary, then so is the restricted cotorsion pair in\/~$\sE$.
\end{lem}

\begin{proof}
 This is~\cite[Lemmas~1.5 and~1.6]{Pal}.
\end{proof}

 If $\sK$ is an exact category, then the category of complexes
$\Com(\sK)$ is endowed with the \emph{termwise exact structure} in
which a short sequence of complexes $0\rarrow K^\bu\rarrow L^\bu
\rarrow M^\bu\rarrow0$ is admissible exact in $\Com(\sK)$ if and only
if the short sequence $0\rarrow K^n\rarrow L^n\rarrow M^n\rarrow0$
is admissible exact in $\sK$ for every degree $n\in\boZ$.
 If $\sK$ is an abelian category endowed with the abelian exact
structure, then $\Com(\sK)$ is also an abelian category, and
the termwise exact structure on $\Com(\sK)$ coincides with the abelian
exact structure.
 If $\sK$ is a Grothendieck category, then so is $\Com(\sK)$.

 Given an object $K\in\sK$, we denote by $D_{n,n+1}(K)\in\Com(\sK)$
the contractible two-term complex $\dotsb\rarrow0\rarrow 0\rarrow K
\overset{\id_K}\rarrow K\rarrow0\rarrow0\rarrow\dotsb$ situated
in the cohomological degrees $n$ and~$n+1$.
 The following lemma is well known.

\begin{lem} \label{disk-complexes-Ext-lemma}
 Let\/ $\sK$ be an exact category, $K\in\sK$ be an object, and
$C^\bu$ be a complex in\/~$\sK$.
 Then for every pair of integers $n$ and $i\in\boZ$ there are
natural isomorphisms of abelian groups
\begin{align*}
 \Ext^i_{\Com(\sK)}(D_{n,n+1}(K),C^\bu) &\simeq \Ext^i_\sK(K,C^n), \\
 \Ext^i_{\Com(\sK)}(C^\bu,D_{n,n+1}(K)) &\simeq \Ext^i_\sK(C^{n+1},K).
\end{align*}
\end{lem}

\begin{proof}
 This is~\cite[Lemma~1.8]{Pal}; see~\cite[Lemma~3.1]{Gil} for
an earlier version.
\end{proof}

\begin{lem} \label{Ext1-as-homotopy-Hom-lemma}
 Let\/ $\sK$ be an exact category, and let $A^\bu$ and $B^\bu$ be two
complexes in\/~$\sK$.
 Assume that\/ $\Ext^1_\sK(A^n,B^n)=0$ for all $n\in\boZ$.
 Then there is a natural isomorphism of abelian groups
$$
 \Ext^1_{\Com(\sK)}(A^\bu,B^\bu)\simeq\Hom_{\Hot(\sK)}(A^\bu,B^\bu[1]).
$$
\end{lem}

\begin{proof}
 This is~\cite[Lemma~2.1]{Gil}; see also~\cite[Lemma~1.6]{BHP}.
\end{proof}

\begin{lem} \label{acycl-acycl-orthogonality-lemma}
 Let $(\sA,\sB)$ be a cotorsion pair in an exact category\/~$\sK$.
 Then, for every complex $A^\bu\in\Ac(\sA)$ and every complex $B^\bu\in
\Ac(\sB)$, any morphism of complexes $A^\bu\rarrow B^\bu$ is homotopic
to zero.
 In other words, one has\/ $\Ext^1_{\Com(\sK)}(A^\bu,B^\bu)=0$.
\end{lem}

\begin{proof}
 This observation goes back to~\cite[Lemma~3.9]{Gil}.
 See~\cite[Lemma~B.1.8]{Pcosh} for the details.
\end{proof}

\begin{lem} \label{Hom-criterion-of-acyclicity-in-cotorsion-pair}
 Let $(\sA,\sB)$ be a cotorsion pair in an exact category\/~$\sK$. \par
\textup{(a)} Assume that the kernels of all morphisms exist in\/ $\sK$
and the class\/ $\sA$ is generating in\/~$\sK$.
 Then a complex $B^\bu$ in\/ $\sB$ is acyclic in\/ $\sB$ if and only
if the complex of abelian groups\/ $\Hom_\sK(A,B^\bu)$ is acyclic
for all objects $A\in\sA$. \par
\textup{(b)} Assume that the cokernels of all morphisms exist in\/ $\sK$
and the class\/ $\sB$ is cogenerating in\/~$\sK$.
 Then a complex $A^\bu$ in\/ $\sA$ is acyclic in\/ $\sA$ if and only
if the complex of abelian groups\/ $\Hom_\sK(A^\bu,B)$ is acyclic
for all objects $B\in\sB$.
\end{lem}

\begin{proof}
 Let us prove part~(a).
 The ``only if'' implication is obvious.
 To prove the ``if'', denote by $Z^n\in\sK$ the kernel of
the differential $B^n\rarrow B^{n+1}$ computed in the category~$\sK$.
 Then the complex $B^\bu$ can be obtained by splicing three-term
complexes $Z^n\rarrow B^n\rarrow Z^{n+1}$.
 Consequently, the complex $\Hom_\sK(A,B^\bu)$ can be obtained by
splicing three-term complexes of abelian groups $\Hom_\sK(A,Z^n)
\rarrow\Hom_\sK(A,B^n)\rarrow\Hom_\sK(A,Z^{n+1})$.
 Furthermore, the morphism $Z^n\rarrow B^n$ is a monomorphism in $\sK$,
hence the map $\Hom_\sK(A,Z^n)\rarrow\Hom_\sK(A,B^n)$ is injective.
 As the complex $\Hom_\sK(A,B^\bu)$ is acyclic, it follows that
$0\rarrow\Hom_\sK(A,Z^n)\rarrow\Hom_\sK(A,B^n)\rarrow\Hom_\sK(A,Z^{n+1})
\rarrow0$ is a short exact sequence of abelian groups.

 Besides, the morphism $Z^{n+1}\rarrow B^{n+1}$ is a monomorphism
in~$\sK$.
 Hence the morphism $Z^n\rarrow B^n$ is also the kernel of
the morphism $B^n\rarrow Z^{n+1}$.

 Let $G\in\sA$ be an object such that there is an admissible
epimorphism $G\rarrow Z^{n+1}$ in~$\sK$.
 Then the morphism $G\rarrow Z^{n+1}$ factorizes as $G\rarrow B^n
\rarrow Z^{n+1}$.
 By~\cite[Proposition~7.6]{Bueh} or the dual version
of~\cite[Proposition~2.16]{Bueh}, it follows that $B^n\rarrow Z^{n+1}$
is an admissible epimorphism in~$\sK$.
 Thus $0\rarrow Z^n\rarrow B^n\rarrow Z^{n+1}\rarrow0$ is an admissible
short exact sequence in~$\sK$.
 Finally, surjectivity of the map $\Hom_\sK(A,B^n)\rarrow
\Hom_\sK(A,Z^{n+1})$ implies $\Ext^1_\sK(A,Z^n)=0$.
 So we have $Z^n\in\sB$ for all $n\in\boZ$, and the complex $B^\bu$
is acyclic in~$\sB$.
\end{proof}

 Let $\sK$ be a category and $\alpha$~be an ordinal.
 Any ordinal is an ordered set, and any poset can be viewed as
a category; so one can consider functors $\alpha\rarrow\sK$.
 Such a functor is called a \emph{$\alpha$\+indexed chain} (of
objects and morphisms in~$\sK$) and denoted by
$(K_i\in\sK)_{0\le i<\alpha}$ or $(K_i\to K_j)_{0\le i<j<\alpha}$.
 A $\alpha$\+indexed chain $(K_i\in\sK)_{0\le i<\alpha}$ is said
to be \emph{smooth} if $K_j=\varinjlim_{i<j}K_i$ for all limit
ordinals $0<j<\alpha$.

 Let $\sK$ be an exact category.
 A $(\alpha+1)$\+indexed chain $(F_i\to F_j)_{0\le i<j\le\alpha}$ in
$\sK$ is said to be a \emph{$\alpha$\+indexed filtration} (of
the object $F=F_\alpha$) if the following conditions hold:
\begin{itemize}
\item $F_0=0$;
\item the transition morphism $F_i\rarrow F_{i+1}$ is an admissible
monomorphism for every ordinal $0\le i<\alpha$;
\item the chain $(F_i\to F_j)_{0\le i<j\le\alpha}$ is smooth.
\end{itemize}
 Denote by $S_i=\coker(F_i\to F_{i+1})$ the cokernels of the admissible
monomorphisms $F_i\rarrow F_{i+1}$ in~$\sK$.
 Then the object $F\in\sK$ is said to be \emph{filtered by}
the objects $S_i\in\sK$, \,$0\le i<\alpha$.
 In an alternative terminology, the object $F=F_\alpha$ is said to be
a \emph{transfinitely iterated extension} (\emph{in the sense of
the direct limit}) of the objects $S_i\in\sK$.

 Given a class of objects $\sS\subset\sK$, we denote by
$\Fil(\sS)\subset\sK$ the class of all objects filtered by
(objects isomorphic to) the objects from~$\sS$.
 A class of objects $\sF\subset\sK$ is said to be \emph{deconstructible}
if there exists a \emph{set} of objects $\sS\subset\sK$ such that
$\sF=\Fil(\sS)$.

 The following assertion is known as the \emph{Eklof
lemma}~\cite[Lemma~1]{ET}, \cite[Lemma~6.2]{GT}.

\begin{lem} \label{eklof-lemma}
 For any exact category\/ $\sK$ and a class of objects\/
$\sB\subset\sK$, the class of objects ${}^{\perp_1}\sB$ is closed
under transfinitely iterated extensions in\/~$\sK$; so\/
$\Fil({}^{\perp_1}\sB)={}^{\perp_1}\sB$.
 Equivalently, for any class of objects\/ $\sS\subset\sK$, one has\/
$\sS^{\perp_1}=\Fil(\sS)^{\perp_1}$.
\end{lem}

\begin{proof}
 The argument from~\cite[Lemma~4.5]{PR} is applicable; see
also~\cite[Lemma~7.5]{PS6}.
\end{proof}

 Given a class of objects $\sF\subset\sK$, we denote by $\sF^\oplus
\subset\sK$ the class of all direct summands of objects from~$\sF$.
 The following two theorems are two versions/generalizations of
the \emph{Eklof--Trlifaj theorem}~\cite[Theorems~2 and~10]{ET},
\cite[Theorem~6.11 and Corollary~6.13]{GT}.

\begin{thm} \label{loc-pres-eklof-trlifaj}
 Let\/ $\sK$ be a locally presentable abelian category (in the sense
of\/~\cite[Definition~1.17 or Theorem~1.20]{AR}), endowed with
the abelian exact structure.
 Let\/ $\sS\subset\sK$ be a \emph{set} of objects and $(\sA,\sB)$ be
the cotorsion pair generated by\/ $\sS$ in\/~$\sK$.
 In this context: \par
\textup{(a)} if the class\/ $\sA$ is generating and the class\/ $\sB$
is cogenerating in\/ $\sK$, then the cotorsion pair $(\sA,\sB)$ is
complete; \par
\textup{(b)} if the class\/ $\Fil(\sS)$ is generating in\/ $\sK$,
then\/ $\sA=\Fil(\sS)^\oplus$.
\end{thm}

\begin{proof}
 This is~\cite[Corollary~3.6 and Theorem~4.8(d)]{PR}
or~\cite[Theorems~3.3 and~3.4]{PS4}.
\end{proof}

\begin{thm} \label{efficient-eklof-trlifaj}
 Let\/ $\sK$ be an efficient exact category (in the sense
of\/~\cite[Definition~3.4]{Sto-ICRA}).
 Let\/ $\sS\subset\sK$ be a \emph{set} of objects and $(\sA,\sB)$ be
the cotorsion pair generated by\/ $\sS$ in\/~$\sK$.
 In this context: \par
\textup{(a)} if the class\/ $\sA$ is generating in\/ $\sK$, then
the cotorsion pair $(\sA,\sB)$ is complete; \par
\textup{(b)} if the class\/ $\Fil(\sS)$ is generating in\/ $\sK$,
then\/ $\sA=\Fil(\sS)^\oplus$.
\end{thm}

\begin{proof}
 This is~\cite[Theorem~5.16]{Sto-ICRA}.
\end{proof}

 Notice that Grothendieck categories~$\sK$ (endowed with the abelian
exact structure) satisfy the assumptions of both
Theorems~\ref{loc-pres-eklof-trlifaj} and~\ref{efficient-eklof-trlifaj}.
 A common special case of the two theorems applicable to Grothendieck
categories $\sK$ can be found in~\cite[Proposition~2.4]{Gil3}.
 A Grothendieck category $\sK$ has enough injective objects, so
for any cotorsion pair $(\sA,\sB)$ in $\sK$ the class $\sB$ is
automatically cogenerating.
 The condition that the class $\sA$ is generating is \emph{not}
automatic and needs to be checked, however.

 We refer to~\cite[Section~11]{Bueh} for a discussion of projective
and injective objects in an exact category~$\sK$.
 The full subcategory of projective objects in $\sK$ will be denoted
by $\sK_\proj\subset\sK$, and the full subcategory of injective
objects by $\sK^\inj\subset\sK$.

\begin{lem} \label{inj-proj-objects-in-left-right-classes}
 Let $(\sA,\sB)$ be a complete cotorsion pair in an exact
category~$\sK$.
 Then \par
\textup{(a)} there are enough injective objects in the exact category\/
$\sA$, and the kernel\/ $\sA\cap\sB$ of the cotorsion pair $(\sA,\sB)$
is the class of all injective objects in\/~$\sA$; \par
\textup{(b)} there are enough projective objects in the exact category\/
$\sB$, and the kernel\/ $\sA\cap\sB$ of the cotorsion pair $(\sA,\sB)$
is the class of all projective objects in\/~$\sB$.
\end{lem}

\begin{proof}
 Let us prove part~(a).
 First of all, the class $\sA={}^{\perp_1}\sB$ is closed under
extensions in $\sK$, hence there is the inherited exact structure
on $\sA$ and the assertion of the lemma makes sense.
 Furthermore, we have $\Ext^1_\sA(A,B)=\Ext^1_\sK(A,B)=0$ for all
$A\in A$ and $B\in\sA\cap\sB$, so the objects of $\sA\cap\sB$ are
injective in~$\sA$.
 To show that there are enough such injective objects, suppose given
an object $A\in\sA$.
 Consider a special preenvelope
sequence~\eqref{special-preenvelope-sequence} in the exact
category~$\sK$,
$$
 0\lrarrow A\lrarrow B\lrarrow A'\lrarrow0
$$
with $B\in\sB$ and $A'\in\sA$.
 Then we have $B\in\sA\cap\sB$, since $\sA$ is closed under extensions
in~$\sK$.
 So every object of $\sA$ is an admissible subobject of an object from
$\sA\cap\sB$ in the exact category~$\sA$.
 It follows that all the injective objects of $\sA$ are direct summands
of objects from $\sA\cap\sB$; and it remains to point out that
the classes $\sA$ and $\sB$ are closed under direct summands in $\sK$,
hence the class $\sA\cap\sB$ is closed under direct summands in $\sA$
and~$\sB$.
\end{proof}

 Let $\sK$ be an exact category.
 A complex $A^\bu$ in $\sK$ is said to be \emph{Becker-coacyclic}
if, for every complex of injective objects $J^\bu$ in $\sK$, any
morphism of complexes $A^\bu\rarrow J^\bu$ is homotopic to zero.
 The full subcategory of Becker-coacyclic complexes is denoted by
$\Ac^\bco(\sK)\subset\Com(\sK)$ or $\Ac^\bco(\sK)\subset\Hot(\sK)$.
 It is clear from the definition that $\Ac^\bco(\sK)$ is a thick
subcategory in the triangulated category $\Hot(\sK)$.
 The triangulated Verdier quotient category
$$
 \sD^\bco(\sK)=\Hot(\sK)/\Ac^\bco(\sK)
$$
is called the \emph{coderived category} (\emph{in the sense of Becker})
of the exact category~$\sK$.

 Dually, a complex $B^\bu$ in $\sK$ is said to be
\emph{Becker-contraacyclic} if, for every complex of projective objects
$P^\bu$ in $\sK$, any morphism of complexes $P^\bu\rarrow B^\bu$ is
homotopic to zero.
 The full subcategory of Becker-contraacyclic complexes is denoted by
$\Ac^\bctr(\sK)\subset\Com(\sK)$ or $\Ac^\bctr(\sK)\subset\Hot(\sK)$.
 It is clear from the definition that $\Ac^\bctr(\sK)$ is a thick
subcategory in the triangulated category $\Hot(\sK)$.
 The triangulated Verdier quotient category
$$
 \sD^\bctr(\sK)=\Hot(\sK)/\Ac^\bctr(\sK)
$$
is called the \emph{contraderived category} (\emph{in the sense of
Becker}) of the exact category~$\sK$.

 Any short exact sequence $0\rarrow K^\bu\rarrow L^\bu\rarrow M^\bu
\rarrow0$ of complexes in $\sK$ can be viewed as a bicomplex with
three rows.
 As such, it has a total complex, which we will denote by
$\Tot(K^\bu\to L^\bu\to M^\bu)$.

\begin{lem} \label{totalizations-are-co-contra-acyclic}
 Let\/ $\sK$ be an exact category.
 Then \par
\textup{(a)} the total complex\/ $\Tot(K^\bu\to L^\bu\to M^\bu)$ of
any (termwise admissible) short exact sequence $0\rarrow K^\bu\rarrow
L^\bu\rarrow M^\bu\rarrow0$ is both a Becker-coacyclic and
a Becker-contraacyclic complex in\/~$\sK$; \par
\textup{(b)} the full subcategory of Becker-coacyclic complexes\/
$\Ac^\bco(\sK)$ is closed under infinite coproducts in\/ $\Hot(\sK)$
and\/ $\Com(\sK)$; \par
\textup{(c)} the full subcategory of Becker-contraacyclic complexes\/
$\Ac^\bctr(\sK)$ is closed under infinite products in\/ $\Hot(\sK)$
and\/ $\Com(\sK)$.
\end{lem}

\begin{proof}
 Parts~(b) and~(c) are obvious.
 It may be helpful to notice that any (co)product in $\Com(\sK)$ is
a termwise (co)product, hence it is also a (co)product in $\Hot(\sK)$.
 Part~(a) goes back to~\cite[Theorem~3.5 and Remark~3.5]{Pkoszul};
the details can be found in~\cite[Lemmas~7.1 and~9.1]{PS4}
or~\cite[Lemmas~A.1.4 and B.7.1(a)]{Pcosh}.

 For the sake of completeness of the exposition, let us sketch
a proof of part~(a).
 Put $A^\bu=\Tot(K^\bu\to L^\bu\to M^\bu)$.
 Let $P^\bu$ be a complex of projective objects in~$\sK$.
 Then the short sequence of complexes of abelian groups
$0\rarrow\Hom_\sK(P^\bu,K^\bu)\rarrow\Hom_\sK(P^\bu,L^\bu)\rarrow
\Hom_\sK(P^\bu,M^\bu)\rarrow0$ is exact, and the complex of abelian
groups $\Hom_\sK(P^\bu,A^\bu)$ is the totalization of this short exact
sequence of complexes of abelian groups.
 It remains to observe that the totalization of any short exact
sequence of complexes of abelian groups is an acyclic complex.
 Thus we have $H^0\Hom_\sK(P^\bu,A^\bu)=0$, so every morphism of
complexes $P^\bu\rarrow A^\bu$ is homotopic to zero.
 The dual argument shows that, for any complex of injective objects
$J^\bu$ in $\sK$, every morphism of complexes $A^\bu\rarrow J^\bu$
is homotopic to zero.
\end{proof}

 A further discussion of the Becker coderived and contraderived
categories of exact categories can be found
in~\cite[Section~B.7]{Pcosh}.
 The exposition in~\cite[Appendix~B]{Pcosh} also offers some further
details on cotorsion pairs in exact categories.

\Section{Four Cotorsion Pairs in the Category of Complexes}

 Let $(\sA,\sB)$ be a hereditary complete cotorsion pair generated
by a set of objects $\sS$ in a Grothendieck category~$\sK$.
 The aim of this section is to construct four related hereditary
complete cotorsion pairs in the abelian category of complexes
$\Com(\sK)$.
 These constructions of cotorsion pairs in the category of complexes
are essentially due to Gillespie~\cite[Proposition~3.6]{Gil},
\cite[Proposition~3.2]{Gil2}, \cite[Lemma~4.9]{Gil3}.
 Our exposition provides some additional information.

\begin{prop} \label{coacyclic-all-cotorsion-pair-prop}
 Let\/ $\sK$ be either a locally presentable abelian category with
the abelian exact structure, or an efficient exact category in
the sense of\/~\cite{Sto-ICRA}.
 Let $(\sA,\sB)$ be a hereditary complete cotorsion pair generated
by a set of objects\/ $\sS$ in\/~$\sK$.
 Then the two classes\/ $\sA'=\Ac^\bco(\sA)\subset\Com(\sA)\subset
\Com(\sK)$ and\/ $\sB'=\Com(\sB)\subset\Com(\sK)$ form a hereditary
complete cotorsion pair $(\sA',\sB')$ generated by a set of objects in
the abelian/exact category\/ $\Com(\sK)$.
 The kernel\/ $\sA'\cap\sB'$ of the cotorsion pair $(\sA',\sB')$ in\/
$\Com(\sK)$ is the class of all contractible complexes with the terms
in\/ $\sA\cap\sB$.
\end{prop}

\begin{proof}
 This is one of the two cotorsion pairs
from~\cite[Proposition~3.2]{Gil2}.

 Notice first of all that if $\sK$ is a locally presentable abelian
category, then so is $\Com(\sK)$ \,\cite[Lemma~6.3]{PS4}.
 It is also easy to check that the exact category $\Com(\sK)$ is
efficient in the sense of~\cite[Definition~3.4]{Sto-ICRA} whenever
an exact category~$\sK$ is.

 Denote by $\sS'$ the set of all contractible two-term complexes
$D_{n,n+1}(S)\in\Com(\sK)$ with $S\in\sS$ and $n\in\boZ$.
 Let $(\sA',\sB')$ be the cotorsion pair generated by $\sS'$ in
$\Com(\sK)$.
 Then it is clear from Lemma~\ref{disk-complexes-Ext-lemma} (for $i=1$)
that $\sB'=\Com(\sB)$.
 In particular, all contractible two-term complexes $D_{n,n+1}(B)$
with $B\in\sB$ belong to $\sB'$, and it follows from the same lemma
that $\sA'\subset\Com(\sA)$.
 Furthermore, it is clear from Lemma~\ref{Ext1-as-homotopy-Hom-lemma}
that all contractible complexes with the terms in $\sA$ belong
to~$\sA'$.

 By assumptions, every object of $\sK$ is an admissible quotient of
an object from $\sA$, and it follows that every object of $\Com(\sK)$
is an admissible quotient of a contractible complex with the terms
in~$\sA$.
 Dually, every object of $\sK$ is an admissible subobject of an object
from $\sB$, and it follows that every object of $\Com(\sK)$ is
an admissible subobject of a contractible complex with the terms
in~$\sB$.
 Thus the class $\sA'$ is generating in $\Com(\sK)$, and the class
$\sB'$ is cogenerating in $\Com(\sK)$.
 By Theorem~\ref{loc-pres-eklof-trlifaj}(a)
or~\ref{efficient-eklof-trlifaj}(a), it follows that the cotorsion pair
$(\sA',\sB')$ is complete in $\Com(\sK)$.
 (Moreover, one can always adjoin to $\sS$ a generator of $\sK$
belonging to~$\sA$; then $\Fil(\sS')$ is a generating class in
$\Com(\sK)$, and part~(b) of the respective theorem tells us that
$\sA'=\Fil(\sS')^\oplus$.)

 The class $\sB'=\Com(\sB)$ is closed under cokernels of admissible
monomorphisms in $\Com(\sK)$, since the class $\sB$ is closed under
cokernels of admissible monomorphisms in~$\sK$.
 Thus the cotorsion pair $(\sA',\sB')$ in $\Com(\sK)$ is hereditary.

 Now we need to show that $\sA'=\Ac^\bco(\sA)$.
 By Lemma~\ref{inj-proj-objects-in-left-right-classes}(a), we have
$\sA^\inj=\sA\cap\sB$.
 The full subcategory $\sE=\Com(\sA)$ is closed under extensions and
kernels of admissible epimorphisms in $\Com(\sK)$, since so is the full
subcategory $\sA$ in~$\sK$.
 By Lemma~\ref{restricting-cotorsion-pair-lemma}(a), the hereditary
complete cotorsion pair $(\sA',\sB')$ in $\Com(\sK)$ restricts to
a hereditary complete cotorsion pair $(\sA',\>\sE\cap\sB')=
(\sA',\>\Com(\sA\cap\sB))$ in $\Com(\sA)$.
 Hence $\sA'={}^{\perp_1}\Com(\sA\cap\sB)$ in $\Com(\sA)$.
 By Lemma~\ref{Ext1-as-homotopy-Hom-lemma}, this means that
$\sA'=\Ac^\bco(\sA)$.

 It remains to prove that $\sA'\cap\sB'$ is the class of all
contractible complexes in $\sA\cap\sB$.
 Indeed, we have $\sA'\subset\Com(\sA)$ and $\sB'=\Com(\sB)$, hence
$\sA'\cap\sB'\subset\Com(\sA\cap\sB)$.
 All contractible complexes in $\sA$ belong to~$\sA'$; hence all
contractible complexes in $\sA\cap\sB$ belong to $\sA'\cap\sB'$.
 Conversely, any complex belonging to $\sA'\cap\sB'$ is
a Becker-coacyclic complex of injective objects in $\sA$, hence
a contractible complex.
\end{proof}

\begin{prop} \label{all-contraacyclic-cotorsion-pair-prop}
 Let\/ $\sK$ be a Grothendieck category and $(\sA,\sB)$ be
a hereditary complete cotorsion pair generated by a set of objects\/
$\sS$ in\/~$\sK$.
 Then the two classes\/ $\sA''=\Com(\sA)\subset\Com(\sK)$ and $\sB''=
\Ac^\bctr(\sB)\subset\Com(\sB)\subset\Com(\sK)$ form a hereditary
complete cotorsion pair $(\sA'',\sB'')$ generated by a set of objects
in the Grothendieck category\/ $\Com(\sK)$.
 The kernel\/ $\sA''\cap\sB''$ of the cotorsion pair $(\sA'',\sB'')$
in\/ $\Com(\sK)$ is the class of all contractible complexes with
the terms in\/ $\sA\cap\sB$.
\end{prop}

\begin{proof}
 This is the other one of the two cotorsion pairs
from~\cite[Proposition~3.2]{Gil2}.
 The assertion that this cotorsion pair is complete was proved
in~\cite[Lemma~4.9]{Gil3}.

 Adjoining to $\sS$ a generator of $\sK$ belonging to $\sA$, we can
assume without loss of generality that the class $\Fil(\sS)$ is
generating in~$\sK$.
 Then, by Theorem~\ref{loc-pres-eklof-trlifaj}(b)
or~\ref{efficient-eklof-trlifaj}(b), we have $\sA=\Fil(\sS)^\oplus$.
 By~\cite[Proposition~2.9(1)]{Sto-hill}, it follows that the class
$\sA$ is deconstructible in~$\sK$.
 Now~\cite[Proposition~4.3]{Sto-hill} tells us that the class
$\Com(\sA)$ is deconstructible in $\Com(\sK)$.

 Let $\sS''\subset\Com(\sA)$ be a set of complexes such that
$\Com(\sA)=\Fil(\sS'')$.
 Let $(\sA'',\sB'')$ be the cotorsion pair generated by $\sS''$ in
$\Com(\sK)$.
 The class $\sA$ is generating in $\sK$, hence the class of all
contractible complexes in $\sA$ is generating in $\Com(\sK)$.
 Thus the class $\Fil(\sS'')$ is generating in $\Com(\sK)$.
 By Lemma~\ref{eklof-lemma}, we have $\Fil(\sS'')\subset\sA''$.
 Applying Theorem~\ref{loc-pres-eklof-trlifaj}(a\+-b)
or~\ref{efficient-eklof-trlifaj}(a\+-b), we conclude that
the cotorsion pair $(\sA'',\sB'')$ is complete in $\Com(\sK)$
and $\sA''=\Fil(\sS'')=\Com(\sA)$.

 The class $\sA''=\Com(\sA)$ is closed under kernels of admissible
epimorphisms in $\Com(\sK)$, since the class $\sA$ is closed under
kernels of admissible epimorphisms in~$\sK$.
 Thus the cotorsion pair $(\sA'',\sB'')$ in $\Com(\sK)$ is hereditary.

 Now we need to show that $\sB''=\Ac^\bctr(\sB)$.
 By Lemma~\ref{disk-complexes-Ext-lemma} (for $i=1$), we know that
$\sB''\subset\Com(\sB)$.
 By Lemma~\ref{inj-proj-objects-in-left-right-classes}(b), we have
$\sB_\proj=\sA\cap\sB$.
 The full subcategory $\sE=\Com(\sB)$ is closed under extensions and
cokernels of admissible monomorphisms in $\Com(\sK)$, since so is
the full subcategory $\sB$ in~$\sK$.
 By Lemma~\ref{restricting-cotorsion-pair-lemma}(b), the hereditary
complete cotorsion pair $(\sA'',\sB'')$ in $\Com(\sK)$ restricts to
a hereditary complete cotorsion pair $(\sE\cap\sA'',\>\sB'')=
(\Com(\sA\cap\sB),\>\sB'')$ in $\Com(\sB)$.
 Hence $\sB''=\Com(\sA\cap\sB)^{\perp_1}$ in $\Com(\sB)$.
 By Lemma~\ref{Ext1-as-homotopy-Hom-lemma}, this means that
$\sB''=\Ac^\bctr(\sB)$.

 It remains to prove that $\sA''\cap\sB''$ is the class of all
contractible complexes in $\sA\cap\sB$.
 Indeed, we have $\sA''=\Com(\sA)$ and $\sB''\subset\Com(\sB)$, hence
$\sA''\cap\sB''\subset\Com(\sA\cap\sB)$.
 All contractible complexes in $\sB$ belong to~$\sB''$; hence all
contractible complexes in $\sA\cap\sB$ belong to $\sA''\cap\sB''$.
 Conversely, any complex belonging to $\sA''\cap\sB''$ is
a Becker-contraacyclic complex of projective objects in $\sB$, hence
a contractible complex.
\end{proof}

 Let $(\sA,\sB)$ be a cotorsion pair in an exact category~$\sK$.
 In the spirit of the notation of Gillespie in~\cite{Gil,Gil3},
we put $\DG(\sB)=\Ac(\sA)^{\perp_1}\subset\Com(\sK)$ and
$\DG(\sA)={}^{\perp_1}\Ac(\sB)\subset\Com(\sK)$.
 In particular, we have $D_{n,n+1}(A)\in\Ac(\sA)$ and $D_{n,n+1}(B)\in
\Ac(\sB)$ for all $A\in\sA$, \,$B\in\sB$, and $n\in\boZ$; hence
Lemma~\ref{disk-complexes-Ext-lemma} (for $i=1$) implies the inclusions
$\DG(\sB)\subset\Com(\sB)$ and $\DG(\sA)\subset\Com(\sA)$.

 By Lemma~\ref{Ext1-as-homotopy-Hom-lemma}, the class $\DG(\sB)$ can
be described as the class of all complexes $D^\bu$ in $\sB$ such that,
for every complex $A^\bu\in\Ac(\sA)$, all morphisms of complexes
$A^\bu\rarrow D^\bu$ are homotopic to zero.
 Dually, the class $\DG(\sA)$ can be described as the class of all
complexes $E^\bu$ in $\sA$ such that, for every complex
$B^\bu\in\Ac(\sB)$, all morphisms of complexes $E^\bu\rarrow B^\bu$
are homotopic to zero.

\begin{prop} \label{acyclic-dg-cotorsion-pair-prop}
 Let\/ $\sK$ be a Grothendieck category and $(\sA,\sB)$ be
a hereditary complete cotorsion pair generated by a set of objects\/
$\sS$ in\/~$\sK$.
 Then the two classes\/ $\sA'''=\Ac(\sA)\subset\Com(\sA)\subset
\Com(\sK)$ and\/ $\sB'''=\DG(\sB)\subset\Com(\sB)\subset\Com(\sK)$
form a hereditary complete cotorsion pair $(\sA''',\sB''')$ generated
by a set of objects in the Grothendieck category\/ $\Com(\sK)$.
 The kernel\/ $\sA'''\cap\sB'''$ of the cotorsion pair $(\sA''',\sB''')$
in\/ $\Com(\sK)$ is the class of all contractible complexes with
the terms in\/ $\sA\cap\sB$.
\end{prop}

\begin{proof}
 This is one of the two cotorsion pairs
from~\cite[Proposition~3.6]{Gil}.
 The assertion that this cotorsion pair is complete was proved
in~\cite[Lemma~4.9]{Gil3}.

 The class $\sA$ is deconstructible in $\sK$, as we have seen in
the proof of Proposition~\ref{all-contraacyclic-cotorsion-pair-prop}.
 By~\cite[Proposition~4.4]{Sto-hill}, it follows that the class
$\Ac(\sA)$ is deconstructible in $\Com(\sK)$.
 Let $\sS'''\subset\Ac(\sA)$ be a set of complexes such that
$\Ac(\sA)=\Fil(\sS''')$.
 Let $(\sA''',\sB''')$ be the cotorsion pair generated by $\sS'''$ in
$\Com(\sK)$.
 The class $\sA$ is generating in $\sK$, hence the class of all
contractible complexes in $\sA$ is generating in $\Com(\sK)$.
 Thus the class $\Fil(\sS''')$ is generating in $\Com(\sK)$.
 By Lemma~\ref{eklof-lemma}, we have $\Fil(\sS''')\subset\sA'''$.
 Applying Theorem~\ref{loc-pres-eklof-trlifaj}(a\+-b)
or~\ref{efficient-eklof-trlifaj}(a\+-b), we conclude that
the cotorsion pair $(\sA''',\sB''')$ is complete in $\Com(\sK)$
and $\sA'''=\Fil(\sS''')=\Ac(\sA)$.
 It follows that $\sB'''=(\sA''')^{\perp_1}=\DG(\sB)$ by
the definition.

 The class $\sA'''=\Ac(\sA)$ is closed under kernels of admissible
epimorphisms in $\Com(\sK)$, since the class $\sA$ is closed under
kernels of admissible epimorphisms in~$\sK$.
 (See~\cite[Corollary~3.6 and Proposition~7.6]{Bueh} for the assertion
that the kernel of a termwise admissible epimorphism of acyclic
complexes is an acyclic complex in a weakly idempotent-complete
exact category~$\sA$.)
 Thus the cotorsion pair $(\sA''',\sB''')$ in $\Com(\sK)$ is hereditary.

 It remains to prove that $\sA'''\cap\sB'''$ is the class of all
contractible complexes in $\sA\cap\sB$.
 Indeed, we have $\sA'''\subset\Com(\sA)$ and $\sB'''\subset\Com(\sB)$,
hence $\sA'''\cap\sB'''\subset\Com(\sA\cap\sB)$.
 All contractible complexes in $\sB$ belong to~$\sB'''=\DG(\sB)$; hence
all contractible complexes in $\sA\cap\sB$ belong to $\sA'''\cap\sB'''$.
 Conversely, all complexes belonging to $\sA'''\cap\sB'''=
\Ac(\sA)\cap\DG(\sB)$ are contractible.
\end{proof}

\begin{prop} \label{dg-acyclic-cotorsion-pair-prop}
 Let\/ $\sK$ be either a locally presentable abelian category with
the abelian exact structure, or an efficient exact category in
the sense of\/~\cite{Sto-ICRA}.
 In the latter case, assume additionally that the kernels of all
morphisms exist in\/~$\sK$.
 Let $(\sA,\sB)$ be a hereditary complete cotorsion pair generated
by a set of objects\/ $\sS$ in\/~$\sK$.
 Then the two classes\/ $\sA''''=\DG(\sA)\subset\Com(\sA)\subset
\Com(\sK)$ and\/ $\sB''''=\Ac(\sB)\subset\Com(\sB)\subset\Com(\sK)$
form a hereditary complete cotorsion pair $(\sA'''',\sB'''')$ generated
by a set of objects in the abelian/exact category\/ $\Com(\sK)$.
 The kernel\/ $\sA''''\cap\sB''''$ of the cotorsion pair
$(\sA'''',\sB'''')$ in\/ $\Com(\sK)$ is the class of all contractible
complexes with the terms in\/ $\sA\cap\sB$.
\end{prop}

\begin{proof}
 This is the other one of the two cotorsion pairs
from~\cite[Proposition~3.6]{Gil}.
 In the case of a Grothendieck category $\sK$, the assertion that this
cotorsion pair is complete was proved in~\cite[Lemma~4.9]{Gil3}.
 Some additional information can be found
in~\cite[Proposition~B.2.6]{Pcosh}.

 As in the proof of Proposition~\ref{coacyclic-all-cotorsion-pair-prop},
we notice that if $\sK$ is a locally presentable abelian category, then
so is $\Com(\sK)$; and if $\sK$ is an efficient exact category in
the sense of~\cite[Definition~3.4]{Sto-ICRA}, then so is $\Com(\sK)$.

 Denote by $\sS''''$ the set of all one-term complexes
$\dotsb\rarrow0\rarrow S\rarrow 0\rarrow\dotsb$, with the object
$S\in\sS$ situated in an arbitrary cohomological degree $n\in\boZ$.
 Let $(\sA'''',\sB'''')$ be the cotorsion pair generated by $\sS''''$
in $\Com(\sK)$.
 The contractible two-term complex $D_{n,n+1}(S)$ is an extension of
two appropriately shifted copies of the one-term complex
$\dotsb\rarrow0\rarrow S\rarrow 0\rarrow\dotsb$, so we have
$D_{n,n+1}(S)\in\sA''''$.
 By Lemma~\ref{disk-complexes-Ext-lemma} (for $i=1$), it follows that
$\sB''''\subset\Com(\sB)$.

 One can see from Lemma~\ref{Ext1-as-homotopy-Hom-lemma} that all
contractible complexes with the terms in $\sA$ belong to $\sA''''$ and
all contractible complexes with the terms in $\sB$ belong to~$\sB''''$.
 Similarly to the proof of
Proposition~\ref{coacyclic-all-cotorsion-pair-prop}, every object of
$\Com(\sK)$ is an admissible quotient of a contractible complex with
the terms in $\sA$ and an admissible subobject of a contractible
complex with the terms in~$\sB$.
 Thus the class $\sA''''$ is generating in $\Com(\sK)$, and the class
$\sB''''$ is cogenerating in $\Com(\sK)$.
 By Theorem~\ref{loc-pres-eklof-trlifaj}(a)
or~\ref{efficient-eklof-trlifaj}(a), it follows that the cotorsion pair
$(\sA'''',\sB'''')$ is complete in $\Com(\sK)$.

 Moreover, without loss of generality we can assume that $\sS$ contains
a generator of $\sK$ belonging to~$\sA$.
 Then $\Fil(\sS'''')$ is a generating class in $\Com(\sK)$, and
part~(b) of the respective theorem tells us that
$\sA''''=\Fil(\sS'''')^\oplus$.
 Applying the same Theorem~\ref{loc-pres-eklof-trlifaj}(b)
or~\ref{efficient-eklof-trlifaj}(b) to the cotorsion pair $(\sA,\sB)$
in $\sK$, we see that $\sA=\Fil(\sS)^\oplus$.
 Thus all one-term complexes $\dotsb\rarrow0\rarrow A\rarrow 0\rarrow
\dotsb$ with $A\in\sA$ belong to $\sA''''$.

 Now we need to show that $\sB''''=\Ac(\sB)$.
 Indeed, all morphisms from a complex belonging to $\sS''''$ to
an acyclic complex in $\sB$ are homotopic to zero, so
Lemma~\ref{Ext1-as-homotopy-Hom-lemma} tells us that
$\Ac(\sB)\subset\sB''''$.
 Conversely, let $B^\bu$ be a complex belonging to~$\sB''''$.
 We know from the previous paragraph and
Lemma~\ref{Ext1-as-homotopy-Hom-lemma} that the complex of abelian
groups $\Hom_\sK(A,B^\bu)$ is acyclic for every object $A\in\sA$.
 By Lemma~\ref{Hom-criterion-of-acyclicity-in-cotorsion-pair}(a),
it follows that the complex $B^\bu$ is acyclic in~$\sB$.
 (This is where we use the assumption of existence of kernels in~$\sK$.)
 Thus $\sB''''=\Ac(\sB)$; so $\sA''''={}^{\perp_1}(\sB'''')=\DG(\sA)$
by the definition.

 The class $\sB''''=\Ac(\sB)$ is closed under cokernels of admissible
monomorphisms in $\Com(\sK)$, since the class $\sB$ is closed under
cokernels of admissible monomorphisms in~$\sK$.
 (See~\cite[Corollary~3.6 and the dual version of Proposition~7.6]{Bueh}
for the assertion that the cokernel of a termwise admissible
monomorphism of acyclic complexes is an acyclic complex in a weakly
idempotent-complete exact category~$\sB$.)
 Thus the cotorsion pair $(\sA'''',\sB'''')$ in $\Com(\sK)$ is
hereditary.

 It remains to prove that $\sA''''\cap\sB''''$ is the class of all
contractible complexes in $\sA\cap\sB$.
 Indeed, we have $\sA''''\subset\Com(\sA)$ and
$\sB''''\subset\Com(\sB)$, hence
$\sA''''\cap\sB''''\subset\Com(\sA\cap\sB)$.
 All contractible complexes in $\sA$ belong to~$\sA''''=\DG(\sA)$;
hence all contractible complexes in $\sA\cap\sB$ belong to
$\sA''''\cap\sB''''$.
 Conversely, all complexes belonging to $\sA''''\cap\sB''''=
\DG(\sA)\cap\Ac(\sB)$ are contractible.
\end{proof}

\Section{Cotorsion Periodicity-Type Property}

 In this section we discuss the case when the cotorsion pairs from
Propositions~\ref{coacyclic-all-cotorsion-pair-prop}
and~\ref{acyclic-dg-cotorsion-pair-prop} coincide, and the connection
with cotorsion periodicity.

\begin{lem} \label{DG-A-intersect-Ac-K-is-Ac-A}
 Let\/ $\sK$ be an idempotent-complete exact category with enough
injective objects and $(\sA,\sB)$ be a hereditary cotorsion pair
in\/~$\sK$.
 Then every acyclic complex in\/ $\sA$ belongs to\/ $\DG(\sA)$, and
conversely, every complex from\/ $\DG(\sA)$ that is acyclic as
a complex in\/ $\sK$ is acyclic in\/~$\sA$,
$$
 \Ac(\sK)\cap\DG(\sA)=\Ac(\sA).
$$
\end{lem}

\begin{proof}
 The inclusion $\Ac(\sA)\subset\DG(\sA)$ holds by the very general
Lemma~\ref{acycl-acycl-orthogonality-lemma}.
 The inclusion $\Ac(\sA)\subset\Ac(\sK)$ is obvious.
 The inclusion $\Ac(\sK)\cap\DG(\sA)\subset\Ac(\sA)$ is the nontrivial
assertion.
 The following argument goes back to~\cite[Theorem~3.12]{Gil}.

 Let $A^\bu$ be a complex from $\DG(\sA)$ that is acyclic in~$\sK$.
 In order to prove that the complex $A^\bu$ is acyclic in $\sA$, we
need to check that the complex of abelian groups $\Hom_\sK(A^\bu,B)$
is acyclic for every $B\in\sB$.
 Let $0\rarrow B\rarrow J^0\rarrow J^1\rarrow J^2\rarrow\dotsb$ be
an injective coresolution of the object $B$ in~$\sK$.
 Denote by $J^\bu$ the complex $0\rarrow J^0\rarrow J^1\rarrow J^2
\rarrow\dotsb$ and by $C^\bu$ the complex $0\rarrow B\rarrow J^0
\rarrow J^1\rarrow J^2\rarrow\dotsb$.
 Then the complex of abelian groups $\Hom_\sK(A^\bu,J^\bu)$ is acyclic,
since $A^\bu$ is an acyclic complex in $\sK$ and $J^\bu$ is a bounded
below complex of injective objects in~$\sK$.
 On the other hand, the complex $A^\bu$ belongs to $\DG(\sA)$ and
the complex $C^\bu$ belongs to $\Ac(\sB)$ (since $\sK^\inj\subset\sB$
and the class $\sB$ is closed under cokernels of monomorphisms
in~$\sK$).
 In view of Lemma~\ref{Ext1-as-homotopy-Hom-lemma}, the complex
of abelian groups $\Hom_\sK(A^\bu,C^\bu)$ is acyclic.
 Thus the complex $\Hom_\sK(A^\bu,B)$ is acyclic as well.
\end{proof}

\begin{lem} \label{DG-B-intersect-Ac-K-is-Ac-B}
 Let\/ $\sK$ be a Grothendieck category and $(\sA,\sB)$ be a hereditary
complete cotorsion pair generated by a set of objects in\/~$\sK$.
 Then every acyclic complex in\/ $\sB$ belongs to\/ $\DG(\sB)$, and
conversely, every complex from\/ $\DG(\sB)$ that is acyclic as
a complex in\/ $\sK$ is acyclic in\/~$\sB$,
$$
 \Ac(\sK)\cap\DG(\sB)=\Ac(\sB).
$$
\end{lem}

\begin{proof}
 The inclusion $\Ac(\sB)\subset\DG(\sB)$ holds by the very general
Lemma~\ref{acycl-acycl-orthogonality-lemma}.
 The inclusion $\Ac(\sB)\subset\Ac(\sK)$ is obvious.
 The inclusion $\Ac(\sK)\cap\DG(\sB)\subset\Ac(\sB)$ is the nontrivial
assertion.
 The argument from~\cite[Theorem~3.12]{Gil} requires enough projective
objects in~$\sK$.
 The following indirect proof based on
Lemma~\ref{DG-A-intersect-Ac-K-is-Ac-A} and
Proposition~\ref{dg-acyclic-cotorsion-pair-prop} allows us to avoid
this assumption.

 Let $D^\bu$ be a complex from $\DG(\sB)$ that is acyclic in~$\sK$.
 Consider a special preenvelope exact sequence $0\rarrow D^\bu\rarrow
B^\bu\rarrow A^\bu\rarrow0$ in $\Com(\sK)$ for the cotorsion pair
$(\DG(\sA),\Ac(\sB))$ from
Proposition~\ref{dg-acyclic-cotorsion-pair-prop}.
 So we have $B^\bu\in\Ac(\sB)$ and $A^\bu\in\DG(\sA)$.

 According to the first paragraph of this proof, the complex $B^\bu$
belongs to $\DG(\sB)$.
 Since the class $\DG(\sB)$ is closed under cokernels of monomorphisms
in $\Com(\sK)$ (e.~g., by
Proposition~\ref{acyclic-dg-cotorsion-pair-prop}), we have
$A^\bu\in\DG(\sA)\cap\DG(\sB)$.
 Furthermore, both the complexes $D^\bu$ and $B^\bu$ are acyclic
in~$\sK$, hence so is the complex~$A^\bu$.
 Hence we actually have $A^\bu\in\Ac(\sK)\cap\DG(\sA)\cap\DG(\sB)$.
 By Lemma~\ref{DG-A-intersect-Ac-K-is-Ac-A}, it follows that
$A^\bu\in\Ac(\sA)\cap\DG(\sB)$.
 Thus $A^\bu$ is a contractible complex with the terms in $\sA\cap\sB$.
 Finally, $D^\bu$ is the kernel of a termwise admissible epimorphism
$B^\bu\rarrow A^\bu$ of acyclic complexes in the exact category~$\sB$;
so $D^\bu$ is also an acyclic complex in~$\sB$.
\end{proof}

\begin{thm} \label{cotorsion-periodity-type-equivalent-conditions}
 Let\/ $\sK$ be a Grothendieck category and $(\sA,\sB)$ be
a hereditary complete cotorsion pair generated by a set of objects\/
in\/~$\sK$.
 Then the following conditions are equivalent:
\begin{enumerate}
\item all acyclic complexes in\/ $\sA$ are Becker-coacyclic in\/ $\sA$,
i.~e., $\Ac(\sA)\subset\Ac^\bco(\sA)$;
\item the classes of acyclic and Becker-coacyclic complexes in\/ $\sA$
coincide, i.~e., $\Ac^\bco(\sA)=\Ac(\sA)$;
\item the classes\/ $\DG(\sB)$ and\/ $\Com(\sB)$ coincide, i.~e.,
$\DG(\sB)=\Com(\sB)$;
\item in any acyclic complex in\/ $\sK$ with the terms belonging
to\/ $\sB$, the objects of cocycles also belong to\/~$\sB$.
\end{enumerate}
\end{thm}

\begin{proof}
 (2)~$\Longleftrightarrow$~(3)
 By Propositions~\ref{coacyclic-all-cotorsion-pair-prop}
and~\ref{acyclic-dg-cotorsion-pair-prop}, we have two (hereditary
complete) cotorsion pairs $(\Ac^\bco(\sA),\>\Com(\sB))$ and
$(\Ac(\sA),\DG(\sB))$ in the abelian category $\Com(\sK)$.
 In any two cotorsion pairs in one and the same abelian/exact category,
the two left-hand classes coincide if and only if the two right-hand
classes coincide.

 (1)~$\Longleftrightarrow$~(2)
 Under the assumptions of the theorem, the inclusion
$\Ac^\bco(\sA)\subset\Ac(\sA)$ always holds, as one can see from
Propositions~\ref{coacyclic-all-cotorsion-pair-prop}
and~\ref{acyclic-dg-cotorsion-pair-prop}.
 Indeed, one has $\DG(\sB)\subset\Com(\sB)$, hence
$\Ac^\bco(\sA)={}^{\perp_1}(\Com(\sB))\subset{}^{\perp_1}(\DG(\sB))=
\Ac(\sA)$.
 See Corollary~\ref{coacyclic-in-A-are-acyclic-in-A} below for
a slightly more general result.

 (3)~$\Longrightarrow$~(4)
 Let $B^\bu$ be an acyclic complex in $\sK$ with the terms belonging
to~$\sB$.
 According to~(3), the complex $B^\bu$ belongs to~$\DG(\sB)$.
 By Lemma~\ref{DG-B-intersect-Ac-K-is-Ac-B}, it follows that $B^\bu$
is an acyclic complex in~$\sB$.

 (4)~$\Longrightarrow$~(3)
 We follow the argument from~\cite[Theorem~5.3]{BCE}.
 Let $B^\bu$ be a complex in~$\sB$.
 Consider a special preenvelope exact sequence $0\rarrow B^\bu\rarrow
D^\bu\rarrow A^\bu\rarrow0$ in $\Com(\sK)$ for the cotorsion pair
$(\Ac(\sA),\DG(\sB))$ from
Proposition~\ref{acyclic-dg-cotorsion-pair-prop}.
 So we have $D^\bu\in\DG(\sB)$ and $A^\bu\in\Ac(\sA)$.

 Since the class $\sB$ is closed under cokernels of admissible
monomorphisms in $\sK$, the terms of the complex $A^\bu$ actually
belong to $\sA\cap\sB$.
 So $A^\bu$ is an acyclic complex in $\sK$ with the terms in~$\sB$.
 By~(4), it follows that the objects of cocycles of the complex $A^\bu$
belong to~$\sB$.
 As $A^\bu\in\Ac(\sA)$ by assumption, we can conclude that the objects
of cocycles of the acyclic complex $A^\bu$ belong to $\sA\cap\sB$.
 As all admissible short exact sequences in $\sA\cap\sB$ are split,
this means that $A^\bu$ is a contractible complex with the terms
in $\sA\cap\sB$.

 The short exact sequence of complexes $0\rarrow B^\bu\rarrow D^\bu
\rarrow A^\bu\rarrow0$ is termwise split, so it follows that
the morphism $B^\bu\rarrow D^\bu$ is a homotopy equivalence of
complexes in~$\sB$.
 Using Lemma~\ref{Ext1-as-homotopy-Hom-lemma}, we finally see that
the complex $B^\bu$ belongs to $\DG(\sB)$ since the complex $D^\bu$
does.
\end{proof}

\Section{Coderived and Contraderived Categories for a Cotorsion Pair}

 Let $(\sA,\sB)$ be a hereditary complete cotorsion pair generated
by a set of objects in a Grothendieck category~$\sK$.
 In this section we study the Becker coderived category
$\sD^\bco(\sA)$ and the Becker contraderived category $\sD^\bctr(\sB)$.
 In the case of $\sD^\bco(\sA)$, the assumptions on the category $\sK$
can be relaxed.

\begin{cor} \label{coderived-of-A-well-behaved}
 Let\/ $\sK$ be either a locally presentable abelian category with
the abelian exact structure, or an efficient exact category in
the sense of\/~\cite{Sto-ICRA}.
 Let $(\sA,\sB)$ be a hereditary complete cotorsion pair generated
by a set of objects in\/~$\sK$.
 Then the composition of the triangulated inclusion functor\/
$\Hot(\sA^\inj)\rarrow\Hot(\sA)$ with the triangulated Verdier
quotient functor\/ $\Hot(\sA)\rarrow\sD^\bco(\sA)$ is a triangulated
equivalence
$$
 \Hot(\sA^\inj)\simeq\sD^\bco(\sA).
$$
\end{cor}

\begin{proof}
 Notice first of all that, by
Lemma~\ref{inj-proj-objects-in-left-right-classes}(a), there are enough
injective objects in $\sA$, and $\sA^\inj=\sA\cap\sB$.
 Furthermore, it follows immediately from the definition of
the Becker coderived category that the triangulated functor
$\Hot(\sA^\inj)\rarrow\sD^\bco(\sA)$ is fully faithful.
 In order to prove the corollary, it suffices to construct, for every
complex $C^\bu$ in $\sA$, a complex of injective objects $B^\bu$ in
$\sA$ together with a morphism of complexes $C^\bu\rarrow B^\bu$
whose cone is Becker-coacyclic in~$\sA$.

 According to the proof of
Proposition~\ref{coacyclic-all-cotorsion-pair-prop}, we have
a hereditary complete cotorsion pair
$(\Ac^\bco(\sA),\>\Com(\sA\cap\sB))$ in the exact category $\Com(\sA)$.
 Let $0\rarrow C^\bu\rarrow B^\bu\rarrow A^\bu\rarrow0$ be a special
preenvelope exact sequence of the object $C^\bu\in\Com(\sA)$ with
respect to this cotorsion pair.
 So we have $B^\bu\in\Com(\sA\cap\sB)$ and $A^\bu\in\Ac^\bco(\sA)$.

 It remains to point out that the totalization of the short exact
sequence $0\rarrow C^\bu\rarrow B^\bu\rarrow A^\bu\rarrow0$ is
Becker-coacyclic in $\sA$ by
Lemma~\ref{totalizations-are-co-contra-acyclic}(a).
 Since the complex $A^\bu$ is also Becker-coacyclic in $\sA$,
it follows that the cone of the morphism $C^\bu\rarrow B^\bu$ is
Becker-coacyclic in~$\sA$.
\end{proof}

 The following corollary is to be compared with~\cite[Lemma~B.7.3(a)
and Remark~B.7.4]{Pcosh}.

\begin{cor} \label{coacyclic-in-A-are-acyclic-in-A}
 Let\/ $\sK$ be either a locally presentable abelian category with
the abelian exact structure, or an efficient exact category in
the sense of\/~\cite{Sto-ICRA}.
 In the latter case, assume additionally that the cokernels of all
morphisms exist in\/~$\sK$.
 Let $(\sA,\sB)$ be a hereditary complete cotorsion pair generated
by a set of objects in\/~$\sK$.
 Then every Becker-coacyclic complex in\/ $\sA$ is acyclic in\/ $\sA$,
that is\/ $\Ac^\bco(\sA)\subset\Ac(\sA)$.
\end{cor}

\begin{proof}
 Let $A^\bu$ be a Becker-coacyclic complex in~$\sA$.
 Then, by Proposition~\ref{coacyclic-all-cotorsion-pair-prop} and
Lemma~\ref{Ext1-as-homotopy-Hom-lemma}, the complex of abelian groups
$\Hom_\sK(A^\bu,B^\bu)$ is acyclic for every complex $B^\bu$ in~$\sB$.
 In particular, the complex $\Hom_\sK(A^\bu,B)$ is acyclic for every
object $B\in\sB$.
 By Lemma~\ref{Hom-criterion-of-acyclicity-in-cotorsion-pair}(b), it
follows that the complex $A^\bu$ is acyclic in~$\sA$.
\end{proof}

\begin{cor} \label{contraderived-of-B-well-behaved}
 Let\/ $\sK$ be a Grothendieck category and $(\sA,\sB)$ be a hereditary
complete cotorsion pair generated by a set of objects in\/~$\sK$.
 Then the composition of the triangulated inclusion functor\/
$\Hot(\sB_\proj)\rarrow\Hot(\sB)$ with the triangulated Verdier
quotient functor\/ $\Hot(\sB)\rarrow\sD^\bctr(\sB)$ is a triangulated
equivalence
$$
 \Hot(\sB_\proj)\simeq\sD^\bctr(\sB).
$$
\end{cor}

\begin{proof}
 Notice that, by
Lemma~\ref{inj-proj-objects-in-left-right-classes}(b), there are enough
projective objects in $\sB$, and $\sB_\proj=\sA\cap\sB$.
 Furthermore, it follows immediately from the definition of
the Becker contraderived category that the triangulated functor
$\Hot(\sB_\proj)\rarrow\sD^\bctr(\sB)$ is fully faithful.
 In order to prove the corollary, it suffices to construct, for every
complex $C^\bu$ in $\sB$, a complex of projective objects $A^\bu$ in
$\sB$ together with a morphism of complexes $A^\bu\rarrow C^\bu$
whose cone is Becker-contraacyclic in~$\sB$.

 According to the proof of
Proposition~\ref{all-contraacyclic-cotorsion-pair-prop}, we have
a hereditary complete cotorsion pair
$(\Com(\sA\cap\sB),\>\Ac^\bctr(\sB))$ in the exact category $\Com(\sB)$.
 Let $0\rarrow B^\bu\rarrow A^\bu\rarrow C^\bu\rarrow0$ be a special
precover exact sequence of the object $C^\bu\in\Com(\sB)$ with
respect to this cotorsion pair.
 So we have $A^\bu\in\Com(\sA\cap\sB)$ and $B^\bu\in\Ac^\bctr(\sB)$.

 It remains to point out that the totalization of the short exact
sequence $0\rarrow B^\bu\rarrow A^\bu\rarrow C^\bu\rarrow0$ is
Becker-contraacyclic in $\sB$ by
Lemma~\ref{totalizations-are-co-contra-acyclic}(a).
 Since the complex $B^\bu$ is also Becker-contraacyclic in $\sB$,
it follows that the cone of the morphism $A^\bu\rarrow C^\bu$ is
Becker-contraacyclic in~$\sB$.
\end{proof}

 Once again, the following corollary should be compared
with~\cite[Lemma~B.7.3(b) and Remark~B.7.4]{Pcosh}.

\begin{cor} \label{contraacyclic-in-B-are-acyclic-in-B}
 Let\/ $\sK$ be a Grothendieck category and $(\sA,\sB)$ be a hereditary
complete cotorsion pair generated by a set of objects in\/~$\sK$.
 Then every Becker-contraacyclic complex in\/ $\sB$ is acyclic in\/
$\sB$, that is\/ $\Ac^\bctr(\sB)\subset\Ac(\sB)$.
\end{cor}

\begin{proof}
 Let $B^\bu$ be a Becker-contraacyclic complex in~$\sB$.
 Then, by Proposition~\ref{all-contraacyclic-cotorsion-pair-prop} and
Lemma~\ref{Ext1-as-homotopy-Hom-lemma}, the complex of abelian groups
$\Hom_\sK(A^\bu,B^\bu)$ is acyclic for every complex $A^\bu$ in~$\sA$.
 In particular, the complex $\Hom_\sK(A,B^\bu)$ is acyclic for every
object $A\in\sA$.
 By Lemma~\ref{Hom-criterion-of-acyclicity-in-cotorsion-pair}(a), it
follows that the complex $B^\bu$ is acyclic in~$\sB$.

 Alternatively, one can refer to
Propositions~\ref{all-contraacyclic-cotorsion-pair-prop}
and~\ref{dg-acyclic-cotorsion-pair-prop}.
 One has $\DG(\sA)\subset\Com(\sA)$, hence $\Ac^\bctr(\sB)\subset
(\Com(\sA))^{\perp_1}\subset(\DG(\sA))^{\perp_1}=\Ac(\sB)$.
\end{proof}

 The following lemma generalizes
Lemma~\ref{inj-proj-objects-in-left-right-classes}.

\begin{lem} \label{inj-proj-objects-in-intersected-classes}
 Let\/ $\sK$ be an exact category, $\sE\subset\sK$ be a full additive
subcategory, and $(\sA,\sB)$ be a complete cotorsion pair in~$\sK$. \par
\textup{(a)} Assume the full subcategory\/ $\sE$ is closed under
extensions and cokernels of admissible monomorphisms in\/ $\sK$, and
that\/ $\sB\subset\sE$.
 Then there are enough injective objects in the exact category\/
$\sE\cap\sA$, and the kernel\/ $\sA\cap\sB$ of the cotorsion pair
$(\sA,\sB)$ is the class of all injective objects in\/ $\sE\cap\sA$.
\par
\textup{(b)} Assume the full subcategory\/ $\sE$ is closed under
extensions and kernels of admissible epimorphisms in\/ $\sK$, and
that\/ $\sA\subset\sE$.
 Then there are enough projective objects in the exact category\/
$\sE\cap\sB$, and the kernel\/ $\sA\cap\sB$ of the cotorsion pair
$(\sA,\sB)$ is the class of all projective objects in\/ $\sE\cap\sB$.
\end{lem}

\begin{proof}
 To prove part~(a), it suffices to recall that $(\sE\cap\sA,\>\sB)$
is a complete cotorsion pair in $\sE$ by
Lemma~\ref{restricting-cotorsion-pair-lemma}(b).
 Applying Lemma~\ref{inj-proj-objects-in-left-right-classes}(a)
and noticing that $\sE\cap\sA\cap\sB=\sA\cap\sB$, we arrive to
the desired assertion.
 The proof of part~(b) is similar and based on
Lemmas~\ref{restricting-cotorsion-pair-lemma}(a)
and~\ref{inj-proj-objects-in-left-right-classes}(b).
\end{proof}

 The following corollary is a generalization
of Corollary~\ref{coderived-of-A-well-behaved}.

\begin{cor} \label{coderived-of-E-intersect-A-well-behaved}
 Let\/ $\sK$ be either a locally presentable abelian category with
the abelian exact structure, or an efficient exact category in
the sense of\/~\cite{Sto-ICRA}.
 Let $(\sA,\sB)$ be a hereditary complete cotorsion pair generated
by a set of objects in\/~$\sK$.
 Let\/ $\sE\subset\sK$ be a full additive subcategory closed under
extensions and cokernels of admissible monomorphisms in\/ $\sK$,
and such that\/ $\sB\subset\sE$.
 Consider the exact category\/ $\sF=\sE\cap\sA$.
 Then the composition of the triangulated inclusion functor\/
$\Hot(\sF^\inj)\rarrow\Hot(\sF)$ with the triangulated Verdier
quotient functor\/ $\Hot(\sF)\rarrow\sD^\bco(\sF)$ is a triangulated
equivalence
$$
 \Hot(\sF^\inj)\simeq\sD^\bco(\sF).
$$
\end{cor}

\begin{proof}
 Lemma~\ref{inj-proj-objects-in-intersected-classes}(a) tells us that
there are enough injective objects in $\sF$, and $\sF^\inj=\sA\cap\sB$.
 In other words, the classes of injective objects in the exact
categories $\sF$ and $\sA$ coincide.
 It follows that a complex in $\sF$ is Becker-coacyclic in $\sF$ if
and only if it is Becker-coacyclic as a complex in~$\sA$.

 As mentioned in the proof of
Corollary~\ref{coderived-of-A-well-behaved}, it is clear from
the definition of the Becker coderived category that the triangulated
functor $\Hot(\sF^\inj)\rarrow\sD^\bco(\sF)$ is fully faithful.
 In order to prove the corollary, we need to construct, for every
complex $C^\bu$ in $\sF$, a complex of injective objects $B^\bu$
in $\sF$ together with a morphism of complexes $C^\bu\rarrow B^\bu$
whose cone is Becker-coacyclic in~$\sF$.

 A complex $B^\bu$ in $\sA\cap\sB$ together with a morphism of
complexes $C^\bu\rarrow B^\bu$ whose cone is Becker-coacyclic
in $\sA$ was constructed in the proof of
Corollary~\ref{coderived-of-A-well-behaved}.
 It remains to point out that both $C^\bu$ and $B^\bu$ are complexes
in $\sF$, hence so is the cone.
 Since the cone is Becker-coacyclic in $\sA$, it is also
Becker-coacyclic in~$\sF$.
\end{proof}

 The next corollary is a generalization
of Corollary~\ref{contraderived-of-B-well-behaved}.

\begin{cor} \label{contraderived-of-E-intersect-B-well-behaved}
 Let\/ $\sK$ be a Grothendieck category and $(\sA,\sB)$ be a hereditary
complete cotorsion pair generated by a set of objects in\/~$\sK$.
 Let\/ $\sE\subset\sK$ be a full additive subcategory closed under
extensions and kernels of epimorphisms in\/ $\sK$, and such that\/
$\sA\subset\sE$.
 Consider the exact category\/ $\sC=\sE\cap\sB$.
 Then the composition of the triangulated inclusion functor\/
$\Hot(\sC_\proj)\rarrow\Hot(\sC)$ with the triangulated Verdier
quotient functor\/ $\Hot(\sC)\rarrow\sD^\bctr(\sC)$ is a triangulated
equivalence
$$
 \Hot(\sC_\proj)\simeq\sD^\bctr(\sC).
$$
\end{cor}

\begin{proof}
 The argument is dual to the proof of
Corollary~\ref{coderived-of-E-intersect-A-well-behaved}, and based on
Lemma~\ref{inj-proj-objects-in-intersected-classes}(b) and
(the proof of) Corollary~\ref{contraderived-of-B-well-behaved}.
\end{proof}

 The following corollary is the main result of this section.

\begin{cor} \label{cotorsion-pair-co-contra-equivalence}
 Let\/ $\sK$ be a Grothendieck category and $(\sA,\sB)$ be a hereditary
complete cotorsion pair generated by a set of objects in\/~$\sK$.
 Then there are natural equivalences of triangulated categories
$$
 \sD^\bco(\sA)\simeq\Hot(\sA\cap\sB)\simeq\sD^\bctr(\sB).
$$
\end{cor}

\begin{proof}
 Combine the results of
Corollaries~\ref{coderived-of-A-well-behaved}
and~\ref{contraderived-of-B-well-behaved}.
\end{proof}

\Section{Adjoint Functors for a Nested Pair of Cotorsion Pairs}

 Let $\sK$ be an exact category, and let $(\sA_1,\sB_1)$ and
$(\sA_2,\sB_2)$ be two cotorsion pairs in~$\sK$.
 We will say that the two cotorsion pairs $(\sA_1,\sB_1)$ and
$(\sA_2,\sB_2)$ form a \emph{nested pair of cotorsion pairs} if
$\sA_1\subset\sA_2$, or equivalently, $\sB_1\supset\sB_2$.
 If this is the case, we will also write
$(\sA_1,\sB_1)\le(\sA_2,\sB_2)$.

 The following two corollaries are to be compared
with~\cite[Lemma~B.7.5]{Pcosh}.

\begin{cor} \label{coderived-A1-to-A2-well-defined}
 Let\/ $\sK$ be either a locally presentable abelian category with
the abelian exact structure, or an efficient exact category in
the sense of\/~\cite{Sto-ICRA}.
 Let $(\sA_1,\sB_1)$ and $(\sA_2,\sB_2)$ be two hereditary complete
cotorsion pairs in\/ $\sK$, each of them generated by a set of objects.
 Assume that $(\sA_1,\sB_1)$ and $(\sA_2,\sB_2)$ are a nested pair of
cotorsion pairs, i.~e., $(\sA_1,\sB_1)\le(\sA_2,\sB_2)$.
 Then the inclusion of exact categories\/ $\sA_1\rarrow\sA_2$ induces
a well-defined triangulated functor between the Becker coderived
categories
$$
 \sD^\bco(\sA_1)\lrarrow\sD^\bco(\sA_2).
$$
 In other words, every Becker-coacyclic complex in\/ $\sA_1$ is also
Becker-coacyclic in\/~$\sA_2$.
\end{cor}

\begin{proof}
 This is a corollary of
Proposition~\ref{coacyclic-all-cotorsion-pair-prop}.
 We have $\sB_2\subset\sB_1$, hence $\Com(\sB_2)\subset\Com(\sB_1)$,
and it follows that $\Ac^\bco(\sA_1)={}^{\perp_1}(\Com(\sB_1))
\subset{}^{\perp_1}(\Com(\sB_2))=\Ac^\bco(\sA_2)$.
\end{proof}

\begin{cor} \label{contraderived-B2-to-B1-well-defined}
 Let\/ $\sK$ be a Grothendieck category, and let $(\sA_1,\sB_1)$ and
$(\sA_2,\sB_2)$ be two hereditary complete cotorsion pairs in\/ $\sK$,
each of them generated by a set of objects.
 Assume that $(\sA_1,\sB_1)$ and $(\sA_2,\sB_2)$ are a nested pair of
cotorsion pairs, i.~e., $(\sA_1,\sB_1)\le(\sA_2,\sB_2)$.
 Then the inclusion of exact categories\/ $\sB_2\rarrow\sB_1$ induces
a well-defined triangulated functor between the Becker contraderived
categories
$$
 \sD^\bctr(\sB_2)\lrarrow\sD^\bctr(\sB_1).
$$
 In other words, every Becker-contraacyclic complex in\/ $\sB_2$ is
also Becker-contraacyclic in\/~$\sB_1$.
\end{cor}

\begin{proof}
 This is a corollary of
Proposition~\ref{all-contraacyclic-cotorsion-pair-prop}.
 We have $\sA_1\subset\sA_2$, hence $\Com(\sA_1)\subset\Com(\sA_2)$,
and it follows that $\Ac^\bctr(\sB_2)=(\Com(\sA_2))^{\perp_1}
\subset(\Com(\sA_1))^{\perp_1}=\Ac^\bctr(\sB_1)$.
\end{proof}

 The following theorem is the main result of this section.
 
\begin{thm} \label{nested-cotorsion-pairs-adjunction-theorem}
 Let\/ $\sK$ be a Grothendieck category, and let $(\sA_1,\sB_1)$ and
$(\sA_2,\sB_2)$ be two hereditary complete cotorsion pairs in\/ $\sK$,
each of them generated by a set of objects.
 Assume that $(\sA_1,\sB_1)$ and $(\sA_2,\sB_2)$ are a nested pair of
cotorsion pairs, i.~e., $(\sA_1,\sB_1)\le(\sA_2,\sB_2)$.
 Then the identifications of triangulated categories from
Corollary~\ref{cotorsion-pair-co-contra-equivalence} make the functor\/
$\sD^\bco(\sA_1)\rarrow\sD^\bco(\sA_2)$ from
Corollary~\ref{coderived-A1-to-A2-well-defined} left adjoint to
the functor\/ $\sD^\bctr(\sB_2)\rarrow\sD^\bctr(\sB_1)$ from
Corollary~\ref{contraderived-B2-to-B1-well-defined},
$$
 \xymatrix{
  \sD^\bco(\sA_1) \ar@{=}[r] \ar[d]
  & \Hot(\sA_1\cap\sB_1) \ar@{=}[r] \ar@<-3pt>[d]_F
  & \sD^\bctr(\sB_1) \\
  \sD^\bco(\sA_2) \ar@{=}[r]
  & \Hot(\sA_2\cap\sB_2) \ar@<-3pt>[u]_G \ar@{=}[r]
  & \sD^\bctr(\sB_2) \ar[u]
 }
$$
\end{thm}

\begin{proof}
 Let us interpret both the functors as acting between the homotopy
categories $\Hot(\sA_1\cap\sB_1)$ and $\Hot(\sA_2\cap\sB_2)$, then
check that they are adjoint to each other.

 Given a complex $K^\bu\in\Hot(\sA_1\cap\sB_1)$, in order to
construct the complex $F(K^\bu)\in\Hot(\sA_2\cap\sB_2)$, we need
to perform the following steps.
 The complex $K^\bu$ has to be viewed as an object of the homotopy
category $\Hot(\sA_1)$, then as an object of the homotopy category
$\Hot(\sA_2)$.
 Finally, following the proof of
Corollary~\ref{coderived-of-A-well-behaved} applied to the cotorsion
pair $(\sA_2,\sB_2)$, we need to pick a short exact sequence
$0\rarrow K^\bu\rarrow U^\bu\rarrow A^\bu\rarrow0$ of complexes
in $\sK$ such that $U^\bu$ is a complex in $\sA_2\cap\sB_2$ and $A^\bu$
is a Becker-coacyclic complex in~$\sA_2$.
 Then we have $F(K^\bu)=U^\bu$.

 Given a complex $L^\bu\in\Hot(\sA_2\cap\sB_2)$, in order to
construct the complex $G(L^\bu)\in\Hot(\sA_1\cap\sB_1)$, we need
to proceed as follows.
 The complex $L^\bu$ has to be viewed as an object of the homotopy
category $\Hot(\sB_2)$, then as an object of the homotopy category
$\Hot(\sB_1)$.
 Finally, following the proof of
Corollary~\ref{contraderived-of-B-well-behaved} applied to the cotorsion
pair $(\sA_1,\sB_1)$, we need to pick a short exact sequence
$0\rarrow B^\bu\rarrow V^\bu\rarrow L^\bu\rarrow0$ of complexes
in $\sK$ such that $V^\bu$ is a complex in $\sA_1\cap\sB_1$ and $B^\bu$
is a Becker-contraacyclic complex in~$\sB_1$.
 Then we have $G(L^\bu)=V^\bu$.
 
 We need to construct a natural isomorphism of abelian groups
\begin{multline*}
 \Hom_{\Hot(\sA_2\cap\sB_2)}(F(K^\bu),L^\bu)=
 H^0\Hom_\sK(U^\bu,L^\bu) \\
 \,\simeq\, H^0\Hom_\sK(K^\bu,V^\bu)=
 \Hom_{\Hot(\sA_1\cap\sB_1)}(K^\bu,G(L^\bu)).
\end{multline*}
for all complexes $K^\bu\in\Hot(\sA_1\cap\sB_1)$ and
$L^\bu\in\Hot(\sA_2\cap\sB_2)$.
 For this purpose, we observe that the morphisms of complexes
$K^\bu\rarrow U^\bu$ and $V^\bu\rarrow L^\bu$ induce quasi-isomorphisms
of complexes of abelian groups
$$
 \Hom_\sK(U^\bu,L^\bu)\lrarrow\Hom_\sK(K^\bu,L^\bu)\llarrow
 \Hom_\sK(K^\bu,V^\bu).
$$

 Indeed, $A^\bu$ is a complex in $\sA_2$ and $L^\bu$ is a complex
in $\sB_2$, hence the short exact sequence of complexes
$0\rarrow K^\bu\rarrow U^\bu\rarrow A^\bu\rarrow0$ induces a short
exact sequence of complexes of abelian groups
$$
 0\lrarrow\Hom_\sK(A^\bu,L^\bu)\lrarrow\Hom_\sK(U^\bu,L^\bu)
 \lrarrow\Hom_\sK(K^\bu,L^\bu)\lrarrow0.
$$
 Since $A^\bu$ is a Becker-coacyclic complex in $\sA_2$ and $L^\bu$ is
a complex in $\sA_2\cap\sB_2=\sA_2^\inj$, the complex
$\Hom_\sK(A^\bu,L^\bu)$ is acyclic.
 Thus the map $\Hom_\sK(U^\bu,L^\bu)\rarrow\Hom_\sK(K^\bu,L^\bu)$
is a quasi-isomorphism.

 Similarly, $B^\bu$ is a complex in $\sB_1$ and $K^\bu$ is a complex
in $\sA_1$, hence the short exact sequence of complexes
$0\rarrow B^\bu\rarrow V^\bu\rarrow L^\bu\rarrow0$ induces a short
exact sequence of complexes of abelian groups
$$
 0\lrarrow\Hom_\sK(K^\bu,B^\bu)\lrarrow\Hom_\sK(K^\bu,V^\bu)
 \lrarrow\Hom_\sK(K^\bu,L^\bu)\lrarrow0.
$$
 Since $B^\bu$ is a Becker-contraacyclic complex in $\sB_1$ and
$K^\bu$ is a complex in $\sA_1\cap\sB_1=(\sB_1)_\proj$, the complex
$\Hom_\sK(K^\bu,B^\bu)$ is acyclic.
 Thus the map $\Hom_\sK(K^\bu,V^\bu)\rarrow\Hom_\sK(K^\bu,L^\bu)$
is a quasi-isomorphism.
\end{proof}

\begin{rem}
 Let $(\sA,\sB)$ be a hereditary complete cotorsion pair generated by
a set of objects in a Grothendieck category~$\sK$.
 Then Propositions~\ref{coacyclic-all-cotorsion-pair-prop}
and~\ref{all-contraacyclic-cotorsion-pair-prop} provide a nested pair
of hereditary complete cotorsion pairs $(\sA',\sB')\le(\sA'',\sB'')$
in the abelian category of complexes $\Com(\sK$).
 The kernels of the two cotorsion pairs in $\Com(\sK)$ agree: both
the classes $\sA'\cap\sB'$ and $\sA''\cap\sB''$ consist precisely of
all contractible complexes with the terms in $\sA\cap\sB$.

 By a theorem of Gillespie~\cite[Main Theorem~1.2]{Gil2b}, there
exists a unique abelian model structure on $\Com(\sK)$ in which
$\sA''=\Com(\sA)$ is the class of cofibrant objects,
$\sB'=\Com(\sB)$ is the class of fibrant objects,
$\sA'=\Ac^\bco(\sA)$ is the class of trivial cofibrant objects,
and $\sB''=\Ac^\bctr(\sB)$ is the class of trivial fibrant objects.
 The triangulated category
$\sD^\bco(\sA)\simeq\Hot(\sA\cap\sB)\simeq\sD^\bctr(\sB)$ from
Corollary~\ref{cotorsion-pair-co-contra-equivalence} is the homotopy
category of this stable model structure on the category of complexes
$\Com(\sK)$.

 Now let $(\sA_1,\sB_1)\le(\sA_2,\sB_2)$ be a nested pair of
cotorsion pairs in $\sK$, each of them hereditary, complete, and
generated by a set of objects.
 So $\sA_1\subset\sA_2$, or equivalently, $\sB_1\supset\sB_2$.
 Then we obviously have $\Com(\sA_1)\subset\Com(\sA_2)$ and
$\Com(\sB_1)\supset\Com(\sB_2)$.
 Corollaries~\ref{coderived-A1-to-A2-well-defined}
and~\ref{contraderived-B2-to-B1-well-defined} also say that
$\Ac^\bco(\sA_1)\subset\Ac^\bco(\sA_2)$
and $\Ac^\bctr(\sB_1)\supset\Ac^\bctr(\sB_2)$.
 Thus the adjunction of two identity functors
$\,\id\:\Com(\sK)\rightleftarrows\Com(\sK):\!\id\,$ is a Quillen
adjunction between the category $\Com(\sK)$ with the abelian model
structure arising from $(\sA_1,\sB_1)$ and the same category $\Com(\sK)$
with the abelian model structure arising from $(\sA_2,\sB_2)$ as
described in the previous paragraphs.

 The adjunction of triangulated functors between the homotopy
categories $\Hot(\sA_1\cap\sB_1)\rightleftarrows\Hot(\sA_2\cap\sB_2)$
from Theorem~\ref{nested-cotorsion-pairs-adjunction-theorem} is
induced by the Quillen adjunction of functors between abelian model
categories $\,\id\:\Com(\sK)\rightleftarrows\Com(\sK):\!\id$.
\end{rem}

\Section{Very Flat-Type Cotorsion Pairs}  \label{very-flat-type-secn}

 In this section we discuss cotorsion pairs $(\VF,\CA)$ in exact
categories $\sK$ such that all the objects of $\VF$ have finite
projective dimensions in~$\sK$.

 The abbreviation ``$\VF$'' stands for ``very flat''.
 The abbreviation ``$\CA$'' stands for ``contraadjusted''.
 See Examples~\ref{very-flat-and-very-flaprojective-in-modules-exs}
and~\ref{very-flat-and-very-flaprojective-in-qcoh-exs} below.

 Let us start with the standard definition.
 Given an object $P$ in an exact category~$\sK$ and an integer $d\ge-1$,
one says that the object $P$ has \emph{projective
dimension\/~$\le\nobreak d$} if $\Ext^{d+1}_\sK(P,X)=0$ for all
$X\in\sK$.

\begin{lem} \label{finite-projective-dimensions-are-bounded}
 Let\/ $\sK$ be an exact category with countable coproducts and
$(\VF,\CA)$ be a cotorsion pair in\/~$\sK$.
 Assume that all the objects of\/ $\VF$ have finite projective
dimensions in\/~$\sK$.
 Then the projective dimensions of the objects of\/ $\VF$ are
uniformly bounded: There exists a finite integer~$d$ such that, for
every object $P\in\VF$, the projective dimension of $P$ in\/ $\sK$
does not exceed~$d$.
\end{lem}

\begin{proof}
 By~\cite[Corollary 8.3]{CoFu} or the dual version
of~\cite[Corollary A.2]{CoSt} (extended straightforwardly from
abelian to exact categories), the class $\VF$ is closed under
coproducts in~$\sK$.
 Now if $P_i\in\VF$ is a sequence of objects such that the projective
dimension of $P_i$ in $\sK$ is not smaller than~$i$, then
the coproduct $\coprod_i P_i\in\VF$ is an object of infinite projective
dimension in~$\sK$.
\end{proof}

 Let $\sK$ be an exact category.
 A class of objects $\sB\subset\sK$ is said to be \emph{coresolving}
if $\sB$ is cogenerating (in the sense of the definition in
Section~\ref{preliminaries-secn}), closed under extensions, and closed
under cokernels of admissible monomorphisms in~$\sK$.
 Dually, a class of objects $\sA\subset\sK$ is said to be
\emph{resolving} if $\sA$ is generating, closed under extensions, and
closed under kernels of admissible epimorphisms in~$\sK$.

 Notice that the left-hand class $\sA$ of any hereditary cotorsion pair
in $\sK$ is resolving in~$\sK$.
 Dually, the right-hand class $\sB$ of any hereditary cotorsion pair in
$\sK$ is coresolving in~$\sK$.

 Assume that the exact category $\sK$ is weakly idempotent-complete
(in the sense of~\cite[Section~7]{Bueh}).
 Let $\sB$ be a coresolving class in $\sK$ and $d\ge-1$ be an integer.
 One says that the \emph{coresolution dimension} of an object
$K\in\sK$ with respect to the coresolving class $\sB$ \emph{does not
exceed~$d$} if there exists an (admissible) exact sequence $0\rarrow K
\rarrow B^0\rarrow B^1\rarrow\dotsb\rarrow B^d\rarrow0$ in $\sK$ with
$B^i\in\sB$ for all $0\le i\le d$.

 The coresolution dimension does not depend on the choice of
a coresolution, in the following sense.
 Suppose that the coresolution dimension of $K$ with respect to $\sB$
does not exceed~$d$, and let $0\rarrow K\rarrow B^0\rarrow B^1\rarrow
\dotsb\rarrow B^{d-1}\rarrow C\rarrow0$ be an exact sequence in $\sK$
with $B^i\in\sB$ for all $0\le i\le d-1$.
 Then $C\in\sB$.
 This result, proved on various generality levels, can be found
in~\cite[Lemma~2.1]{Zhu}, \cite[Proposition~2.3]{Sto0},
or~\cite[Corollary~A.5.2]{Pcosh}.

\begin{lem} \label{finite-coresolution-dimensions-are-bounded}
 Let\/ $\sK$ be a weakly idempotent-complete exact category such that
either countable coproducts or countable products exist in\/~$\sK$.
 Let\/ $\CA\subset\sK$ be a coresolving class of objects closed under
direct summands in\/~$\sK$.
 Assume that all the objects of\/ $\sK$ have finite coresolution
dimensions with respect to\/~$\CA$.
 Then there exists a finite integer~$d$ such that the coresolution
dimensions of all the objects of\/ $\sK$ with respect to\/ $\CA$ do
not exceed~$d$.
\end{lem}

\begin{proof}
 Let $K_i\in\sK$ be a sequence of objects such that the coresolution
dimension of $K_i$ with respect to $\CA$ is not smaller than~$i$.
 By assumption, one of the two objects $K=\coprod_i K_i$ or
$K=\prod_i K_i$ exists in~$\sK$.
 Now every object $K_i$ is a direct summand of~$K$.
 It remains to show that the $\sB$\+coresolution dimension of
a direct summand $K'$ cannot exceed the $\sB$\+coresolution dimension
of the direct sum $K'\oplus K''$ for any two objects $K'$, $K''\in\sK$
and any coresolving subcategory $\sB\subset\sK$ closed under direct
summands in~$\sK$.

 Indeed, let $d$~be a finite $\sB$\+coresolution dimension of
$K'\oplus K''$.
 Pick two exact sequences $0\rarrow K'\rarrow B'{}^0\rarrow B'{}^1
\rarrow\dotsb\rarrow B'{}^{d-1}\rarrow C'\rarrow0$ and
$0\rarrow K''\rarrow B''{}^0\rarrow B''{}^1\rarrow\dotsb\rarrow
B''{}^{d-1}\rarrow C''\rarrow0$ in $\sK$ with $B'{}^i$, $B''{}^i
\in\sB$ for all $0\le i\le d-1$.
 Then we have an exact sequence $0\rarrow K'\oplus K''\rarrow
B'{}^0\oplus B''{}^0\rarrow B'{}^1\oplus B''{}^1\rarrow\dotsb\rarrow
B'{}^{d-1}\oplus B''{}^{d-1}\rarrow C'\oplus C''\rarrow 0$ in $\sK$
with $B'{}^i\oplus B''{}^i\in\sB$ for all $0\le i\le d-1$.
 Since the $\sB$\+coresolution dimension of $K'\oplus K''$ does not
exceed~$d$ by assumption, it follows that $C'\oplus C''\in\sB$.
 As the class $\sB$ is closed under direct summands in $\sK$, we can
conclude that $C'\in\sB$.
 Thus the $\sB$\+coresolution dimension of $K'$ does not exceed~$d$,
either.
\end{proof}

\begin{lem} \label{co-resolving-Ext-agrees}
 Let\/ $\sK$ be an exact category and\/ $\sE\subset\sK$ be a resolving
or coresolving full subcategory.
 Then the inclusion of exact categories\/ $\sE\rarrow\sK$ induces
isomorphisms of the\/ $\Ext$ groups
$$
 \Ext^n_\sE(X,Y)\simeq\Ext^n_\sK(X,Y)
$$
for all objects $X$, $Y\in\sE$ and all integers $n\ge0$.
\end{lem}

\begin{proof}
 The assertion follows from~\cite[Theorem~12.1(b)]{Kel},
\cite[Proposition~13.2.2(i)]{KS},
or~\cite[Proposition~A.2.1 or~A.3.1(a)]{Pcosh}.
\end{proof}

 An exact category $\VF$ is said to have \emph{homological
dimension\/~$\le\nobreak d$} (where $d\ge-1$ is an integer) if
$\Ext_\VF^{d+1}(X,Y)=0$ for all objects $X$, $Y\in\VF$.
 The following lemma is a version of~\cite[Lemma~B.1.9]{Pcosh}.

\begin{lem} \label{very-flat-type-cotorsion-pair-lemma}
 Let\/ $\sK$ be a weakly idempotent-complete exact category and
$(\VF,\CA)$ be a hereditary cotorsion pair in\/~$\sK$.
 Then, for any fixed integer $d\ge-1$, the following three conditions
are equivalent:
\begin{enumerate}
\item the projective dimensions of all the objects of\/ $\VF$ do not
exceed~$d$ in\/~$\sK$;
\item the homological dimension of the exact category\/ $\VF$ does
not exceed~$d$;
\item the coresolution dimensions of all the objects of\/ $\sK$ with
respect to\/ $\CA$ do not exceed~$d$.
\end{enumerate}
\end{lem}

\begin{proof}
 (1)~$\Longrightarrow$~(2) holds by Lemma~\ref{co-resolving-Ext-agrees}.
 
 (2)~$\Longrightarrow$~(1)
 The proof of~\cite[Lemma~B.1.9\,(4)\,$\Rightarrow$\,(3)]{Pcosh}
applies.

 (1)~$\Longrightarrow$~(3)
 Let $0\rarrow K\rarrow C^0\rarrow C^1\rarrow\dotsb\rarrow C^{d-1}
\rarrow B\rarrow0$ be an exact sequence in $\sK$ with $C^i\in\CA$
for all $0\le i\le d-1$.
 The exact sequence $0\rarrow K\rarrow C^0\rarrow C^1\rarrow\dotsb
\rarrow C^{d-1}\rarrow B\rarrow0$ is obtained by splicing (admissible)
short exact sequences $0\rarrow Z^i\rarrow C^i\rarrow Z^{i+1}\rarrow0$
in~$\sK$; so $Z^0=K$ and $Z^d=B$.
 Then, for every object $P\in\VF$, we have $\Ext^1_\sK(P,B)\simeq
\Ext^2_\sK(P,Z^{d-1})\simeq\dotsb\simeq\Ext^d_\sK(P,Z^1)\simeq
\Ext^{d+1}_\sK(P,K)=0$.
 Hence $B\in\CA$.

 (3)~$\Longrightarrow$~(1)
 Given an object $K\in\sK$, pick an exact sequence $0\rarrow K\rarrow
C^0\rarrow C^1\rarrow\dotsb\rarrow C^{d-1}\rarrow C^d\rarrow0$ in $\sK$
with $C^i\in\CA$ for all $0\le i\le d$.
 Then, following the computation in the previous paragraph, for every
object $P\in\VF$ we have $\Ext^{d+1}_\sK(P,K)\simeq\Ext^1_\sK(P,C^d)=0$.
 Hence the projective dimension of $P$ does not exceed~$d$.
\end{proof}

\begin{cor} \label{very-flat-type-cotorsion-pair-cor}
 Let\/ $\sK$ be a weakly idempotent-complete exact category with
countable coproducts and $(\VF,\CA)$ be a hereditary cotorsion pair
in\/~$\sK$.
 Then the following conditions are equivalent:
\begin{enumerate}
\item all the objects of\/ $\VF$ have finite projective dimensions
in\/~$\sK$;
\item the exact category\/ $\VF$ has finite homological dimension;
\item all the objects of\/ $\sK$ have finite coresolution dimensions
with respect to\/~$\CA$.
\end{enumerate}
\end{cor}

\begin{proof}
 Follows from
Lemmas~	\ref{finite-projective-dimensions-are-bounded},
\ref{finite-coresolution-dimensions-are-bounded},
and~\ref{very-flat-type-cotorsion-pair-lemma}.
\end{proof}

\begin{lem} \label{exact-category-of-finite-homol-dim}
 Let\/ $\VF$ be an exact category of finite homological dimension.
 In this context: \par
\textup{(a)} Assume that there are enough injective objects in\/~$\VF$.  
 Then the classes of acyclic and Becker-coacyclic complexes in\/~$\VF$
coincide, that is\/ $\Ac^\bco(\VF)=\Ac(\VF)$.
 Furthermore, the composition of triangulated functors\/
$\Hot(\VF^\inj)\rarrow\Hot(\VF)\rarrow\sD^\bco(\VF)$ is a triangulated
equivalence
$$
 \Hot(\VF^\inj)\simeq\sD^\bco(\VF)=\sD(\VF).
$$ \par
\textup{(b)} Assume that there are enough projective objects in\/~$\VF$.
 Then the classes of acyclic and Becker-contraacyclic complexes
in\/~$\VF$ coincide, that is\/ $\Ac^\bctr(\VF)=\Ac(\VF)$.
 Furthermore, the composition of triangulated functors\/
$\Hot(\VF_\proj)\rarrow\Hot(\VF)\rarrow\sD^\bctr(\VF)$ is a triangulated
equivalence
$$
 \Hot(\VF_\proj)\simeq\sD^\bctr(\VF)=\sD(\VF).
$$
\end{lem}

\begin{proof}
 This is~\cite[Proposition~7.5]{Pphil}
or~\cite[Theorem~B.7.6]{Pcosh}.
\end{proof}

\begin{lem} \label{finite-coresol-dim-periodicity}
 Let\/ $\sK$ be an idempotent-complete exact category and\/
$\CA\subset\sK$ be a coresolving full subcategory.
 Assume that there is a finite integer~$d$ such that the coresolution
dimensions of all the objects of\/ $\sK$ with respect to\/ $\CA$ do
not exceed~$d$.
 Then every complex in\/ $\CA$ that is acyclic in\/ $\sK$ is also
acyclic in\/~$\CA$.
\end{lem}

\begin{proof}
 Let $C^\bu$ be a complex in $\CA$ obtained by splicing admissible
short exact sequences $0\rarrow Z^n\rarrow C^n\rarrow Z^{n+1}\rarrow0$
in~$\sK$.
 Then, for every $n\in\boZ$, we have an exact sequence
$0\rarrow Z^{n-d}\rarrow C^{n-d}\rarrow C^{n-d+1}\rarrow\dotsb
\rarrow C^{n-1}\rarrow Z^n\rarrow0$ in~$\sK$.
 As the object $Z^{n-d}\in\sK$ has coresolution dimension not
exceeding~$d$ with respect to $\CA$, it follows that $Z^n\in\CA$.
 For a generalization to the case of a weakly idempotent-complete
exact category\/ $\sK$, see~\cite[Corollary~A.5.4]{Pcosh}.
\end{proof}

 Let $\sK$ be a Grothendieck category.
 We will say that a hereditary complete cotorsion pair $(\VF,\CA)$
generated by a set of objects in\/ $\sK$ has \emph{very flat type} if
the equivalent conditions of
Corollary~\ref{very-flat-type-cotorsion-pair-cor} are satisfied.

\begin{cor}
 Let\/ $\sK$ be a Grothendieck category and $(\VF,\CA)$ be a cotorsion
pair of the very flat type.
 Then the equivalent conditions of
Theorem~\ref{cotorsion-periodity-type-equivalent-conditions} are
satisfied for the cotorsion pair $(\sA,\sB)=(\VF,\CA)$.
\end{cor}

\begin{proof}
 The condition of
Theorem~\ref{cotorsion-periodity-type-equivalent-conditions}(2) is
satisfied by Lemmas~\ref{inj-proj-objects-in-left-right-classes}(a)
and~\ref{exact-category-of-finite-homol-dim}(a).
 The condition of
Theorem~\ref{cotorsion-periodity-type-equivalent-conditions}(4) is
satisfied by Lemmas~\ref{finite-coresolution-dimensions-are-bounded}
and~\ref{finite-coresol-dim-periodicity}.
\end{proof}

\begin{lem} \label{very-flat-type-conventional-derived-equivalence}
 Let\/ $\sK$ be a Grothendieck category and $(\VF,\CA)$ be a cotorsion
pair of the very flat type.
 Then the inclusion of exact/abelian categories\/ $\CA\rarrow\sK$
induces an equivalence of the conventional derived categories\/
$$
 \sD(\CA)\simeq\sD(\sK).
$$
\end{lem}

\begin{proof}
 This is a special case of the dual assertion
to~\cite[Proposition~5.14]{Sto0} or of the dual assertion
to~\cite[Proposition~A.5.8]{Pcosh}.
\end{proof}

\begin{ex} \label{projective-cotorsion-pair-is-very-flat-type-ex}
 Let $\sK$ be a Grothendieck category with enough projective objects.
 Then the two classes $\VF=\sK_\proj$ and $\CA=\sK$ form a cotorsion
pair $(\VF,\CA)$ of the very flat type in~$\sK$.
 In other words, the projective cotorsion pair in a Grothendieck
category with enough projective objects is of the very flat type.
 In this case, one can take $d=0$ in the notation of
Lemma~\ref{very-flat-type-cotorsion-pair-lemma}.

 In particular, for any associative ring $A$, the pair of classes
of projective left $A$\+modules $\VF=A\Modl_\proj$ and all left
$A$\+modules $\CA=A\Modl$ is a cotorsion pair of the very flat type in
the module category $\sK=A\Modl$.
\end{ex}

\begin{exs} \label{very-flat-and-very-flaprojective-in-modules-exs}
 (1)~Let $R$ be a commutative ring.
 For any element $s\in R$, we denote by $R[s^{-1}]$ the ring $R$ with
the element~$s$ formally inverted, i.~e., the localization of
the ring $R$ at the multiplicative subset $\{1,s,s^2,s^3,\dotsc\}
\subset R$.

 The pair of classes of $R$\+modules $(R\Modl_\vfl,\>R\Modl^\cta)$ is
defined as the cotorsion pair generated by the set of all $R$\+modules
$R[s^{-1}]$, \,$s\in R$.
 The $R$\+modules belonging to $R\Modl_\vfl$ are called \emph{very
flat}, while the $R$\+modules belonging to $R\Modl^\cta$ are called
\emph{contraadjusted}.
 The pair of classes $\VF=R\Modl_\vfl$ and $\CA=R\Modl^\cta$ is
a hereditary complete cotorsion pair in $R\Modl$, and the projective
dimensions of very flat $R$\+modules do not exceed~$1$
\,\cite[Section~1.1]{Pcosh}, \cite[Section~4.3]{Pphil}.
 Thus $(R\Modl_\vfl,\>R\Modl^\cta)$ is a cotorsion pair of the very flat
type in $R\Modl$, and in this case one can take $d=1$ in the notation
of Lemma~\ref{very-flat-type-cotorsion-pair-lemma}.
 
\smallskip
 (2)~Let $R$ be a commutative ring, $A$ be an associative ring, and
$R\rarrow A$ be a ring homomorphism.
 A left $A$\+module $F$ is said to be \emph{$A/R$\+very flaprojective}
if $\Ext^1_A(F,C)=0$ for all left $A$\+modules $C$ that are
\emph{contraadjusted as $R$\+modules}.
 Let us call the latter $A$\+modules \emph{$R$\+contraadjusted}, and
denote the class of $R$\+contraadjusted left $A$\+modules by
$A\Modl^{R\dcta}$.
 In particular, all injective $A$\+modules are $R$\+contraadjusted
by~\cite[Lemma~1.3.4(a)]{Pcosh}, \cite[Lemma~1.6(a)]{Pdomc},
or~\cite[Lemma~2.3]{PBas} (see
Lemma~\ref{restriction-of-scalars-cotorsion} below).
 The class of $A/R$\+very flaprojective left $A$\+modules is denoted
by $A\Modl_{A/R\dvflp}$.

 The pair of classes of $A/R$\+very flaprojective left $A$\+modules
$\VF=A\Modl_{A/R\dvflp}$ and $R$\+contraadjusted left $A$\+modules
$\CA=A\Modl^{R\dcta}$ is a hereditary complete cotorsion pair in
$A\Modl$ generated by the set of all left $A$\+modules of the form
$A\ot_RR[s^{-1}]$, \,$s\in R$ \,\cite[Section~1.6]{Pdomc}.
 All $R$\+modules have coresolution dimension at most~$1$ with respect
to the coresolving subcategory of contraadjusted $R$\+modules, and it
follows that all $A$\+modules have coresolution dimension at most~$1$
with respect to the coresolving subcategory of $R$\+contraadjusted
$A$\+modules.
 Hence by Lemma~\ref{very-flat-type-cotorsion-pair-lemma} the projective
dimensions of $A/R$\+very flaprojective $A$\+modules do not exceed~$1$.
 Thus $(A\Modl_{A/R\dvflp},\>A\Modl^{R\dcta})$ is a cotorsion pair of
the very flat type in $A\Modl$, and in this case one can take $d=1$ in
the notation of Lemma~\ref{very-flat-type-cotorsion-pair-lemma}.
\end{exs}

\begin{exs} \label{very-flat-and-very-flaprojective-in-qcoh-exs}
 (1)~Let $X$ be a quasi-compact semi-separated scheme with a finite
affine open covering $X=\bigcup_{\alpha=1}^N U_\alpha$.
 Consider the Grothendieck category $\sK=X\Qcoh$ of quasi-coherent
sheaves on~$X$.
 Let us denote by $\Ext_X^*({-},{-})$ the $\Ext$ groups in the abelian
category $X\Qcoh$.

 A quasi-coherent sheaf $\cF$ on $X$ is said to be \emph{very flat}
if, for every affine open subscheme $U\subset X$,
the $\cO_X(U)$\+module $\cF(U)$ is very flat.
 It suffices to check this condition for the affine open subschemes
$U=U_\alpha$, \,$1\le\alpha\le N$ \,\cite[Section~1.10]{Pcosh}.
 A quasi-coherent sheaf $\cC$ on $X$ is said to be \emph{contraadjusted}
if $\Ext^1_X(\cF,\cC)=0$ for all very flat quasi-coherent sheaves $\cF$
on~$X$ \,\cite[Section~2.5]{Pcosh}.
 Denote the class of very flat quasi-coherent sheaves by
$X\Qcoh_\vfl\subset X\Qcoh$ and the class of contraadjusted
quasi-coherent sheaves by $X\Qcoh^\cta\subset X\Qcoh$.

 The pair of classes $\VF=X\Qcoh_\vfl$ and $\CA=X\Qcoh^\cta$ is
a hereditary complete cotorsion pair in $X\Qcoh$
\,\cite[Corollary~4.1.4 and Lemma~B.1.2]{Pcosh}.
 The class of very flat quasi-coherent sheaves on~$X$ is deconstructible
in $X\Qcoh$ \,\cite[Lemma~B.3.5 and the beginning of the proof of
Lemma~5.1.1(a)]{Pcosh}; so this cotorsion pair is generated by a set
of objects.
 The projective dimensions of very flat quasi-coherent sheaves on $X$
do not exceed the number $N$ of affine open subschemes in any given
affine open covering of~$X$ \,\cite[Lemmas~4.7.1(b), 5.5.12(a,c),
and~B.1.9]{Pcosh}.
 Thus $(X\Qcoh_\vfl,\>X\Qcoh^\cta)$ is a cotorsion pair of the very flat
type in $X\Qcoh$, and in this case one can take $d=N$ in the notation
of Lemma~\ref{very-flat-type-cotorsion-pair-lemma}.

\smallskip
 (2)~Let $\cA$ be a quasi-coherent sheaf of associative algebras
over~$\cO_X$.
 More generally, the following discussion also applies to the case
when $\cA$ is a \emph{quasi-coherent quasi-algebra} over~$X$ in
the sense of~\cite[Section~3.1]{Pdomc}.
 This means that $\cA$ is a sheaf of associative rings over $X$
endowed with a morphism of sheaves of rings $\cO_X\rarrow\cA$ such
that the image of $\cO_X$ is not central in $\cA$ but is ``close
to being central'' in some sense spelled out
in~\cite[Sections~1.4\+-1.5]{Pdomc}.
 Under the latter assumption, one can prove that $\cA$ is
quasi-coherent as an $\cO_X$\+module with respect to its left
$\cO_X$\+module structure if and only if it is quasi-coherent as
an $\cO_X$\+module with respect to its right $\cO_X$\+module
structure~\cite[Lemma~3.2]{Pdomc}; this quasi-coherence condition
is also assumed.

 A \emph{quasi-coherent left $\cA$\+module} on $X$ is a sheaf of
left $\cA$\+modules whose underlying sheaf of $\cO_X$\+modules is
quasi-coherent~\cite[Section~3.3]{Pdomc}.
 Consider the Grothendieck category $\sK=\cA\Qcoh$ of quasi-coherent
left $\cA$\+modules on~$X$.

 A quasi-coherent left $\cA$\+module $\cC$ is said to be
\emph{$X$\+contraadjusted} if the underlying quasi-coherent sheaf
(of $\cO_X$\+modules) of $\cC$ is contraadjusted.
 A quasi-coherent left $\cA$\+module $\cF$ is said to be
\emph{very flaprojective} if, for every affine open subscheme
$U\subset X$, the left $\cA(U)$\+module $\cF(U)$ is
$\cA(U)/\cO_X(U)$\+very flaprojective.
 It suffices to check this condition for the affine open subschemes
$U=U_\alpha$, \,$1\le\alpha\le N$ \,\cite[Section~4.1]{Pdomc}.
 Denote the class of very flaprojective quasi-coherent left
$\cA$\+modules by $\cA\Qcoh_\vflp\subset\cA\Qcoh$ and the class of
$X$\+contraadjusted quasi-coherent left $\cA$\+modules by
$\cA\Qcoh^{X\dcta}\subset\cA\Qcoh$.

 The pair of classes $\VF=\cA\Qcoh_\vflp$ and $\CA=\cA\Qcoh^{X\dcta}$ is
a hereditary complete cotorsion pair in $\cA\Qcoh$
\,\cite[Theorem~4.5 and Remark~4.7]{Pdomc}.
 Similarly to~\cite[Lemma~B.3.5]{Pcosh}, one can prove that the class
of very flaprojective quasi-coherent left $\cA$\+modules is
deconstructible in $\cA\Qcoh$ (since the class of $A/R$\+very
flaprojective left $A$\+modules is deconstructible in $A\Modl$ in
the context of
Example~\ref{very-flat-and-very-flaprojective-in-modules-exs}(2));
so this cotorsion pair is generated by a set of objects.

 All quasi-coherent sheaves on $X$ have coresolution dimension at most
$N$ with respect to the coresolving subcategory of contraadjusted
quasi-coherent sheaves, and it follows that all quasi-coherent
$\cA$\+modules have coresolution dimension at most $N$ with respect to
the coresolving subcategory of $X$\+contraadjusted quasi-coherent
$\cA$\+modules~\cite[Lemma~5.7]{Pdomc}.
 Hence by Lemma~\ref{very-flat-type-cotorsion-pair-lemma} the projective
dimensions of very flaprojective quasi-coherent $\cA$\+modules do not
exceed~$N$.
 Thus $(\cA\Qcoh_\vflp,\allowbreak\>\cA\Qcoh^{X\dcta})$ is a cotorsion
pair of the very flat type in $\cA\Qcoh$, and in this case one can take
$d=N$ in the notation of
Lemma~\ref{very-flat-type-cotorsion-pair-lemma}.
\end{exs}

\begin{exs} \label{flat-projdim1-modules-and-comodules-exs}
 (1)~Let $A$ be an associative ring and $n\ge0$ be an integer.
 Denote by $A\Modl^\pdn\subset A\Modl$ the full subcategory of left
$A$\+modules of projective dimension at most~$n$.
 Let $A\Modl^\pdn_\flat=A\Modl^\pdn\cap A\Modl_\flat$ be the full
subcategory of flat $A$\+modules of projective dimension~$\le n$
in $A\Modl$.

 For example, all countably presented flat $A$\+modules have projective
dimensions at most~$1$ by~\cite[Corollary~2.23]{GT}.
 Moreover, all $<\nobreak\aleph_n$\+presented flat $A$\+modules have
projective dimensions at most~$n$ by~\cite[Proposition~5.3]{Jen}.

 By~\cite[Lemmas~6.4 and~8.9]{GT}, the class of modules $A\Modl^\pdn$
is deconstructible in $A\Modl$.
 According to~\cite[Section~2]{BBE}, the class of flat modules
$A\Modl_\flat$ is also deconstructible.
 According to~\cite[Proposition~2.9(2)]{Sto-hill}, it follows that
the class of modules $A\Modl^\pdn_\flat=A\Modl^\pdn\cap A\Modl_\flat$
is deconstructible as well.

 Both the classes $A\Modl^\pdn$ and $A\Modl^\pdn_\flat$ are also
clearly resolving (in particular, generating) and closed under
direct summands in $A\Modl$.
 By the Eklof--Trlifaj theorem, it follows that both $(A\Modl^\pdn,
\allowbreak\>(A\Modl^\pdn)^{\perp_1})$ and $(A\Modl^\pdn_\flat,
\allowbreak\>(A\Modl^\pdn_\flat)^{\perp_1})$ are hereditary complete
cotorsion pairs in $A\Modl$, each of them generated by a set of objects.

 By construction, all the objects of $A\Modl^\pdn$ and
$A\Modl^\pdn_\flat$ have projective dimensions at most~$n$ in $A\Modl$.
 Thus both $(\VF,\CA)=(A\Modl^\pdn,\allowbreak\>
(A\Modl^\pdn)^{\perp_1})$ and $(\VF,\CA)=(A\Modl^\pdn_\flat,
\allowbreak\>(A\Modl^\pdn_\flat)^{\perp_1})$ are
cotorsion pairs of the very flat type in $A\Modl$.
 In both cases, one can take $d=n$ in the notation of
Lemma~\ref{very-flat-type-cotorsion-pair-lemma}.

\smallskip
 (2)~Let $\cE$ be a \emph{coring} over the ring~$A$.
 So $\cE$ is an $A$\+$A$\+bimodule endowed with $A$\+$A$\+bimodule
maps of \emph{comultiplication} $\cE\rarrow\cE\ot_A\cE$ and
\emph{counit} $\cE\rarrow A$ satisfying the usual coassociativity and
counitality axioms~\cite[Section~17.1]{BW}, \cite[Section~1.1]{Psemi},
\cite[Section~2]{Pflcc}.
 A \emph{left\/ $\cE$\+comodule} $\cM$ is a left $A$\+module endowed
with a left $A$\+module map of \emph{left\/ $\cE$\+coaction}
$\cM\rarrow\cE\ot_A\cM$ satisfying the coassociativity and counitality
axioms~\cite[Sections~18.1\+-3]{BW}, \cite[Section~1.1]{Psemi},
\cite[Section~2]{Pflcc}.

 Assume that $\cE$ is flat as a right $A$\+module.
 Then the category of left $\cE$\+comodules $\cE\Comodl$ is 
a Grothendieck abelian category~\cite[Sections~18.6, 18.14, 18.16,
and~18.23]{BW}, \cite[Proposition~2.12(a)]{Prev},
\cite[Lemma~2.1]{Pflcc}.
 The forgetful functor $\cE\Comodl\rarrow A\Modl$ is exact and
preserves coproducts.

 Furthermore, an $\cE$\+comodule $\cF$ is said to be \emph{$A$\+flat}
if its underlying $A$\+module $\cF$ is flat.
 The category $\cE\Comodl$ can be viewed as a kind of quasi-coherent
sheaf category over a noncommutative quasi-compact semi-separated
stack~\cite[Section~2]{KR}, \cite{KR2}, \cite[Example~2.5]{Pflcc}.
 The $A$\+flat $\cE$\+comodules correspond precisely to flat
quasi-coherent sheaves~\cite[Remark~2.6]{Pflcc}.

 Assume that the coring $\cE$ satisfies the following two additional
conditions discussed in the paper~\cite[Section~5]{Pflcc}:
\begin{itemize}
\item[($*$)] every left $\cE$\+comodule is a quotient comodule of
an $A$\+flat $\cE$\+comodule;
\item[(${*}{*}$)] a left $\cE$\+comodule $\cM$ has finite projective
dimension as an object of $\cE\Comodl$ whenever its underlying
$A$\+module $\cM$ has finite projective dimension as an object of
$A\Modl$.
\end{itemize}

 Clearly, it follows from~(${*}{*}$) that there exists an integer
$d\ge1$ such that a left $\cE$\+comodule $\cM$ has projective
dimension~$\le d$ whenever its underlying $A$\+module $\cM$ has
projective dimension~$\le1$.
 According to~\cite[Theorem~3.1]{Pflcc}, it follows from~(${*}$)
that every left $\cE$\+comodule is a quotient comodule of a coproduct
of $\cE$\+comodules whose underlying $A$\+modules are flat and
countably presented.
 All such $A$\+modules have projective dimensions at most~$1$, as
mentioned in~(1) above.

 Now let $\VF=\cE\Comodl^{A\dpdo}_{A\dflat}\subset\cE\Comodl$ denote
the class of all left $\cE$\+comodules whose underlying left
$A$\+modules are flat of projective dimension~$\le 1$.
 According to the previous paragraph and the discussion in~(1),
the full subcategory $\cE\Comodl^{A\dpdo}_{A\dflat}$ is resolving
in $\cE\Comodl$; it is also clearly closed under direct summands.
 Using the Hill lemma~\cite[Theorem~7.10]{GT},
\cite[Theorem~2.1]{Sto-hill} in the module category $A\Modl$ and
the fact that the class of modules $A\Modl^\pdo_\flat$ is
deconstructible (as explained in~(1)), one can show that the class of
comodules $\cE\Comodl^{A\dpdo}_{A\dflat}$ is deconstructible
in $\cE\Comodl$.

 By the Eklof--Trlifaj theorem (one can use either
Theorem~\ref{loc-pres-eklof-trlifaj} or
Theorem~\ref{efficient-eklof-trlifaj}), it follows that
$(\VF,\CA)=(\cE\Comodl^{A\dpdo}_{A\dflat},\allowbreak\>
(\cE\Comodl^{A\dpdo}_{A\dflat})^{\perp_1})$ is a hereditary complete
cotorsion pair generated by a set of objects in $\cE\Comodl$.
 As we have seen, all the objects of $\cE\Comodl^{A\dpdo}_{A\dflat}$
have finite projective dimensions~$\le\nobreak d$ in $\cE\Comodl$
under our assumptions.
 Thus $(\VF,\CA)$ is a cotorsion pair of the very flat type in
$\cE\Comodl$.
\end{exs}

\Section{Flat-Type Cotorsion Pairs}  \label{flat-type-secn}

 The definition of a cotorsion pair of the flat type $(\Flat,\Cot)$ in
a Grothendieck category $\sK$ uses an auxiliary cotorsion pair of
the very flat type $(\VF,\CA)$ such that $(\VF,\CA)\le(\Flat,\Cot)$.
 We prove that the property of $(\Flat,\Cot)$ to be of the flat type
does not depend on the choice of $(\VF,\CA)$.

 Let us start with a more abstract definition.
 We will say that an exact category $\sF$ is \emph{of the flat type} if
the following three conditions are satisfied:
\begin{enumerate}
\renewcommand{\theenumi}{\roman{enumi}}
\item there are enough injective and projective objects in~$\sF$;
\item the classes of acyclic, Becker-coacyclic, and
Becker-contraacyclic complexes in $\sF$ coincide, i.~e.,
$\Ac^\bco(\sF)=\Ac(\sF)=\Ac^\bctr(\sF)$;
\item the compositions of triangulated functors
$\Hot(\sF^\inj)\rarrow\Hot(\sF)\rarrow\sD^\bco(\sF)=\sD(\sF)$
and $\Hot(\sF_\proj)\rarrow\Hot(\sF)\rarrow\sD^\bctr(\sF)=\sD(\sF)$
are triangulated equivalences
$$
 \Hot(\sF^\inj)\simeq\sD^\bco(\sF)=\sD(\sF)=\sD^\bctr(\sF)\simeq
 \Hot(\sF_\proj).
$$
\end{enumerate}
 We will say that an exact category $\sF$ is \emph{weakly of the flat
type} if it satisfies conditions~(i\+-ii).
 By Lemma~\ref{exact-category-of-finite-homol-dim}, any exact
category of finite homological dimension with enough injective and
projective objects is of the flat type.

\begin{lem} \label{proj-resolving-inj-coresolving}
\textup{(a)} Let\/ $\sE$ be an exact category and\/ $\sA\subset\sE$
be a resolving full subcategory closed under direct summands in\/~$\sE$.
 Then the classes of projective objects in the exact categories\/ $\sE$
and\/ $\sA$ coincide, $\sA_\proj=\sE_\proj$.
 There are enough projective objects in\/ $\sA$ if and only if
there are enough projective objects in\/~$\sE$. \par
\textup{(b)} Let\/ $\sE$ be an exact category and\/ $\sB\subset\sE$ be
a coresolving full subcategory closed under direct summands in\/~$\sE$.
 Then the classes of injective objects in the exact categories\/ $\sE$
and\/ $\sB$ coincide, $\sB^\inj=\sE^\inj$.
 There are enough injective objects in\/ $\sB$ if and only if
there are enough injective objects in\/~$\sE$.
\end{lem}

\begin{proof}
 It is explained in~\cite[first paragraph of the proof of
Proposition~B.7.9]{Pcosh} how to show that all the projective objects
of $\sA$ are projective in~$\sE$.
 Conversely, if $P\in\sE$ is a projective object, pick an object
$A\in\sA$ together with an admissible epimorphism $A\rarrow P$ in~$\sE$.
 Then $P$ is a direct summand of $A$ in $\sE$, hence $P\in\sA$.
 The second assertions of both parts~(a) and~(b) are even easier.
\end{proof}

 The following proposition can be viewed as a generalization of
Lemma~\ref{exact-category-of-finite-homol-dim}.

\begin{prop} \label{flat-type-equivalence-for-cotorsion-pair-prop}
 Let\/ $\sE$ be a weakly idempotent-complete exact category with
enough injective and projective objects, and let $(\sA,\sB)$ be
a hereditary complete cotorsion pair in\/~$\sE$.
 In this setting: \par
\textup{(a)} Assume that the exact category\/ $\sA$ has finite
homological dimension.
 Then the exact category\/ $\sB$ is (weakly) of the flat type if and
only if the exact category\/ $\sE$ is (weakly) of the flat type. \par
\textup{(b)} Assume that the exact category\/ $\sB$ has finite
homological dimension.
 Then the exact category\/ $\sA$ is (weakly) of the flat type if and
only if the exact category\/ $\sE$ is (weakly) of the flat type.
\end{prop}

\begin{proof}
 Let us prove part~(a).
 By Lemma~\ref{proj-resolving-inj-coresolving}(b), the classes of
injective objects in $\sB$ and $\sE$ coincide, $\sB^\inj=\sE^\inj$, and
there are enough such objects in $\sB$ since there are enough of them
in~$\sE$.
 By Lemma~\ref{inj-proj-objects-in-left-right-classes}(b), there are
also enough projective objects in $\sB$, and one has
$\sB_\proj=\sA\cap\sB$.
 So condition~(i) is satisfied for both $\sB$ and~$\sE$.

 Furthermore, by Lemma~\ref{very-flat-type-cotorsion-pair-lemma},
the coresolution dimensions of all the objects of $\sE$ with
respect to $\sB$ do not exceed the homological dimension~$d$ of
the exact category~$\sA$.
 By Lemma~\ref{finite-coresol-dim-periodicity}, it follows that
a complex in $\sB$ is acyclic if and only if it is acyclic in~$\sE$,
that is $\Ac(\sB)=\Com(\sB)\cap\Ac(\sE)$.
 Since $\sB^\inj=\sE^\inj$, it follows immediately from the definitions
that a complex in $\sB$ is Becker-coacyclic if and only if it is
Becker-coacyclic in $\sE$, that is $\Ac^\bco(\sB)=\Com(\sB)\cap
\Ac^\bco(\sE)$.
 By~\cite[Proposition~B.7.10]{Pcosh}, a complex in $\sB$ is
Becker-contraacyclic if and only if it is Becker-contraacyclic
in $\sE$, that is $\Ac^\bctr(\sB)=\Com(\sB)\cap\Ac^\bctr(\sE)$.
 Thus condition~(ii) for the exact category $\sE$ implies condition~(ii)
for the exact category~$\sB$.

 Conversely, let $E^\bu$ be a complex in~$\sE$.
 Since the category $\sB$ is coresolving in~$\sE$, there exists
a contractible complex $B^{0,\bu}$ in $\sB$ together with a termwise
admissible monomorphism of complexes $E^\bu\rarrow B^{0,\bu}$.
 Similarly, there exists a contractible complex $B^{1,\bu}$ in $\sB$
together with a termwise admissible monomorphism of complexes
$B^{0,\bu}/E^\bu\rarrow B^{1,\bu}$.
 Proceeding in this way, we construct a finite exact sequence of
complexes $0\rarrow E^\bu\rarrow B^{0,\bu}\rarrow B^{1,\bu}\rarrow
\dotsb\rarrow B^{d-1,\bu}\rarrow F^\bu\rarrow0$ in $\sE$ such that
all the complexes $B^{n,\bu}$, \,$0\le n\le d-1$, are contractible
complexes in~$\sB$.
 Since the coresolution dimensions of all the objects of $\sE$ with
respect to $\sB$ do not exceed~$d$, the terms of the complex $F^\bu$
also belong to~$\sB$.
 Now the complex $E^\bu$ is acyclic in $\sE$ if and only if the complex
$F^\bu$ is acyclic in~$\sE$ (by~\cite[Corollary~3.6]{Bueh}),
the complex $E^\bu$ is Becker-coacyclic in $\sE$ if and only if
the complex $F^\bu$ is Becker-coacyclic in~$\sE$
(by Lemma~\ref{totalizations-are-co-contra-acyclic}(a)), and
the complex $E^\bu$ is Becker-contraacyclic in $\sE$ if and only if
the complex $F^\bu$ is Becker-contraacyclic in~$\sE$
(by Lemma~\ref{totalizations-are-co-contra-acyclic}(a)).
 Taking the previous paragraph into account, one can see that
condition~(ii) for the exact category $\sB$ implies condition~(ii)
for the exact category~$\sE$.

 Finally, the inclusion of exact categories $\sB\rarrow\sE$ induces
a triangulated equivalence of the conventional derived categories
$\sD(\sB)\simeq\sD(\sE)$ by the dual version
of~\cite[Proposition~5.14]{Sto0} or the dual version
of~\cite[Proposition~A.5.8]{Pcosh}.
 The same inclusion of exact categories induces a triangulated
equivalence of the Becker coderived categories $\sD^\bco(\sB)
\simeq\sD^\bco(\sE)$ by the dual version
of~\cite[Proposition~B.7.9]{Pcosh}, and a triangulated equivalence of
the Becker contraderived categories $\sD^\bctr(\sB)\simeq
\sD^\bctr(\sE)$ by~\cite[Proposition~B.7.10]{Pcosh}.

 Besides, by Lemma~\ref{proj-resolving-inj-coresolving}(a), the classes
of projective objects in $\sA$ and $\sE$ coincide,
$\sA_\proj=\sE_\proj$, and there are enough such objects in $\sA$ since
there are enough of them in~$\sE$.
 One has $\sD^\bco(\sA)=\sD(\sA)=\sD^\bctr(\sA)$, and
the functors $\Hot(\sB_\proj)=\Hot(\sA^\inj)\rarrow\sD(\sA)$ and
$\Hot(\sE_\proj)=\Hot(\sA_\proj)\rarrow\sD(\sA)$ are triangulated
equivalences by Lemma~\ref{exact-category-of-finite-homol-dim}.
 In view of these observations, condition~(iii) for the exact category
$\sB$ is equivalent to condition~(iii) for the exact category~$\sE$.
\end{proof}

\begin{rem} \label{flat-type-equivalence-for-cotorsion-pair-remark}
 Looking into the proof of
Proposition~\ref{flat-type-equivalence-for-cotorsion-pair-prop}, one
can see that it actually proves the equivalence of the respective
subconditions of the condition of (weakly) flat type.
 Specifically, let $\sE$ be a weakly idempotent-complete exact category
and $(\sA,\sB)$ be a hereditary complete cotorsion pair in~$\sE$.
 In this setting: \par
\textup{(a)} if $\sA$ has finite homological dimension, then
$\Ac^\bco(\sB)=\Ac(\sB)$ if and only if $\Ac^\bco(\sE)=\Ac(\sE)$; \par
\textup{(b)} if $\sA$ has finite homological dimension, then the functor
$\Hot(\sB^\inj)\rarrow\sD^\bco(\sB)$ is an equivalence if and only if
the functor $\Hot(\sE^\inj)\rarrow\sD^\bco(\sE)$ is an equivalence; \par
\textup{(c)} if $\sA$ has finite homological dimension and there are
enough projective objects in $\sE$, then $\Ac^\bctr(\sB)=\Ac(\sB)$
if and only if $\Ac^\bctr(\sE)=\Ac(\sE)$; \par
\textup{(d)} if $\sA$ has finite homological dimension and there are
enough projective objects in $\sE$, then the functor $\Hot(\sB_\proj)
\rarrow\sD^\bctr(\sB)$ is an equivalence if and only if the functor
$\Hot(\sE_\proj)\rarrow\sD^\bctr(\sE)$ is an equivalence;  \par
\textup{(e)} if $\sB$ has finite homological dimension and there are
enough injective objects in $\sE$, then $\Ac^\bco(\sA)=\Ac(\sA)$
if and only if $\Ac^\bco(\sE)=\Ac(\sE)$; \par
\textup{(f)} if $\sB$ has finite homological dimension and there are
enough injective objects in $\sE$, then the functor $\Hot(\sA^\inj)
\rarrow\sD^\bco(\sA)$ is an equivalence if and only if the functor
$\Hot(\sE^\inj)\rarrow\sD^\bco(\sE)$ is an equivalence; \par
\textup{(g)} if $\sB$ has finite homological dimension, then
$\Ac^\bctr(\sA)=\Ac(\sA)$ if and only if $\Ac^\bctr(\sE)=\Ac(\sE)$; \par
\textup{(h)} if $\sB$ has finite homological dimension, then the functor
$\Hot(\sA_\proj)\rarrow\sD^\bctr(\sA)$ is an equivalence if and only if
the functor $\Hot(\sE_\proj)\rarrow\sD^\bctr(\sE)$ is an equivalence.
\end{rem}

\begin{lem} \label{ac=ac-bctr-equivalent-conditions}
 Let\/ $\sK$ be a Grothendieck category, $(\VF,\CA)$ be a cotorsion pair
of the very flat type in\/ $\sK$, and\/ $\Flat\subset\sK$ be a resolving
full subcategory closed under direct summands in\/ $\sK$ such that\/
$\VF\subset\Flat$.
 Then the following two conditions are equivalent:
\begin{enumerate}
\item for any given complex $F^\bu$ in\/ $\Flat\cap\CA$ such that
$F^\bu$ is acyclic in\/ $\sK$, the objects of cocycles of $F^\bu$ belong
to\/ $\Flat$ if and only if, for every complex $P^\bu$ in\/ $\VF$,
every morphism of complexes $P^\bu\rarrow F^\bu$ is homotopic to zero;
\item the classes of acyclic and Becker-contraacyclic complexes in
the exact category\/ $\Flat\cap\CA$ coincide, i.~e.,
$\Ac^\bctr(\Flat\cap\CA)=\Ac(\Flat\cap\CA)$.
\end{enumerate}
\end{lem}

\begin{proof}
 There are three observations underlying the assertion of the lemma:
\begin{enumerate}
\renewcommand{\theenumi}{\Alph{enumi}}
\item A complex $F^\bu$ in $\Flat\cap\CA$ is acyclic in
$\Flat\cap\CA$ if and only if $F^\bu$ is acyclic in $\sK$ with
the objects of cocycles belonging to~$\Flat$.
 The point is that any complex in $\CA$ that is acyclic in $\sK$ is
also acyclic in $\CA$ (i.~e., has the objects of cocycles belonging
to~$\CA$), by Lemma~\ref{finite-coresol-dim-periodicity}.
\item A complex $F^\bu$ in $\Flat\cap\CA$ is Becker-contraacyclic
in $\Flat\cap\CA$ if and only if, for every complex $P^\bu$ in $\VF$,
every morphism of complexes $P^\bu\rarrow F^\bu$ is homotopic to zero.
 Indeed, by Proposition~\ref{all-contraacyclic-cotorsion-pair-prop} and
Lemma~\ref{Ext1-as-homotopy-Hom-lemma}, a complex $K^\bu$ in $\CA$ is
Becker-contraacyclic in $\CA$ if and only if, for every complex $P^\bu$
in $\VF$, every morphism of complexes $P^\bu\rarrow K^\bu$ is homotopic
to zero.
 It remains to point out that $\Flat\cap\CA$ is a resolving full
subcategory closed under direct summands in $\CA$, hence the classes of
projective objects in $\CA$ and $\Flat\cap\CA$ coincide by
Lemma~\ref{proj-resolving-inj-coresolving}(a), and a complex in
$\Flat\cap\CA$ is Becker-contraacyclic in $\Flat\cap\CA$ if and only
if it is Becker-contraacyclic in~$\CA$.
\item Every Becker-contraacyclic complex in $\Flat\cap\CA$ is
acyclic in~$\sK$.
 Indeed, every Becker-contraacyclic complex in $\CA$ is acyclic in
$\CA$ by Corollary~\ref{contraacyclic-in-B-are-acyclic-in-B}.
\end{enumerate}

 Having collected the observations above, let us spell out the proof
of the promised equivalence (1)~$\Longleftrightarrow$~(2).

 (1)~$\Longrightarrow$~(2) Let $F^\bu$ be an acyclic complex in
the exact category $\Flat\cap\CA$.
 Then, in particular, $F^\bu$ is an acyclic complex in $\sK$ with
the objects of cocycles belonging to $\Flat$.
 By~(1), it follows that, for every complex $P^\bu$ in $\VF$,
every morphism of complexes $P^\bu\rarrow F^\bu$ is homotopic to zero.
 Using~(B), we see that the complex $F^\bu$ is Becker-contraacyclic
in $\Flat\cap\CA$.

 Conversely, let $F^\bu$ be a Becker-contraacyclic complex in
the exact category $\Flat\cap\CA$.
 By~(B), it follows that, for any complex $P^\bu$ in $\VF$, every
morphism of complexes $P^\bu\rarrow F^\bu$ is homotopic to zero.
 By~(C), the complex $F^\bu$ is acyclic in~$\sK$.
 By~(1), it follows that the objects of cocycles of $F^\bu$ belong
to $\Flat$.
 Using~(A), we can conclude that the complex $F^\bu$ is acyclic
in $\Flat\cap\CA$.

 (2)~$\Longrightarrow$~(1) Let $F^\bu$ be a complex in $\Flat\cap\CA$
such that $F^\bu$ is acyclic in $\sK$ with the objects of cocycles
belonging to $\Flat$.
 By~(A), it follows that the complex $F^\bu$ is acyclic in
$\Flat\cap\CA$.
 By~(2), the complex $F^\bu$ is Becker-contraacyclic in $\Flat\cap\CA$.
 Using~(B), we see that, for every complex $P^\bu$ in $\VF$, every
morphism of complexes $P^\bu\rarrow F^\bu$ is homotopic to zero.

 Conversely, let $F^\bu$ be a complex in $\Flat\cap\CA$ such that
$F^\bu$ is acyclic in $\sK$ and, for every complex $P^\bu$ in $\VF$,
every morphism of complexes $P^\bu\rarrow F^\bu$ is homotopic to zero.
 By~(B), it follows that the complex $F^\bu$ is Becker-contraacyclic
in $\Flat\cap\CA$.
 By~(2), the complex $F^\bu$ is acyclic in $\Flat\cap\CA$.
 In particular, it follows that the objects of cocycles of $F^\bu$
belong to $\Flat$.
\end{proof}

\begin{lem} \label{ac=ac-bco-equivalent-conditions}
 Let\/ $\sK$ be a Grothendieck category, $(\VF,\CA)$ be a cotorsion pair
of the very flat type in\/ $\sK$, and $(\Flat,\Cot)$ be a hereditary
complete cotorsion pair generated by a set of objects in\/ $\sK$
such that $(\VF,\CA)\le(\Flat,\Cot)$.
 Then the following six conditions are equivalent:
\begin{enumerate}
\item any complex in\/ $\Cot$ that is acyclic in\/ $\sK$ is also
acyclic in\/~$\Cot$;
\item any complex in\/ $\Cot$ that is acyclic in\/ $\CA$ is also
acyclic in\/~$\Cot$;
\item every acyclic complex in the exact category\/ $\Flat$ is
Becker-coacyclic in\/ $\Flat$, i.~e.,
$\Ac(\Flat)\subset\Ac^\bco(\Flat)$;
\item the classes of acyclic and Becker-coacyclic complexes in
the exact category\/ $\Flat$ coincide, i.~e.,
$\Ac^\bco(\Flat)=\Ac(\Flat)$;
\item every acyclic complex in the exact category\/ $\Flat\cap\CA$ is
Becker-coacyclic in\/ $\Flat\cap\CA$, i.~e.,
$\Ac(\Flat\cap\CA)\subset\Ac^\bco(\Flat\cap\CA)$;
\item the classes of acyclic and Becker-coacyclic complexes in
the exact category\/ $\Flat\cap\CA$ coincide, i.~e.,
$\Ac^\bco(\Flat\cap\CA)=\Ac(\Flat\cap\CA)$.
\end{enumerate}
\end{lem}

\begin{proof}
 (1)~$\Longleftrightarrow$~(2)
 Every complex in $\CA$ that is acyclic in $\sK$ is also acyclic
in $\CA$ by Lemma~\ref{finite-coresol-dim-periodicity}.

 (1)~$\Longleftrightarrow$~(3)~$\Longleftrightarrow$~(4)
 This is a part of
Theorem~\ref{cotorsion-periodity-type-equivalent-conditions}.

 (4)~$\Longleftrightarrow$~(6)
 By Lemma~\ref{restricting-cotorsion-pair-lemma}(a), the cotorsion pair
$(\VF,\CA)$ restricts to the full subcategory $\sE=\Flat\subset\sK$,
so we have a hereditary complete cotorsion pair $(\VF,\>\Flat\cap\CA)$
in the exact category~$\Flat$.
 Now it remains to refer to
Remark~\ref{flat-type-equivalence-for-cotorsion-pair-remark}(a).

 (3)~$\Longleftrightarrow$~(5)
 This is similar to (4)~$\Leftrightarrow$~(6).
 The point is that, in the context of
Remark~\ref{flat-type-equivalence-for-cotorsion-pair-remark}, if
$\sA$ has finite homological dimension, then $\Ac(\sB)\subset
\Ac^\bco(\sB)$ if and only if $\Ac(\sE)\subset\Ac^\bco(\sE)$.
 The argument from the proof of
Proposition~\ref{flat-type-equivalence-for-cotorsion-pair-prop} applies.
\end{proof}

\begin{lem} \label{majorating-very-flat-type-pair}
 Let\/ $\sK$ be a Grothendieck category, and let $(\VF_1,\CA_1)$ and
$(\VF_2,\CA_2)$ be two cotorsion pairs of the very flat type in\/~$\sK$.
 Let $(\Flat,\Cot)$ be a cotorsion pair in\/ $\sK$ such that
$(\VF_1,\CA_1)\le(\Flat,\Cot)$ and $(\VF_2,\CA_2)\le(\Flat,\Cot)$.
 Then there exists a cotorsion pair $(\VF_3,\CA_3)$ of the very flat
type in\/ $\sK$ such that $(\VF_1,\CA_1)\le(\VF_3,\CA_3)$ and
$(\VF_2,\CA_2)\le(\VF_3,\CA_3)\le(\Flat,\Cot)$.
\end{lem}

\begin{proof}
 Let $\sS_1\subset\sK$ be a set of objects generating the cotorsion pair
$(\VF_1,\CA_1)$ and $\sS_2\subset\sK$ be a set of objects generating
the cotorsion pair $(\VF_2,\CA_2)$.
 Put $\sS_3=\sS_1\cup\sS_2$ and let $(\VF_3,\CA_3)$ be the cotorsion
pair generated by the set $\sS_3$ in\/~$\sK$.
 Then we have $\VF_1\subset\VF_3$ and $\VF_2\subset\VF_3\subset\Flat$;
so $(\VF_1,\CA_1)\le(\VF_3,\CA_3)$ and $(\VF_2,\CA_2)\le(\VF_3,\CA_3)
\le(\Flat,\Cot)$.
 Moreover, by the definition, $\CA_3=\CA_1\cap\CA_2$.

 In particular, it follows that the class $\VF_3$ is generating in~$\sK$
(since the classes $\VF_1$ and/or $\VF_2$ are).
 The class $\CA_3$ contains all the injective objects of $\sK$, so it
is cogenerating in~$\sK$.
 By Theorem~\ref{loc-pres-eklof-trlifaj}(a)
or~\ref{efficient-eklof-trlifaj}(a), we can conclude that the cotorsion
pair $(\VF_3,\CA_3)$ is complete.
 The class $\CA_3=\CA_1\cap\CA_2$ is closed under cokernels of
monomorphisms in $\sK$, since both the classes $\CA_1$ and $\CA_2$ are;
hence the cotorsion pair $(\VF_3,\CA_3)$ is hereditary.

 Finally, let $d_1$ and~$d_2$ be two integers such that every object
of $\sK$ has coresolution dimension at most~$d_1$ with respect to
$\CA_1$ and at most~$d_2$ with respect to~$\CA_2$.
 Let $d$~be an integer such that $d_1\le d$ and $d_2\le d$.
 Given an object $K\in\sK$, pick an exact sequence $0\rarrow K\rarrow
J^0\rarrow J^1\rarrow\dotsb\rarrow J^{d-1}\rarrow C\rarrow0$ in $\sK$
with injective objects $J^0$,~\dots,~$J^{d-1}$.
 Then we have $C\in\CA_1$ and $C\in\CA_2$.
 Thus $C\in\CA_3$, and the coresolution dimension of $K$ with respect
to $\CA_3$ does not exceed~$d$.
\end{proof}

\begin{prop} \label{independent-of-very-flat-prop}
 Let\/ $\sK$ be a Grothendieck category, and let $(\VF_1,\CA_1)$ and
$(\VF_2,\CA_2)$ be two cotorsion pairs of the very flat type in\/~$\sK$.
 Let $(\Flat,\Cot)$ be a hereditary cotorsion pair in\/ $\sK$ such that
$(\VF_1,\CA_1)\le(\Flat,\Cot)$ and $(\VF_2,\CA_2)\le(\Flat,\Cot)$.
 Then the equivalent conditions of
Lemma~\ref{ac=ac-bctr-equivalent-conditions} are satisfied for
$(\VF_1,\CA_1)$ and\/ $\Flat$ if and only if they are satisfied for
$(\VF_2,\CA_2)$ and\/ $\Flat$.  \emergencystretch=1em
\end{prop}

\begin{proof}
 In view of Lemma~\ref{majorating-very-flat-type-pair}, without loss of
generality we can assume that $(\VF_1,\CA_1)\le(\VF_2,\CA_2)$.
 So we have $\VF_1\subset\VF_2\subset\Flat$ and $\CA_1\supset\CA_2$.
 By Lemma~\ref{restricting-cotorsion-pair-lemma}(b), the cotorsion pair
$(\VF_2,\CA_2)$ restricts to a hereditary complete cotorsion pair
$(\VF_2\cap\CA_1,\>\CA_2)$ in the exact category~$\CA_1$.
 Furthermore, the full subcategory $\Flat\cap\CA_1$ is closed under
extensions and kernels of admissible epimorphisms in~$\CA_1$.
 By Lemma~\ref{restricting-cotorsion-pair-lemma}(a), the cotorsion pair
$(\VF_2\cap\CA_1,\>\CA_2)$ in $\CA_1$ restricts to a hereditary
complete cotorsion pair $(\VF_2\cap\CA_1,\>\Flat\cap\CA_2)$ in
the exact category $\sE=\Flat\cap\CA_1$.
 Put $\sA=\VF_2\cap\CA_1$ and $\sB=\Flat\cap\CA_2$.

 The full subcategory $\VF_2\cap\CA_1$ is coresolving in~$\VF_2$.
 By Lemma~\ref{co-resolving-Ext-agrees}, it follows that the $\Ext$
groups computed in $\VF_2\cap\CA_1$ agree with those in~$\VF_2$.
 Since the exact category $\VF_2$ has finite homological dimension,
so does the exact category $\sA=\VF_2\cap\CA_1$.

 By Lemma~\ref{restricting-cotorsion-pair-lemma}(a), the cotorsion
pair $(\VF_1,\CA_1)$ restricts to a hereditary complete cotorsion pair
$(\VF_1,\>\Flat\cap\CA_1)$ in the exact category~$\Flat$.
 By Lemma~\ref{inj-proj-objects-in-left-right-classes}(b), it follows
that there are enough projective objects in the exact category
$\Flat\cap\CA_1$ (cf.\ the first paragraph of the proof of the next
Corollary~\ref{flat-type-definition-cor}).
 Thus Remark~\ref{flat-type-equivalence-for-cotorsion-pair-remark}(c)
is applicable, and we can conclude that $\Ac^\bctr(\sB)=\Ac(\sB)$
if and only if $\Ac^\bctr(\sE)=\Ac(\sE)$.
 This means that $\Ac^\bctr(\Flat\cap\CA_1)=\Ac(\Flat\cap\CA_1)$ if and
only if $\Ac^\bctr(\Flat\cap\CA_2)=\Ac(\Flat\cap\CA_2)$, as desired.
\end{proof}

\begin{cor} \label{flat-type-definition-cor}
 Let\/ $\sK$ be a Grothendieck category, $(\VF,\CA)$ be a cotorsion pair
of the very flat type in\/ $\sK$, and $(\Flat,\Cot)$ be a hereditary
complete cotorsion pair generated by a set of objects in\/ $\sK$
such that $(\VF,\CA)\le(\Flat,\Cot)$.
 Then the following three conditions are equivalent:
\begin{enumerate}
\item the exact category\/ $\Flat\cap\CA$ is of the flat type;
\item the exact category\/ $\Flat\cap\CA$ is weakly of the flat type;
\item both of the following two conditions hold:
\begin{enumerate}
\renewcommand{\theenumii}{\Roman{enumii}}
\item for any given complex $F^\bu$ in\/ $\Flat\cap\CA$ such that
$F^\bu$ is acyclic in\/ $\sK$, the objects of cocycles of $F^\bu$ belong
to\/ $\Flat$ if and only if, for every complex $P^\bu$ in\/ $\VF$,
every morphism of complexes $P^\bu\rarrow F^\bu$ is homotopic to zero;
\item any complex in\/ $\Cot$ that is acyclic in\/ $\sK$ is also
acyclic in\/~$\Cot$.
\end{enumerate}
\end{enumerate}
\end{cor}

\begin{proof}
 First of all, we need to point out that, under the assumptions of
the corollary, there are always enough injective and projective objects
in $\Flat\cap\CA$.
 Concerning the injectives, the assertion is provided by
Lemma~\ref{inj-proj-objects-in-intersected-classes}(a) applied
to the cotorsion pair $(\sA,\sB)=(\Flat,\Cot)$ and the full subcategory
$\sE=\CA$ in~$\sK$.
 Concerning the projectives, this was already mentioned in the last
paragraph of the proof of
Proposition~\ref{independent-of-very-flat-prop}.
 Alternatively, one can refer to
Lemma~\ref{inj-proj-objects-in-intersected-classes}(b) applied
to the cotorsion pair $(\sA,\sB)=(\VF,\CA)$ and the full subcategory
$\sE=\Flat$ in~$\sK$.

 (1)~$\Longleftrightarrow$~(2)
 Put $\sF=\Flat\cap\CA$.
 The point is that, under the assumptions of the corollary,
the functors $\Hot(\sF^\inj)\rarrow\sD^\bco(\sF)$ and
$\Hot(\sF_\proj)\rarrow\sD^\bctr(\sF)$ are always triangulated
equivalences.
 Indeed, the first assertion is provided by
Corollary~\ref{coderived-of-E-intersect-A-well-behaved}
for the cotorsion pair $(\sA,\sB)=(\Flat,\Cot)$ and the full
subcategory $\sE=\CA$ in~$\sK$.
 The second assertion is provided by
Corollary~\ref{contraderived-of-E-intersect-B-well-behaved}
for the cotorsion pair $(\sA,\sB)=(\VF,\CA)$ and the full
subcategory $\sE=\Flat$ in~$\sK$.

 (2)~$\Longleftrightarrow$~(3)
 Combine Lemma~\ref{ac=ac-bctr-equivalent-conditions} with
Lemma~\ref{ac=ac-bco-equivalent-conditions}\,%
(1)\,$\Leftrightarrow$\,(6).
\end{proof}

 Now we can present the main definition of this section.
 Let $\sK$ be a Grothendieck category and $(\Flat,\Cot)$ be a hereditary
complete cotorsion pair generated by a set of objects in~$\sK$.
 Then we say that the cotorsion pair $(\Flat,\Cot)$ in $\sK$ is
\emph{of the flat type} if there exists a cotorsion pair $(\VF,\CA)$
of the very flat type in $\sK$ such that $(\VF,\CA)\le(\Flat,\Cot)$ and
any one of the equivalent conditions of
Corollary~\ref{flat-type-definition-cor} is satisfied for
$(\VF,\CA)$ and $(\Flat,\Cot)$.
 In other words, the cotorsion pair $(\Flat,\Cot)$ is of the flat type
if and only if it satisfies any one of the equivalent conditions of
Lemma~\ref{ac=ac-bctr-equivalent-conditions} \emph{and also}
satisfies any one of the equivalent conditions of
Lemma~\ref{ac=ac-bco-equivalent-conditions}.

 By Proposition~\ref{independent-of-very-flat-prop}, if there exist
two cotorsion pairs $(\VF_1,\CA_1)$ and $(\VF_2,\CA_2)$ of the very
flat type in $\sK$ such that $(\VF_1,\CA_1)\le(\Flat,\Cot)$ and
$(\VF_2,\CA_2)\le(\Flat,\Cot)$, then the equivalent conditions of
Corollary~\ref{flat-type-definition-cor} are satisfied for
$(\VF_1,\CA_1)$ and $(\Flat,\Cot)$ if and only if they are satisfied
for $(\VF_2,\CA_2)$ and $(\Flat,\Cot)$.
 In this sense, the property of a cotorsion pair $(\Flat,\Cot)$ in $\sK$
to be of the flat type \emph{does not depend on the choice of}
an auxiliary cotorsion pair $(\VF,\CA)$ of the very flat type, though
it \emph{requires existence} of an auxiliary cotorsion pair of
the very flat type.

 The following lemma is convenient to use.

\begin{lem} \label{flat-type-convenient-restatement-lemma}
 Let\/ $\sK$ be a Grothendieck category, $(\VF,\CA)$ be a cotorsion pair
of the very flat type in\/ $\sK$, and $(\Flat,\Cot)$ be a hereditary
complete cotorsion pair generated by a set of objects in\/ $\sK$
such that $(\VF,\CA)\le(\Flat,\Cot)$.
 Then the cotorsion pair $(\Flat,\Cot)$ is of the flat type if and only
if one has
$$
 \Com(\CA)\cap\Ac^\bco(\Flat)=\Ac(\Flat\cap\CA)=
 \Com(\Flat)\cap\Ac^\bctr(\CA).
$$
\end{lem}

\begin{proof}
 The point is that, as it was mentioned already in the proof of
Lemma~\ref{ac=ac-bctr-equivalent-conditions}, a complex in
$\Flat\cap\CA$ is Becker-contraacyclic in $\CA$ if and only if it is
Becker-contraacyclic in $\Flat\cap\CA$.
 To repeat, this holds because the full subcategory of projective
objects in $\CA$ is $\VF\cap\CA$ by
Lemma~\ref{inj-proj-objects-in-left-right-classes}(b), and the full
subcategory of projective objects in $\Flat\cap\CA$ is also
$\VF\cap\CA$ by Lemma~\ref{inj-proj-objects-in-intersected-classes}(b)
(applied to the cotorsion pair $(\sA,\sB)=(\VF,\CA)$ and the full
subcategory $\sE=\Flat\subset\sK$).

 Similarly, a complex in $\Flat\cap\CA$ is Becker-coacyclic in $\Flat$
if and only if it is Becker-coacyclic in $\Flat\cap\CA$.
 This holds because the full subcategory of injective objects in $\Flat$
is $\Flat\cap\Cot$ by
Lemma~\ref{inj-proj-objects-in-left-right-classes}(a), and the full
subcategory of injective objects in $\Flat\cap\CA$ is also
$\Flat\cap\Cot$
by Lemma~\ref{inj-proj-objects-in-intersected-classes}(a)
(applied to the cotorsion pair $(\sA,\sB)=(\Flat,\Cot)$ and the full
subcategory $\sE=\CA\subset\sK$).

 In view of these observations, the assertion of the lemma is
a restatement of the conditions of
Lemmas~\ref{ac=ac-bctr-equivalent-conditions}(2)
and~\ref{ac=ac-bco-equivalent-conditions}(6), or of
Corollary~\ref{flat-type-definition-cor}(1) or~(2).
\end{proof}

 For examples of flat-type cotorsion pairs, see the end of the next
Section~\ref{triang-equiv-secn}.

\Section{Triangulated Equivalences for a Flat-Type Cotorsion Pair}
\label{triang-equiv-secn}

 In this section we consider a Grothendieck category $\sK$ with
a cotorsion pair $(\VF,\CA)$ of the very flat type and a cotorsion pair
$(\Flat,\Cot)$ of the flat type such that $(\VF,\CA)\le(\Flat,\Cot)$.
 The main results of the section claim that various kinds of
derived categories associated with $(\Flat,\Cot)$ are equivalent to
the respective derived categories for $(\VF,\CA)$.

 The following lemma about triangulated quotient categories is very
basic, but it will play an important role in
Sections~\ref{triang-equiv-secn}\+-\ref{sandwiched-secn}.

\begin{lem} \label{triangulated-quotient-lemma}
 Let\/ $\sH$ be a triangulated category and\/ $\sX\subset\sH$ be
a strictly full triangulated subcategory. \par
\textup{(a)} Let\/ $\sE\subset\sH$ be another full triangulated
subcategory.
 Assume that, for every object $H\in\sH$, there exists an object
$E\in\sE$ together with a morphism $E\rarrow H$ in\/ $\sH$ whose
cone belongs to\/~$\sX$.
 Then the inclusion of triangulated categories\/ $\sE\rarrow\sH$
induces an equivalence of the triangulated Verdier quotient categories
$$
 \sE/(\sE\cap\sX)\simeq\sH/\sX.
$$ \par
\textup{(b)} Let\/ $\sG\subset\sH$ be another full triangulated
subcategory.
 Assume that, for every object $H\in\sH$, there exists an object
$G\in\sG$ together with a morphism $H\rarrow G$ in\/ $\sH$ whose
cone belongs to\/~$\sX$.
 Then the inclusion of triangulated categories\/ $\sG\rarrow\sH$
induces an equivalence of the triangulated Verdier quotient categories
$$
 \sG/(\sG\cap\sX)\simeq\sH/\sX.
$$
\end{lem}

\begin{proof}
 This is~\cite[Corollary~10.2.7]{KS} or~\cite[Lemma~1.6]{Pkoszul}.
\end{proof}

\begin{thm} \label{flat-type-conventional-derived-equivalence-theorem}
 Let\/ $\sK$ be a Grothendieck category, $(\VF,\CA)$ be a cotorsion pair
of the very flat type in\/ $\sK$, and $(\Flat,\Cot)$ be a cotorsion pair
of the flat type such that $(\VF,\CA)\le(\Flat,\Cot)$.
 Then the inclusions of exact/abelian categories\/ $\Cot\rarrow\CA
\rarrow\sK$ induce equivalences of the conventional derived categories
$$
 \sD(\Cot)\simeq\sD(\CA)\simeq\sD(\sK).
$$
\end{thm}

\begin{proof}
 The functor $\sD(\CA)\rarrow\sD(\sK)$ is a triangulated equivalence by
Lemma~\ref{very-flat-type-conventional-derived-equivalence}.
 The functor $\sD(\Cot)\rarrow\sD(\sK)$ is a triangulated equivalence
whenever any one of the equivalent conditions of
Lemma~\ref{ac=ac-bco-equivalent-conditions}, or more specifically
the condition of Lemma~\ref{ac=ac-bco-equivalent-conditions}(1),
is satisfied.
 This is the result of~\cite[Proposition~10.6]{PS6}.
\end{proof}

 The following lemma illustrates the point that the equivalent
conditions of Lemma~\ref{ac=ac-bctr-equivalent-conditions} represent
a strong version of the ``flat and very flat periodicity property''
(while Lemma~\ref{ac=ac-bco-equivalent-conditions} is concerned with
the ``cotorsion periodicity property'').

\begin{lem} \label{flat-veryflat-periodicity}
 Let\/ $\sK$ be a Grothendieck category, $(\VF,\CA)$ be a cotorsion pair
of the very flat type in\/ $\sK$, and\/ $\Flat\subset\sK$ be a resolving
full subcategory closed under direct summands in\/ $\sK$ such that\/
$\VF\subset\Flat$.
 Assume that any one of the equivalent conditions of
Lemma~\ref{ac=ac-bctr-equivalent-conditions} is satisfied.
 Then any complex in\/ $\VF$ that is acyclic in\/ $\Flat$ is
acyclic in\/~$\VF$.
\end{lem}

\begin{proof}
 The proof of this lemma uses the inclusion
$\Ac(\Flat\cap\CA)\subset\Ac^\bctr(\Flat\cap\CA)$ from
Lemma~\ref{ac=ac-bctr-equivalent-conditions}(2).

 By Lemma~\ref{inj-proj-objects-in-left-right-classes}(a),
there are enough injective objects in the exact category $\VF$,
and $\VF^\inj=\VF\cap\CA$.
 It follows that for any complex $G^\bu$ in $\VF$ there exists
a contractible complex $C^\bu$ in $\VF\cap\CA$ together with
a termwise admissible monomorphism $G^\bu\rarrow C^\bu$ of
complexes in~$\VF$.
 So we have a termwise short exact sequence of complexes
$0\rarrow G^\bu\rarrow C^\bu\rarrow G^{1,\bu}\rarrow0$
such that $G^{1,\bu}$ is again a complex in~$\VF$.
 Applying the same construction to the complex $G^{1,\bu}$ in $\VF$
and proceeding in this way, we obtain a termwise admissible exact
sequence $0\rarrow G^\bu\rarrow C^{0,\bu}\rarrow C^{1,\bu}
\rarrow\dotsb\rarrow C^{d-1,\bu}\rarrow Q^\bu\rarrow0$ of complexes
in $\VF$, where $C^{i,\bu}$ is a contractible complex in $\VF\cap\CA$
for every $0\le i\le d-1$.
 Here we take $d$~to be an integer such that the homological dimension
of $\VF$ does not exceed~$d$.
 Then the coresolution dimensions of all the objects of $\sK$ with
respect to $\CA$ also do not exceed~$d$ by
Lemma~\ref{very-flat-type-cotorsion-pair-lemma}, so $Q^\bu$ is
actually a complex in $\VF\cap\CA$.

 Now if the complex $G^\bu$ is acyclic in $\Flat$, then so is
the complex~$Q^\bu$ (by~\cite[Corollary~3.6 and the dual version
of Proposition~7.6]{Bueh}).
 Applying the condition of
Lemma~\ref{ac=ac-bctr-equivalent-conditions}(1) to the complex
$P^\bu=F^\bu=Q^\bu$, we see that the complex $Q^\bu$ is contractible.
 Alternatively, one can say that $Q^\bu$ is an acyclic complex of
projective objects in $\Flat\cap\CA$, so if $Q^\bu$ is
Becker-contraacyclic by
Lemma~\ref{ac=ac-bctr-equivalent-conditions}(2), then it follows
that $Q^\bu$ is contractible.
 Thus the complex $Q^\bu$ is acyclic in~$\VF$.
 Using~\cite[Corollary~3.6]{Bueh} again, we conclude that
the complex $G^\bu$ is acyclic in~$\VF$.
\end{proof}

\begin{lem} \label{contraacyclic-in-Cot-same-as-in-CA}
 Let\/ $\sK$ be a Grothendieck category, $(\VF,\CA)$ be a cotorsion pair
of the very flat type in\/ $\sK$, and $(\Flat,\Cot)$ be a cotorsion pair
of the flat type such that $(\VF,\CA)\le(\Flat,\Cot)$.
 Then a complex in\/ $\Cot$ is Becker-contraacyclic in\/ $\CA$ if and
only if it is Becker-contraacyclic in\/ $\Cot$, that is,
$\Ac^\bctr(\Cot)=\Com(\Cot)\cap\Ac^\bctr(\CA)$.
\end{lem}

\begin{proof}
 The proof of this lemma uses the inclusion
$\Ac^\bctr(\Flat\cap\CA)\subset\Ac(\Flat\cap\CA)$ from
Lemma~\ref{ac=ac-bctr-equivalent-conditions}(2) and the inclusion
$\Ac(\Flat\cap\CA)\subset\Ac^\bco(\Flat\cap\CA)$ from
Lemma~\ref{ac=ac-bco-equivalent-conditions}(5).

 By Lemma~\ref{inj-proj-objects-in-left-right-classes}(b), the class
of projective objects in $\CA$ coincides with $\VF\cap\CA$, while
the class of projective objects in $\Cot$ coincides with
$\Flat\cap\Cot$.
 Let $C^\bu$ be a complex in~$\Cot$.
 We need to prove that the complex $\Hom_\sK(P^\bu,C^\bu)$ is acyclic
for all complexes $P^\bu$ in $\VF\cap\CA$ if and only if the complex
$\Hom_\sK(F^\bu,C^\bu)$ is acyclic for all complexes $F^\bu$ in
$\Flat\cap\Cot$.
 We will show that both the conditions are equivalent to the complex
$\Hom_\sK(G^\bu,C^\bu)$ being acyclic for all complexes $G^\bu$ in
$\Flat\cap\CA$.

 Notice that $\VF\cap\CA$ is also the class of all projective objects
in $\Flat\cap\CA$
(by Lemma~\ref{inj-proj-objects-in-intersected-classes}(b) applied
to the cotorsion pair $(\sA,\sB)=(\VF,\CA)$ and the full subcategory
$\sE=\Flat$ in~$\sK$).
 Similarly, $\Flat\cap\Cot$ is the class of all injective objects in
$\Flat\cap\CA$
(Lemma~\ref{inj-proj-objects-in-intersected-classes}(a) applied
to the cotorsion pair $(\sA,\sB)=(\Flat,\Cot)$ and the full subcategory
$\sE=\CA$ in~$\sK$).
 Following the proof of Corollary~\ref{flat-type-definition-cor},
for any complex $G^\bu$ in $\Flat\cap\CA$ there exists a complex
$F^\bu$ in $\Flat\cap\Cot$ together with a morphism of complexes
$G^\bu\rarrow F^\bu$ with the cone belonging to
$\Ac^\bco(\Flat\cap\CA)$.
 Similarly, for any complex $G^\bu$ in $\Flat\cap\CA$ there exists
a complex $P^\bu$ in $\VF\cap\CA$ together with a morphism of complexes
$P^\bu\rarrow G^\bu$ with the cone belonging to
$\Ac^\bctr(\Flat\cap\CA)$.

 Now we have $\Ac^\bctr(\Flat\cap\CA)\subset\Ac^\bco(\Flat\cap\CA)$,
as mentioned in the first paragraph of this proof.
 Furthermore, a complex in $\Flat\cap\CA$ is Becker-coacyclic in
$\Flat$ if and only if it is Becker-coacyclic in $\Flat\cap\CA$,
and we have $\Ac^\bco(\Flat\cap\CA)\subset\Ac^\bco(\Flat)$, as per
the proof of Lemma~\ref{flat-type-convenient-restatement-lemma}.
 It remains to recall that, for any complex
$A^\bu\in\Ac^\bco(\Flat)$ and any complex $C^\bu\in\Com(\Cot)$,
the complex $\Hom_\sK(A^\bu,C^\bu)$ is acyclic by
Proposition~\ref{coacyclic-all-cotorsion-pair-prop} and
Lemma~\ref{Ext1-as-homotopy-Hom-lemma}.
\end{proof}

\begin{thm} \label{flat-type-contraderived-equivalences-theorem}
 Let\/ $\sK$ be a Grothendieck category, $(\VF,\CA)$ be a cotorsion pair
of the very flat type in\/ $\sK$, and $(\Flat,\Cot)$ be a cotorsion pair
of the flat type such that $(\VF,\CA)\le(\Flat,\Cot)$.
 Then \par
\textup{(a)} the inclusion of exact categories\/ $\VF\rarrow\Flat$
induces an equivalence of the conventional derived categories
$$
 \sD(\VF)\simeq\sD(\Flat);
$$ \par
\textup{(b)} the inclusion of exact categories\/ $\VF\rarrow\Flat$
induces an equivalence of the Becker coderived categories
$$
 \sD^\bco(\VF)\simeq\sD^\bco(\Flat);
$$ \par
\textup{(c)} the inclusion of exact categories\/ $\Cot\rarrow\CA$
induces an equivalence of the Becker contraderived categories
$$
 \sD^\bctr(\Cot)\simeq\sD^\bctr(\CA).
$$
\end{thm}

\begin{proof}
 First of all, let us explain why the induced triangulated functors
are well-defined.
 Any exact functor of exact categories induces a triangulated functor
between their derived categories; hence the triangulated functor
$\sD(\VF)\rarrow\sD(\Flat)$ in part~(a).
 The induced triangulated functor between the Becker coderived
categories $\sD^\bco(\VF)\rarrow\sD^\bco(\Flat)$ in part~(b) exists by
Corollary~\ref{coderived-A1-to-A2-well-defined}.
 The induced triangulated functor between the Becker contraderived
categories $\sD^\bctr(\Cot)\rarrow\sD^\bctr(\CA)$ in part~(c) exists by
Corollary~\ref{contraderived-B2-to-B1-well-defined}.

 Furthermore, it should be pointed out that the three assertions~(a\+-c)
are equivalent restatements of one another, and the triangulated
categories involved in them are all the same.
 Under the assumptions of the theorem we have
$\Ac^\bco(\VF)=\Ac(\VF)$ by
Lemmas~\ref{inj-proj-objects-in-left-right-classes}(a)
and~\ref{exact-category-of-finite-homol-dim}(a),
and $\Ac^\bco(\Flat)=\Ac(\Flat)$ by
Lemma~\ref{ac=ac-bco-equivalent-conditions}(4).
 Thus $\sD^\bco(\VF)=\sD(\VF)$ and $\sD^\bco(\Flat)=\sD(\Flat)$, so
$\sD(\VF)\rarrow\sD(\Flat)$ and $\sD^\bco(\VF)\rarrow\sD^\bco(\Flat)$
is one and the same functor.
 So parts~(a) and~(b) are equivalent restatements of each other.
 On the other hand, we have $\sD^\bco(\VF)\simeq\sD^\bctr(\CA)$
and $\sD^\bco(\Flat)\simeq\sD^\bctr(\Cot)$ by
Corollary~\ref{cotorsion-pair-co-contra-equivalence}.
 These two triangulated equivalences, viewed as identifications,
make the functor $\sD^\bco(\VF)\rarrow\sD^\bco(\Flat)$ left adjoint
to the functor $\sD^\bctr(\Cot)\rarrow\sD^\bctr(\CA)$ by
Theorem~\ref{nested-cotorsion-pairs-adjunction-theorem}.
 So proving that the functor in part~(b) is a triangulated equivalence
is equivalent to proving that the functor in part~(c) is a triangulated
equivalence.

 Now let us prove the assertions of parts~(a\+-b) and~(c) independently
of each other.
 In order to prove the triangulated equivalence~(a\+-b), we apply
Lemma~\ref{triangulated-quotient-lemma}(a) to the triangulated
category $\sH=\Hot(\Flat)$ with the full triangulated subcategories
$\sX=\Ac(\Flat)=\Ac^\bco(\Flat)$ and $\sE=\Hot(\VF)$.

 Let $F^\bu$ be a complex in~$\Flat$.
 By Proposition~\ref{all-contraacyclic-cotorsion-pair-prop}, we
have a hereditary complete cotorsion pair $(\Com(\VF),\>\Ac^\bctr(\CA))$
in the abelian category $\Com(\sK)$.
 Let $0\rarrow B^\bu\rarrow P^\bu\rarrow F^\bu\rarrow0$ be a special
precover exact sequence~\eqref{special-precover-sequence} for
the object $F^\bu\in\Com(\sK)$ with respect to this cotorsion pair.
 So we have $P^\bu\in\Com(\VF)$ and $B^\bu\in\Ac^\bctr(\CA)$.
 Since the class $\Flat$ is closed under kernels of epimorphisms in
$\sK$, we actually have $B^\bu\in\Com(\Flat)\cap\Ac^\bctr(\CA)$.

 By Lemma~\ref{flat-type-convenient-restatement-lemma}, it follows
that $B^\bu\in\Ac^\bco(\Flat)$.
 In view of Lemma~\ref{totalizations-are-co-contra-acyclic}(a), we
can conclude that the cone of the morphism $P^\bu\rarrow F^\bu$ is
Becker-coacyclic in~$\Flat$.
 So the assumption of Lemma~\ref{triangulated-quotient-lemma}(a)
is satisfied.
 It remains to recall that $\Com(\VF)\cap\Ac(\Flat)=\Ac(\VF)$ by
Lemma~\ref{flat-veryflat-periodicity} in order to finish the proof of
the assertion that the functor in parts~(a\+-b) is a triangulated
equivalence.

 In order to prove the triangulated equivalence~(c), we apply
Lemma~\ref{triangulated-quotient-lemma}(b) to the triangulated
category $\sH=\Hot(\CA)$ with the full triangulated subcategories
$\sX=\Ac^\bctr(\CA)$ and $\sG=\Hot(\Cot)$.

 Let $C^\bu$ be a complex in $\CA$.
 By Proposition~\ref{coacyclic-all-cotorsion-pair-prop}, we have
a hereditary complete cotorsion pair $(\Ac^\bco(\Flat),\>\Com(\Cot))$
in the abelian category $\Com(\sK)$.
 Let $0\rarrow C^\bu\rarrow D^\bu\rarrow A^\bu\rarrow0$ be a special
preenvelope exact sequence~\eqref{special-preenvelope-sequence} for
the object $C^\bu\in\Com(\sK)$ with respect to this cotorsion pair.
 So we have $D^\bu\in\Com(\Cot)$ and $A^\bu\in\Ac^\bco(\Flat)$.
 Since the class $\CA$ is closed under cokernels of monomorphisms in
$\sK$, we actually have $A^\bu\in\Com(\CA)\cap\Ac^\bco(\Flat)$.

 By Lemma~\ref{flat-type-convenient-restatement-lemma}, it follows
that $A^\bu\in\Ac^\bctr(\CA)$.
 In view of Lemma~\ref{totalizations-are-co-contra-acyclic}(a), this
means that the cone of the morphism $C^\bu\rarrow D^\bu$ is
Becker-contraacyclic in~$\CA$.
 So the assumption of Lemma~\ref{triangulated-quotient-lemma}(b)
is satisfied.
 It remains to recall that $\Com(\Cot)\cap\Ac^\bctr(\CA)=
\Ac^\bctr(\Cot)$ by Lemma~\ref{contraacyclic-in-Cot-same-as-in-CA}
in order to conclude that the functor in part~(c) is a triangulated
equivalence.
\end{proof}

\begin{cor} \label{veryflat-flat-type-square-of-triang-equivalences}
 Let\/ $\sK$ be a Grothendieck category, $(\VF,\CA)$ be a cotorsion pair
of the very flat type in\/ $\sK$, and $(\Flat,\Cot)$ be a cotorsion pair
of the flat type such that $(\VF,\CA)\le(\Flat,\Cot)$.
 Then there is a commutative square diagram of triangulated equivalences
$$
 \xymatrix{
  \sD^{\varnothing=\bco}(\VF) \ar@{=}[d] \ar@<2pt>[r]
  & \sD^{\varnothing=\bco}(\Flat) \ar@{=}[d] \ar@<2pt>@{-}[l] \\
  \sD^\bctr(\CA) \ar@<2pt>@{-}[r]
  & \sD^\bctr(\Cot) \ar@<2pt>[l]
 }
$$
 Here the notation\/ $\sD^{\varnothing=\bco}(\sE)$ for a derived
category\/ $\sE$ means that\/ $\Ac(\sE)=\Ac^\bco(\sE)$, so\/
$\sD(\sE)=\sD^\bco(\sE)$.
 The vertical triangulated equivalences are provided by
Corollary~\ref{cotorsion-pair-co-contra-equivalence}.
 The upper horizontal triangulated equivalence is provided by
Theorem~\ref{flat-type-contraderived-equivalences-theorem}(a\+-b).
 The lower horizontal triangulated equivalence is provided by
Theorem~\ref{flat-type-contraderived-equivalences-theorem}(c).
\end{cor}

\begin{proof}
 One has $\Ac^\bco(\VF)=\Ac(\VF)$ by
Lemma~\ref{exact-category-of-finite-homol-dim}(a)
and $\Ac^\bco(\Flat)=\Ac(\Flat)$ by
Lemma~\ref{ac=ac-bco-equivalent-conditions}(4).
 The square diagram is commutative by
Theorem~\ref{nested-cotorsion-pairs-adjunction-theorem}
(as two category equivalences are adjoint if and only if they are
mutually inverse).
\end{proof}

\begin{ex} \label{very-flat-type-is-flat-type-example}
 Any cotorsion pair of the very flat type is a cotorsion pair of
the flat type.
 Indeed, let $\sK$ be a Grothendieck category and $(\VF,\CA)$ be
a cotorsion pair of the very flat type in~$\sK$.
 Put $\Flat=\VF$ and $\Cot=\CA$.

 Then the exact category $\Flat\cap\CA$ is split (i.~e., all
the admissible short exact sequences in $\Flat\cap\CA$ are split), so
all the objects are both projective and injective in $\Flat\cap\CA$.
 Hence all the three classes $\Ac^\bctr(\Flat\cap\CA)$, \
$\Ac^\bco(\Flat\cap\CA)$, and $\Ac(\Flat\cap\CA)$ coincide with
the class of contractible complexes in $\Flat\cap\CA$.
 In other words, one can say that \emph{any split exact category is of
the flat type}.
 So the conditions of Corollary~\ref{flat-type-definition-cor}(1) or~(2)
are satisfied for $\Flat\cap\CA$ in this example.

 Thus all the examples of cotorsion pairs of the very flat type in
Examples~\ref{projective-cotorsion-pair-is-very-flat-type-ex}\+-%
\ref{flat-projdim1-modules-and-comodules-exs} can be viewed as
examples of cotorsion pairs of the flat type.
\end{ex}

\begin{ex} \label{flat-cotorsion-pair-in-modules-flat-type}
 Let $A$ be an associative ring.
 Consider the category of left $A$\+modules $\sK=A\Modl$.
 Following Example~\ref{projective-cotorsion-pair-is-very-flat-type-ex},
the two classes $\VF=A\Modl_\proj$ and $\CA=A\Modl$ form a cotorsion
pair $(\VF,\CA)$ of the very flat type in~$\sK$.

 A left $A$\+module $C$ is said to be \emph{cotorsion} (in the sense
of Enochs~\cite[Section~2]{En2}) if $\Ext^1_A(F,C)=0$ for all flat
left $A$\+modules~$F$.
 Let $\Flat=A\Modl_\flat$ be the class of flat left $A$\+modules
and $\Cot=A\Modl^\cot$ be the class of cotorsion left $A$\+modules.
 Following~\cite[Section~2]{BBE} or~\cite[Theorem~6.19]{GT} (see also,
e.~g., \cite[Section~4.3]{Pphil}), the pair of classes $(\Flat,\Cot)$
is a hereditary complete cotorsion pair generated by a set of objects
in $A\Modl$.
 Obviously, we have $(\VF,\CA)\le(\Flat,\Cot)$.

 In this context, the condition of
Lemma~\ref{ac=ac-bctr-equivalent-conditions}(1) says that, given
a complex of flat left $A$\+modules $F^\bu$ that is acyclic as a complex
in $A\Modl$, the $A$\+modules of cocycles of $F^\bu$ are flat if and
only if, for every complex of projective left $A$\+modules $P^\bu$,
all morphisms of complexes $P^\bu\rarrow F^\bu$ are homotopic to zero.
 This is a result
of Neeman~\cite[Theorem~8.6\,(i)\,$\Leftrightarrow$\,(iii)]{Neem}.
 The condition of
Lemma~\ref{ac=ac-bco-equivalent-conditions}(1) says that, given
a complex of cotorsion left $A$\+modules $C^\bu$ that is acyclic as
a complex in $A\Modl$, the $A$\+modules of cocycles of $C^\bu$ must be
cotorsion.
 This is the cotorsion periodicity theorem of Bazzoni,
Cort\'es-Izurdiaga, and Estrada~\cite[Theorem~1.2(2),
Proposition~4.8(2), or Theorem~5.1(2)]{BCE}.
 Thus the cotorsion pair $(\Flat,\Cot)=(A\Modl_\flat,\>A\Modl^\cot)$
in $\sK=A\Modl$ is of the flat type.

 In this example, the assertion of Lemma~\ref{flat-veryflat-periodicity}
claims that, in any complex of projective left $A$\+modules with flat
$A$\+modules of cocycles, the $A$\+modules of cocycles are actually
projective.
 This is a restatement of the Benson--Goodearl flat and projective
periodicity theorem~\cite[Theorem~2.5]{BG}, rediscovered by
Neeman in~\cite[Remark~2.15]{Neem}.
 The condition of Corollary~\ref{flat-type-definition-cor}(1) claims
that the exact category of flat left $A$\+modules $\Flat\cap\CA=
\Flat=A\Modl_\flat$ is of the flat type.
 This observation was made in~\cite[Theorems~7.14 and~7.18]{Pphil}.

 Theorem~\ref{flat-type-conventional-derived-equivalence-theorem}
claims that the inclusion $A\Modl^\cot\rarrow A\Modl$ induces
an equivalence of the conventional derived categories
$\sD(A\Modl^\cot)\simeq\sD(A\Modl)$.
 This was observed in~\cite[Theorem~7.17]{Pphil}.
 Theorem~\ref{flat-type-contraderived-equivalences-theorem}(a) claims
that the inclusion $A\Modl_\proj\rarrow A\Modl_\flat$ induces
an equivalence of the conventional derived categories
$\Hot(A\Modl_\proj)=\sD(A\Modl_\proj)\simeq\sD(A\Modl_\flat)$.
 This observation was mentioned in~\cite[Theorem~7.13]{Pphil}.
 Theorem~\ref{flat-type-contraderived-equivalences-theorem}(c) claims
that the inclusion $A\Modl^\cot\rarrow A\Modl$ induces an equivalence
of the Becker contraderived categories
$\sD^\bctr(A\Modl^\cot)\simeq\sD^\bctr(A\Modl)$.
 This was observed in~\cite[Theorem~7.19]{Pphil}.
\end{ex}

\begin{ex} \label{flat-cotorsion-pair-in-qcoh-flat-type}
 Let $X$ be a quasi-compact semi-separated scheme.
 Consider the Grothendieck category $\sK=X\Qcoh$ of quasi-coherent
sheaves on~$X$.
 Following
Example~\ref{very-flat-and-very-flaprojective-in-qcoh-exs}(1),
the classes $\VF=X\Qcoh_\vfl$ of very flat quasi-coherent sheaves
on $X$ and $\CA=X\Qcoh^\cta$ of contraadjusted quasi-coherent sheaves
on $X$ form a cotorsion pair $(\VF,\CA)$ of the very flat type in~$\sK$.

 A quasi-coherent sheaf $\cC$ on $X$ is said to be
\emph{cotorsion}~\cite{EE} if $\Ext_X^1(\cF,\cC)=0$ for all flat
quasi-coherent sheaves $\cF$ on~$X$.
 Denote the class of flat quasi-coherent sheaves by $X\Qcoh_\flat
\subset X\Qcoh$ and the class of cotorsion quasi-coherent sheaves
by $X\Qcoh^\cot\subset X\Qcoh$.

 The pair of classes $\Flat=X\Qcoh_\flat$ and $\Cot=X\Qcoh^\cot$ is
a hereditary complete cotorsion pair generated by a set of objects
in~$\sK$.
 Indeed, the deconstructibility of $X\Qcoh_\flat$ and the existence
of flat covers and cotorsion envelopes in $X\Qcoh$ are the results
of~\cite[Proposition~3.3 and Corollary~4.2]{EE} valid for any
scheme~$X$, while the assertion that the class $X\Qcoh_\flat$ is
generating in $X\Qcoh$ holds for quasi-compact semi-separated
schemes $X$ by~\cite[Section~2.4]{M-n} or~\cite[Lemma~A.1]{EP}
(see also~\cite[Corollary~4.1.11]{Pcosh}).
 Obviously, we have $(\VF,\CA)\le(\Flat,\Cot)$.

 In this context, the condition of
Corollary~\ref{flat-type-definition-cor}(1) or~(2) says that,
in the exact category $X\Qcoh_\flat^\cta=X\Qcoh_\flat\cap X\Qcoh^\cta
=\Flat\cap\CA$, the classes of acyclic, Becker-coacyclic, and
Becker-contraacyclic complexes coincide.
 This is the result of~\cite[Corollary~5.5.11]{Pcosh}.
 Thus the cotorsion pair $(\Flat,\Cot)=(X\Qcoh_\flat,\>X\Qcoh^\cot)$
in $\sK=X\Qcoh$ is of the flat type.

 The condition of
Lemma~\ref{ac=ac-bco-equivalent-conditions}(1) says that, in any
acyclic complex of cotorsion quasi-coherent sheaves on $X$,
the sheaves of cocycles are also cotorsion.
 This result, called the \emph{cotorsion periodicity for quasi-coherent
sheaves}, was first claimed in~\cite[Theorem~3.3]{CET}, but
the proof of~\cite[Lemma~3.2]{CET} was erroneous.
 A correct proof of the cotorsion periodicity for quasi-coherent
sheaves appeared in~\cite[Theorem~10.2 and Corollary~10.4]{PS6},
while the error in~\cite{CET} was subsequently fixed
in~\cite[Theorem~6.3(2) and Remark~6.7]{EGilO}.
 See~\cite[Remark~10.3]{PS6} for a discussion.

 In this example,
Theorem~\ref{flat-type-conventional-derived-equivalence-theorem}
claims that the inclusion $X\Qcoh^\cot\rarrow X\Qcoh$ induces
an equivalence of the conventional derived categories
$\sD(X\Qcoh^\cot)\rarrow\sD(X\Qcoh)$.
 This was observed in~\cite[Corollary~10.7]{PS6}.
 Theorem~\ref{flat-type-contraderived-equivalences-theorem}(a) claims
that the inclusion $X\Qcoh_\vfl\rarrow X\Qcoh_\flat$ induces
an equivalence of the conventional derived categories
$\sD(X\Qcoh_\vfl)\simeq\sD(X\Qcoh_\flat)$.
 This result is due to Estrada and Sl\'avik~\cite[Corollary~6.2]{ES}.
 Theorem~\ref{flat-type-contraderived-equivalences-theorem}(c) claims
that the inclusion $X\Qcoh^\cot\rarrow X\Qcoh^\cta$ induces
an equivalence of the Becker contraderived categories
$\sD^\bctr(X\Qcoh^\cot)\simeq\sD^\bctr(X\Qcoh^\cta)$.
 A closely related result for contraherent cosheaves can be found
in~\cite[Theorem~5.5.10]{Pcosh} (see~\cite[Corollaries~4.7.4(a)
and~4.7.5(a), and Lemmas~4.8.2 and~4.8.4(a)]{Pcosh} for the connection
between the quasi-coherent sheaves and contraherent cosheaves in
this context).
\end{ex}

\begin{ex} \label{flat-cotorsion-pair-in-qcoh-A-modules-flat-type}
 Let $X$ be a quasi-compact semi-separated scheme and $\cA$ be
a quasi-coherent quasi-algebra over~$X$ (as in
Example~\ref{very-flat-and-very-flaprojective-in-qcoh-exs}(2)).
 Consider the Grothendieck category $\sK=\cA\Qcoh$ of quasi-coherent
left $\cA$\+modules on~$X$.
 Let us denote by $\Ext_\cA^*({-},{-})$ the $\Ext$ groups in
the abelian category $\cA\Qcoh$.
 Following
Example~\ref{very-flat-and-very-flaprojective-in-qcoh-exs}(2),
the classes $\VF=\cA\Qcoh_\vflp$ of very flaprojective quasi-coherent
left $\cA$\+modules and $\CA=\cA\Qcoh^{X\dcta}$ of $X$\+contraadjusted
quasi-coherent left $\cA$\+modules form a cotorsion pair $(\VF,\CA)$
of the very flat type in~$\sK$.

 A quasi-coherent left $\cA$\+module $\cF$ is said to be \emph{flat}
if, for every affine open subscheme $U\subset X$, the left
$\cA(U)$\+module $\cF(U)$ is flat.
 It suffices to check this condition for the affine open subschemes $U$
belonging to any given affine open covering of~$X$
\cite[Lemma~1.24(a) and Section~3.4]{Pdomc}.
 A quasi-coherent left $\cA$\+module $\cC$ is said to be
\emph{$\cA$\+cotorsion} (or simply \emph{cotorsion}) if
$\Ext_\cA^1(\cF,\cC)=0$ for all flat quasi-coherent left
$\cA$\+modules~$\cF$ \,\cite[Section~4.3]{Pdomc}.
 Denote the class of flat quasi-coherent left $\cA$\+modules by
$\cA\Qcoh_\flat\subset\cA\Qcoh$ and the class of $\cA$\+cotorsion
quasi-coherent left $\cA$\+modules by $\cA\Qcoh^{\cA\dcot}\subset
\cA\Qcoh$.

 The pair of classes $\Flat=\cA\Qcoh_\flat$ and
$\Cot=\cA\Qcoh^{\cA\dcot}$ is a hereditary complete cotorsion pair
in $\cA\Qcoh$ \,\cite[Theorem~4.18 and Remark~4.20]{Pdomc}.
 Similarly to~\cite[Proposition~3.3]{EE}
or~\cite[Lemma~B.3.5]{Pcosh}, one can prove that the class of flat
quasi-coherent left $\cA$\+modules is deconstructible in $\cA\Qcoh$;
so this cotorsion pair is generated by a set of objects.
 One can easily see that $(\VF,\CA)\le(\Flat,\Cot)$ (since, in
the context of
Example~\ref{very-flat-and-very-flaprojective-in-modules-exs}(2),
all $A/R$\+very flaprojective $A$\+modules are flat as $A$\+modules
by~\cite[Remark~1.29]{Pdomc}).

 One expects the cotorsion pair
$(\Flat,\Cot)=(\cA\Qcoh_\flat,\>\cA\Qcoh^{\cA\dcot})$ in
$\sK=\cA\Qcoh$ to be of the flat type.
 In particular, the condition of
Lemma~\ref{ac=ac-bco-equivalent-conditions}(1) holds
by~\cite[Theorem~5.5]{Pdomc}.
 The condition of
Lemma~\ref{ac=ac-bctr-equivalent-conditions}(2) may be also provable
by developing the quasi-coherent/contraherent $\cA$\+module version
of the theory of~\cite[Chapter~5]{Pcosh} (such
as~\cite[Corollary~5.1.13]{Pcosh}).
 But this has not been worked out yet.
\end{ex}

\begin{ex} \label{flat-cotorsion-pair-in-comodules-flat-type}
 Let $A$ be an associative ring and $\cE$ be a coring over~$A$
(as in Example~\ref{flat-projdim1-modules-and-comodules-exs}(2)).
 Assume that $\cE$ is a flat right $A$\+module, and consider
the Grothendieck category $\sK=\cE\Comodl$ of left $\cE$\+comodules.
 Let us denote by $\Ext_\cE^*({-},{-})$ the Ext groups in the abelian
category $\cE\Comodl$.
 Assuming further that conditions~($*$) and~(${*}{*}$) are
satisfied, Example~\ref{flat-projdim1-modules-and-comodules-exs}(2)
provides us with a cotorsion pair $(\VF,\CA)=
(\cE\Comodl^{A\dpdo}_{A\dflat},\allowbreak\>
(\cE\Comodl^{A\dpdo}_{A\dflat})^{\perp_1})$ of the very flat
type in~$\sK$.

 Denote by $\cE\Comodl_{A\dflat}\subset\cE\Comodl$ the full subcategory
of \emph{$A$\+flat} $\cE$\+comodules, i.~e., $\cE$\+comodules whose
underlying $A$\+module is flat.
 A left $\cE$\+comodule $\cC$ is said to be \emph{cotorsion} if
$\Ext_\cE^1(\cF,\cC)=0$ for all $A$\+flat left $\cE$\+comodules~$\cF$
\,\cite[Section~C.2]{Pcosh}, \cite[Section~5]{Pflcc}.
 Denote the class of cotorsion left $\cE$\+comodules by
$\cE\Comodl^\cot\subset\cE\Comodl$.

 It follows from~($*$) that $\cE\Comodl_{A\dflat}$ is a resolving full
subcategory in $\cE\Comodl$; it is also obviously closed under direct
summands.
 Similarly to the discussion in
Example~\ref{flat-projdim1-modules-and-comodules-exs}(2), one can
use the Hill lemma~\cite[Theorem~7.10]{GT},
\cite[Theorem~2.1]{Sto-hill} in the module category $A\Modl$
together with the fact that the class of flat modules $A\Modl_\flat$
is deconstructible (by~\cite[Section~2]{BBE}) in order to show that
the class of $A$\+flat comodules $\cE\Comodl_{A\dflat}$ is
deconstructible in $\cE\Comodl$.

 By the Eklof--Trlifaj theorem (both
Theorems~\ref{loc-pres-eklof-trlifaj}
and~\ref{efficient-eklof-trlifaj} are applicable), it follows that
$(\Flat,\Cot)=(\cE\Comodl_{A\dflat},\allowbreak\>\cE\Comodl^\cot)$ is
a hereditary complete cotorsion pair generated by a set of objects
in $\cE\Comodl$.
 By construction, we have $\VF\subset\Flat$; so $(\VF,\CA)\le
(\Flat,\Cot)$.
 Furthermore, under the assumptions~($*$) and~(${*}{*}$),
the cotorsion pair $(\Flat,\Cot)$ satisfies the condition
of Lemma~\ref{ac=ac-bco-equivalent-conditions}(1)
by~\cite[Corollary~5.5]{Pflcc}.
 It would be interesting to know whether the cotorsion pair
$(\Flat,\Cot)=(\cE\Comodl_{A\dflat},\allowbreak\>\cE\Comodl^\cot)$
in $\sK=\cE\Comodl$ is of the flat type (i.~e., in other words,
whether the equivalent conditions of
Lemma~\ref{ac=ac-bctr-equivalent-conditions} hold for it).
\end{ex}

 Our final example in this section is \emph{not} a potential example
of a cotorsion pair of the flat type.
 Rather, it is an example of an \emph{exact category of the flat type}
arising as the left-hand class of a hereditary complete cotorsion
pair in a \emph{non-Grothendieck} abelian category.

\begin{ex} \label{flat-contramodules-exact-category-flat-type}
 Let $\fR$ be a complete, separated topological ring where open right
ideals form a base of neighborhoods of zero.
 We refer to~\cite[Section~5]{PR}, \cite[Section~2.1]{Prev}, 
\cite[Section~7]{Pflcc}, or~\cite[Section~3]{Pbc} for the definition
of the abelian category of \emph{left\/ $\fR$\+contramodules}
$\fR\Contra$.
 The contramodules are module-like structures with infinite summation
operations understood algebraically and subject to the natural axioms.
 The category $\fR\Contra$ is \emph{not} Grothendieck, because
the filtered colimit functors are usually \emph{not} exact in it;
instead, $\fR\Contra$ is a locally presentable abelian category
with enough projective objects.

 The definition of a \emph{flat} left $\fR$\+contramodule can be
found in~\cite[Section~5]{PR}, \cite[Section~3.3]{Prev},
\cite[Section~8]{Pflcc}, or~\cite[Section~3]{Pbc}.
 Assuming that $\fR$ \emph{has a countable base of neighborhoods of
zero}, the full subcategory of flat left $\fR$\+contramodules
$\fR\Contra_\flat$ is resolving in $\fR\Contra$ \,\cite[Corollary~6.8,
Lemma~6.9, Corollary~6.13, and Corollary~6.15]{PR},
\cite[Proposition~3.10]{Prev}, \cite[Lemma~8.4(a)]{Pflcc},
\cite[Lemma~4.2]{Pbc}.
 Under the same assumption, the full subcategory $\fR\Contra_\flat$
is the left-hand class of a hereditary complete cotorsion pair
$(\fR\Contra_\flat,\allowbreak\>\fR\Contra^\cot)$ generated by a set
of objects in $\fR\Contra$ \,\cite[Corollaries~7.1, 7.6,
and~7.8]{PR}, \cite[Proposition~8.1 and the subsequent
paragraph]{Pbc}.

 By Lemma~\ref{inj-proj-objects-in-left-right-classes}(a), it follows
that there are enough injective objects in the exact category
$\fR\Contra_\flat$, and the class of injective objects in
$\fR\Contra_\flat$ coincides with the class of flat cotorsion
left $\fR$\+contramodules $\fR\Contra_\flat^\cot=\fR\Contra_\flat
\cap\fR\Contra^\cot$.
 By Lemma~\ref{proj-resolving-inj-coresolving}(a), there are enough
projective objects in $\fR\Contra_\flat$, and the class of projective
objects in $\fR\Contra_\flat$ coincides with the class of projective
objects in $\fR\Contra$.

 Assume further that $\fR$ has a countable base of neighborhoods of
zero \emph{consisting of open two-sided ideals}.
 Then the three classes of acyclic, Becker-coacyclic, and
Becker-contraacyclic complexes in the exact category $\fR\Contra_\flat$
coincide with each other
by~\cite[Theorem~11.1\,(i$^{\mathrm{c}}$)\,$\Leftrightarrow$\,%
(ii$^{\mathrm{c}}$)\,$\Leftrightarrow$\,(vi$^{\mathrm{c}}$)]{Pbc}.
 Moreover, for the exact category $\sF=\fR\Contra_\flat$, the functors
$\Hot(\sF_\proj)\rarrow\sD^\bctr(\sF)$ and
$\Hot(\sF^\inj)\rarrow\sD^\bco(\sF)$ are triangulated
equivalences by~\cite[Corollary~11.2]{Pbc}.
 Thus $\sF=\fR\Contra_\flat$ is an exact category of the flat type, in
the sense of the definition in the beginning of
Section~\ref{flat-type-secn}.
\end{ex}

\Section{A Sandwiched Cotorsion Pair}  \label{sandwiched-secn}

 Let $\sK$ be a Grothendieck category with a cotorsion pair
$(\VF,\CA)$ of the very flat type and a cotorsion pair $(\Flat,\Cot)$
of the flat type such that $(\VF,\CA)\le(\Flat,\Cot)$.
 Let $(\sF,\sC)$ be a hereditary complete cotorsion pair generated by
a set of objects in~$\sK$.
 Assume that the cotorsion pair $(\sF,\sC)$ is sandwiched between
$(\VF,\CA)$ and $(\Flat,\Cot)$, in the sense that
$(\VF,\CA)\le(\sF,\sC)\le(\Flat,\Cot)$.

 What are the conditions that need to be satisfied in order to claim
that the cotorsion pair $(\sF,\sC)$ is of the flat type?
 What are the natural equivalent restatements of such conditions?
 These are the questions we address in this section. 

\begin{lem} \label{coderived-F-to-Flat-Verdier-quotient}
 Let\/ $\sK$ be a Grothendieck category with a cotorsion pair
$(\VF,\CA)$ of the very flat type and a cotorsion pair $(\Flat,\Cot)$
of the flat type.
 Let $(\sF,\sC)$ be a hereditary complete cotorsion pair generated by
a set of objects in\/~$\sK$.
 Assume that $(\VF,\CA)\le(\sF,\sC)\le(\Flat,\Cot)$.
 Then the triangulated functor\/ $\sD^\bco(\sF)\rarrow\sD^\bco(\Flat)$
induced by the inclusion of exact categories\/ $\sF\rarrow\Flat$ is
a triangulated Verdier quotient functor with the kernel category
equivalent to the Verdier quotient category
$$
 \frac{\Hot(\sF)\cap\Ac^\bco(\Flat)}{\Ac^\bco(\sF)}.
$$
\end{lem}

\begin{proof}
 Notice first of all that $\Ac^\bco(\sF)\subset\Ac^\bco(\Flat)$ by
Corollary~\ref{coderived-A1-to-A2-well-defined}; so the triangulated
functor $\sD^\bco(\sF)\rarrow\sD^\bco(\Flat)$ is well-defined.
 We apply Lemma~\ref{triangulated-quotient-lemma}(a) to the triangulated
category $\sH=\Hot(\Flat)$ with the triangulated subcategories
$\sX=\Ac^\bco(\Flat)$ and $\sE=\Hot(\sF)$.
 It suffices to check the assumption of
Lemma~\ref{triangulated-quotient-lemma}(a) in order to prove both
the desired assertions.

 Indeed, let $H^\bu$ be an object of $\Hot(\Flat)$.
 By Proposition~\ref{all-contraacyclic-cotorsion-pair-prop}, the pair
of classes $(\Com(\sF),\>\Ac^\bctr(\sC))$ is a hereditary complete
cotorsion pair in $\Com(\sK)$.
 Let $0\rarrow B^\bu\rarrow F^\bu\rarrow H^\bu\rarrow0$ be a special
precover exact sequence for the object $H^\bu\in\Com(\sK)$ with
respect to this cotorsion pair.
 So we have $F^\bu\in\Com(\sF)$ and $B^\bu\in\Ac^\bctr(\sC)$.
 Since the class $\Flat$ is closed under kernels of epimorphisms in
$\sK$, we actually have $B^\bu\in\Com(\Flat)\cap\Ac^\bctr(\sC)$.

 By Corollary~\ref{contraderived-B2-to-B1-well-defined}, this implies
that $B^\bu\in\Com(\Flat)\cap\Ac^\bctr(\CA)$.
 By Lemma~\ref{flat-type-convenient-restatement-lemma}, it follows
that $B^\bu\in\Ac^\bco(\Flat)$.
 Using Lemma~\ref{totalizations-are-co-contra-acyclic}(a), we can
conclude that the cone of the morphism of complexes $F^\bu\rarrow
H^\bu$ is Becker-coacyclic in $\Flat$, as desired.
\end{proof}

\begin{lem} \label{coderived-VF-F-Flat-adjunction}
 Let\/ $\sK$ be a Grothendieck category with a cotorsion pair
$(\VF,\CA)$ of the very flat type and a cotorsion pair $(\Flat,\Cot)$
of the flat type.
 Let $(\sF,\sC)$ be a hereditary complete cotorsion pair generated by
a set of objects in\/~$\sK$.
 Assume that $(\VF,\CA)\le(\sF,\sC)\le(\Flat,\Cot)$.
 Then the triangulated equivalence\/
$\sD^\bco(\VF)\simeq\sD^\bco(\Flat)$ from
Theorem~\ref{flat-type-contraderived-equivalences-theorem}(b) makes
the triangulated functor\/ $\sD^\bco(\VF)\rarrow\sD^\bco(\sF)$
induced by the inclusion\/ $\VF\rarrow\sF$ left adjoint to
the triangulated functor\/ $\sD^\bco(\sF)\rarrow\sD^\bco(\Flat)$
induced by the inclusion\/ $\sF\rarrow\Flat$.
\end{lem}

\begin{proof}
 In view of Lemma~\ref{coderived-F-to-Flat-Verdier-quotient},
it suffices to show that $\Hom_{\sD^\bco(\sF)}(E^\bu,C^\bu)=0$
for all complexes $E^\bu\in\Com(\VF)$ and
$C^\bu\in\Com(\sF)\cap\Ac^\bco(\Flat)$.
 By Corollary~\ref{coderived-of-A-well-behaved}, there exists
a morphism of complexes $C^\bu\rarrow B^\bu$ with
$B^\bu\in\Com(\sF\cap\sC)$ and the cone belonging to $\Ac^\bco(\sF)$.
 By Corollary~\ref{coderived-A1-to-A2-well-defined}, we have
$\Ac^\bco(\sF)\subset\Ac^\bco(\Flat)$; hence
$B^\bu\in\Com(\sF\cap\sC)\cap\Ac^\bco(\Flat)$.

 Furthermore, $\Hom_{\sD^\bco(\sF)}(E^\bu,C^\bu)\simeq
\Hom_{\sD^\bco(\sF)}(E^\bu,B^\bu)\simeq\Hom_{\Hot(\sF)}(E^\bu,B^\bu)$,
since $B^\bu\in\Com(\sF\cap\sC)=\Com(\sF^\inj)$.
 Now we have $B^\bu\in\Com(\CA)\cap\Ac^\bco(\Flat)$, and
by Lemma~\ref{flat-type-convenient-restatement-lemma} it follows that
$B^\bu\in\Ac^\bctr(\CA)$.
 It remains to refer to
Proposition~\ref{all-contraacyclic-cotorsion-pair-prop}
and Lemma~\ref{Ext1-as-homotopy-Hom-lemma} for the assertion that
$E^\bu\in\Com(\VF)$ and $B^\bu\in\Ac^\bctr(\CA)$ imply
$\Hom_{\Hot(\sK)}(E^\bu,B^\bu)=0$.
\end{proof}

\begin{cor} \label{coderived-VF-to-F-fully-faithful}
 Let\/ $\sK$ be a Grothendieck category with a cotorsion pair
$(\VF,\CA)$ of the very flat type and a cotorsion pair $(\Flat,\Cot)$
of the flat type.
 Let $(\sF,\sC)$ be a hereditary complete cotorsion pair generated by
a set of objects in\/~$\sK$.
 Assume that $(\VF,\CA)\le(\sF,\sC)\le(\Flat,\Cot)$.
 Then the triangulated functor\/ $\sD^\bco(\VF)\rarrow\sD^\bco(\sF)$
induced by the inclusion of exact categories\/ $\VF\rarrow\sF$ is
fully faithful.
\end{cor}

\begin{proof}
 This is a corollary of
Lemmas~\ref{coderived-F-to-Flat-Verdier-quotient}
and~\ref{coderived-VF-F-Flat-adjunction}.
 The point is that a left/right adjoint triangulated functor is
fully faithful if and only if its right/left adjoint is a Verdier
quotient functor~\cite[Proposition~1.3]{GZ}.
\end{proof}

\begin{lem} \label{contraderived-C-to-CA-Verdier-quotient}
 Let\/ $\sK$ be a Grothendieck category with a cotorsion pair
$(\VF,\CA)$ of the very flat type and a cotorsion pair $(\Flat,\Cot)$
of the flat type.
 Let $(\sF,\sC)$ be a hereditary complete cotorsion pair generated by
a set of objects in\/~$\sK$.
 Assume that $(\VF,\CA)\le(\sF,\sC)\le(\Flat,\Cot)$.
 Then the triangulated functor\/ $\sD^\bctr(\sC)\rarrow\sD^\bctr(\CA)$
induced by the inclusion of exact categories\/ $\sC\rarrow\CA$ is
a triangulated Verdier quotient functor with the kernel category
equivalent to the Verdier quotient category
$$
 \frac{\Hot(\sC)\cap\Ac^\bctr(\CA)}{\Ac^\bctr(\sC)}.
$$
\end{lem}

\begin{proof}
 Notice first of all that $\Ac^\bctr(\sC)\subset\Ac^\bctr(\CA)$ by
Corollary~\ref{contraderived-B2-to-B1-well-defined}; so the triangulated
functor $\sD^\bctr(\sC)\rarrow\sD^\bctr(\CA)$ is well-defined.
 We apply Lemma~\ref{triangulated-quotient-lemma}(b) to the triangulated
category $\sH=\Hot(\CA)$ with the triangulated subcategories
$\sX=\Ac^\bctr(\CA)$ and $\sG=\Hot(\sC)$.
 It suffices to check the assumption of
Lemma~\ref{triangulated-quotient-lemma}(b) in order to prove both
the desired assertions.

 Indeed, let $E^\bu$ be an object of $\Hot(\CA)$.
 By Proposition~\ref{coacyclic-all-cotorsion-pair-prop}, the pair of
classes $(\Ac^\bco(\sF),\>\Com(\sC))$ is a hereditary complete cotorsion
pair in $\Com(\sK)$.
 Let $0\rarrow E^\bu\rarrow C^\bu\rarrow A^\bu\rarrow0$ be a special
preenvelope exact sequence for the object $E^\bu\in\Com(\sK)$ with
respect to this cotorsion pair.
 So we have $C^\bu\in\Com(\sC)$ and $A^\bu\in\Ac^\bco(\sF)$.
 Since the class $\CA$ is closed under cokernels of monomorphisms in
$\sK$, we actually have $A^\bu\in\Com(\CA)\cap\Ac^\bco(\sF)$.

 By Corollary~\ref{coderived-A1-to-A2-well-defined}, it follows that
$A^\bu\in\Com(\CA)\cap\Ac^\bco(\Flat)$.
 By Lemma~\ref{flat-type-convenient-restatement-lemma}, this implies
that $A^\bu\in\Ac^\bctr(\CA)$.
 Using Lemma~\ref{totalizations-are-co-contra-acyclic}(a), we can
conclude that the cone of the morphism of complexes $E^\bu\rarrow
C^\bu$ is Becker-contraacyclic in $\CA$, as desired.
\end{proof}

\begin{lem} \label{contraderived-Cot-C-CA-adjunction}
 Let\/ $\sK$ be a Grothendieck category with a cotorsion pair
$(\VF,\CA)$ of the very flat type and a cotorsion pair $(\Flat,\Cot)$
of the flat type.
 Let $(\sF,\sC)$ be a hereditary complete cotorsion pair generated by
a set of objects in\/~$\sK$.
 Assume that $(\VF,\CA)\le(\sF,\sC)\le(\Flat,\Cot)$.
 Then the triangulated equivalence\/
$\sD^\bctr(\Cot)\simeq\sD^\bctr(\CA)$ from
Theorem~\ref{flat-type-contraderived-equivalences-theorem}(c) makes
the triangulated functor\/ $\sD^\bctr(\Cot)\rarrow\sD^\bctr(\sC)$
induced by the inclusion\/ $\Cot\rarrow\sC$ right adjoint to
the triangulated functor\/ $\sD^\bctr(\sC)\rarrow\sD^\bctr(\CA)$
induced by the inclusion\/ $\sC\rarrow\CA$.
\end{lem}

\begin{proof}
 In view of Lemma~\ref{contraderived-C-to-CA-Verdier-quotient},
it suffices to show that $\Hom_{\sD^\bctr(\sC)}(C^\bu,D^\bu)=0$
for all complexes $C^\bu\in\Com(\sC)\cap\Ac^\bctr(\CA)$ and
$D^\bu\in\Com(\Cot)$.
 By Corollary~\ref{contraderived-of-B-well-behaved}, there exists
a morphism of complexes $A^\bu\rarrow C^\bu$ with
$A^\bu\in\Com(\sF\cap\sC)$ and the cone belonging to $\Ac^\bctr(\sC)$.
 By Corollary~\ref{contraderived-B2-to-B1-well-defined}, we have
$\Ac^\bctr(\sC)\subset\Ac^\bctr(\CA)$; hence
$A^\bu\in\Com(\sF\cap\sC)\cap\Ac^\bctr(\CA)$.

 Furthermore, $\Hom_{\sD^\bctr(\sC)}(C^\bu,D^\bu)\simeq
\Hom_{\sD^\bctr(\sC)}(A^\bu,D^\bu)\simeq\Hom_{\Hot(\sC)}(A^\bu,D^\bu)$,
since $A^\bu\in\Com(\sF\cap\sC)=\Com(\sC_\proj)$.
 Now we have $A^\bu\in\Com(\Flat)\cap\Ac^\bctr(\CA)$, and
by Lemma~\ref{flat-type-convenient-restatement-lemma} it follows that
$A^\bu\in\Ac^\bco(\Flat)$.
 It remains to refer to
Proposition~\ref{coacyclic-all-cotorsion-pair-prop}
and Lemma~\ref{Ext1-as-homotopy-Hom-lemma} for the assertion that
$A^\bu\in\Ac^\bco(\Flat)$ and $D^\bu\in\Com(\Cot)$ imply
$\Hom_{\Hot(\sK)}(A^\bu,D^\bu)=0$.
\end{proof}

\begin{cor} \label{contraderived-Cot-to-C-fully-faithful}
 Let\/ $\sK$ be a Grothendieck category with a cotorsion pair
$(\VF,\CA)$ of the very flat type and a cotorsion pair $(\Flat,\Cot)$
of the flat type.
 Let $(\sF,\sC)$ be a hereditary complete cotorsion pair generated by
a set of objects in\/~$\sK$.
 Assume that $(\VF,\CA)\le(\sF,\sC)\le(\Flat,\Cot)$.
 Then the triangulated functor\/ $\sD^\bctr(\Cot)\rarrow\sD^\bctr(\sC)$
induced by the inclusion of exact categories\/ $\Cot\rarrow\sC$ is
fully faithful.
\end{cor}

\begin{proof}
 This is a corollary of
Lemmas~\ref{contraderived-C-to-CA-Verdier-quotient}
and~\ref{contraderived-Cot-C-CA-adjunction}.
 The proof is similar to that of
Corollary~\ref{coderived-VF-to-F-fully-faithful}.
\end{proof}

\begin{cor} \label{F-to-Flat-isomorphic-to-C-to-CA-cor}
 Let\/ $\sK$ be a Grothendieck category with a cotorsion pair
$(\VF,\CA)$ of the very flat type and a cotorsion pair $(\Flat,\Cot)$
of the flat type.
 Let $(\sF,\sC)$ be a hereditary complete cotorsion pair generated by
a set of objects in\/~$\sK$.
 Assume that $(\VF,\CA)\le(\sF,\sC)\le(\Flat,\Cot)$.
 Then the triangulated equivalences from
Corollaries~\ref{cotorsion-pair-co-contra-equivalence}
and~\ref{veryflat-flat-type-square-of-triang-equivalences}
identify the functor
$$
 \sD^\bco(\sF)\lrarrow\sD^\bco(\Flat)
$$
induced by the inclusion of exact categories\/ $\sF\rarrow\Flat$
with the functor
$$
 \sD^\bctr(\sC)\lrarrow\sD^\bctr(\CA)
$$
induced by the inclusion of exact categories\/ $\sC\rarrow\CA$.
 Consequently, there is a natural equivalence of triangulated
categories
$$
 \frac{\Hot(\sF)\cap\Ac^\bco(\Flat)}{\Ac^\bco(\sF)}\simeq
 \frac{\Hot(\sC)\cap\Ac^\bctr(\CA)}{\Ac^\bctr(\sC)}.
$$
\end{cor}

\begin{proof}
 Presuming the triangulated equivalences from
Corollaries~\ref{cotorsion-pair-co-contra-equivalence}
and~\ref{veryflat-flat-type-square-of-triang-equivalences}
as identifications, the functor
$\sD^\bco(\sF)\rarrow\sD^\bco(\Flat)$ is right adjoint to
the functor $\sD^\bco(\VF)\rarrow\sD^\bco(\sF)$ by
Lemma~\ref{coderived-VF-F-Flat-adjunction}, while the functor
$\sD^\bctr(\sC)\lrarrow\sD^\bctr(\CA)$ is right adjoint to
the functor $\sD^\bco(\VF)\rarrow\sD^\bco(\sF)$ by
Theorem~\ref{nested-cotorsion-pairs-adjunction-theorem}.
 Being two right adjoints to one and the same triangulated functor
$\sD^\bco(\VF)\rarrow\sD^\bco(\sF)$, the two triangulated functors
$\sD^\bco(\sF)\rarrow\sD^\bco(\Flat)$ and $\sD^\bctr(\sC)\lrarrow
\sD^\bctr(\CA)$ are consequently naturally isomorphic.

 Alternatively, the functor $\sD^\bco(\sF)\rarrow\sD^\bco(\Flat)$
is left adjoint to the functor $\sD^\bctr(\Cot)\rarrow\sD^\bctr(\sC)$
by Theorem~\ref{nested-cotorsion-pairs-adjunction-theorem}, while
the functor $\sD^\bctr(\sC)\lrarrow\sD^\bctr(\CA)$ is left adjoint
to the functor $\sD^\bctr(\Cot)\rarrow\sD^\bctr(\sC)$ by
Lemma~\ref{contraderived-Cot-C-CA-adjunction}.
 Being two left adjoints to one and the same triangulated functor
$\sD^\bctr(\Cot)\rarrow\sD^\bctr(\sC)$, the two triangulated functors
$\sD^\bco(\sF)\rarrow\sD^\bco(\Flat)$ and $\sD^\bctr(\sC)\lrarrow
\sD^\bctr(\CA)$ are consequently naturally isomorphic.

 The second assertion of the corollary is obtained by passing to
the kernel categories of the two isomorphic Verdier quotient functors
and using Lemmas~\ref{coderived-F-to-Flat-Verdier-quotient}
and~\ref{contraderived-C-to-CA-Verdier-quotient} for the descriptions
of the kernel categories.
\end{proof}

 In the commutative diagram in the following theorem, we use
the notation $\sD^{\varnothing=\bco}(\sE)$ to point out the fact
that $\Ac(\sE)=\Ac^\bco(\sE)$ and therefore $\sD(\sE)=\sD^\bco(\sE)$
for the exact category $\sE$ in question.
 This notation was introduced in
Corollary~\ref{veryflat-flat-type-square-of-triang-equivalences}.

\begin{thm} \label{sandwiched-recollement-theorem}
 Let\/ $\sK$ be a Grothendieck category with a cotorsion pair
$(\VF,\CA)$ of the very flat type and a cotorsion pair $(\Flat,\Cot)$
of the flat type.
 Let $(\sF,\sC)$ be a hereditary complete cotorsion pair generated by
a set of objects in\/~$\sK$.
 Assume that $(\VF,\CA)\le(\sF,\sC)\le(\Flat,\Cot)$.
 Then there is a recollement of triangulated categories
$$
 \xymatrix{
  && \sD^\bco(\sF) \ar@{=}[d] \ar@/_1.5pc/@{->>}[lld]
  && \sD^{\varnothing=\bco}(\VF) \ar@<2pt>[d] \ar@{>->}[ll] \\
  \frac{\Hot(\sF)\cap\Ac^\bco(\Flat)}{\Ac^\bco(\sF)} 
  \ar@{>->}[rr] \ar@{=}[d]
  && \sD^\bco(\sF) \ar@{->>}[rr] \ar@{=}[d]
  && \sD^{\varnothing=\bco}(\Flat) \ar@<2pt>@{-}[u] \ar@{=}[d] \\
  \frac{\Hot(\sC)\cap\Ac^\bctr(\CA)}{\Ac^\bctr(\sC)} \ar@{>->}[rr]
  && \sD^\bctr(\sC) \ar@{->>}[rr] \ar@{=}[d]
  && \sD^\bctr(\CA) \ar@<-2pt>@{-}[d] \\
  && \sD^\bctr(\sC) \ar@/^1.5pc/@{->>}[llu]
  && \sD^\bctr(\Cot) \ar@{>->}[ll] \ar@<-2pt>[u]
 }
$$
\end{thm}

\begin{proof}
 One has $\sD^\bco(\VF)=\sD(\VF)$ by
Lemma~\ref{exact-category-of-finite-homol-dim}(a)
and $\sD^\bco(\Flat)=\sD(\Flat)$ by
Lemma~\ref{ac=ac-bco-equivalent-conditions}(4).
 The Verdier quotient sequence in the second row is the result
of Lemma~\ref{coderived-F-to-Flat-Verdier-quotient}, and the one in
the third row is the result of
Lemma~\ref{contraderived-C-to-CA-Verdier-quotient}.
 The commutative diagram formed by the two middle rows is provided
by Corollary~\ref{F-to-Flat-isomorphic-to-C-to-CA-cor}.
 The adjunction in the upper rightmost square is the result of
Lemma~\ref{coderived-VF-F-Flat-adjunction}, and the one in
the lower rightmost square is the result of
Lemma~\ref{contraderived-Cot-C-CA-adjunction}.
 This proves the existence of the recollement, which in turn implies
the existence of the curvilinear adjoint functors and the Verdier
quotient sequences in the upper and lower lines.
\end{proof}

\begin{thm} \label{sandwiched-equivalent-conditions-theorem}
 Let\/ $\sK$ be a Grothendieck category with a cotorsion pair
$(\VF,\CA)$ of the very flat type and a cotorsion pair $(\Flat,\Cot)$
of the flat type.
 Let $(\sF,\sC)$ be a hereditary complete cotorsion pair generated by
a set of objects in\/~$\sK$.
 Assume that $(\VF,\CA)\le(\sF,\sC)\le(\Flat,\Cot)$.
 Then the following ten conditions are equivalent:
\begin{enumerate}
\item the cotorsion pair $(\sF,\sC)$ is of the flat type;
\item the triangulated functor\/ $\sD^\bco(\VF)\rarrow\sD^\bco(\sF)$
induced by the inclusion\/ $\VF\rarrow\sF$ is a triangulated
equivalence;
\item the triangulated functor\/ $\sD^\bco(\VF)\rarrow\sD^\bco(\sF)$
induced by the inclusion\/ $\VF\rarrow\sF$ is essentially surjective;
\item the triangulated functor\/ $\sD^\bco(\sF)\rarrow\sD^\bco(\Flat)$
induced by the inclusion\/ $\sF\rarrow\Flat$ is a triangulated
equivalence;
\item any complex in\/ $\sF$ that is Becker-coacyclic in\/ $\Flat$ is
also Becker-coacyclic in\/~$\sF$;
\item the triangulated functor\/ $\sD^\bctr(\Cot)\rarrow\sD^\bctr(\sC)$
induced by the inclusion\/ $\Cot\rarrow\sC$ is a triangulated
equivalence;
\item the triangulated functor\/ $\sD^\bctr(\Cot)\rarrow\sD^\bctr(\sC)$
induced by the inclusion\/ $\Cot\rarrow\sC$ is essentially surjective;
\item the triangulated functor\/ $\sD^\bctr(\sC)\rarrow\sD^\bctr(\CA)$
induced by the inclusion\/ $\sC\rarrow\CA$ is a triangulated
equivalence;
\item any complex in\/ $\sC$ that is Becker-contraacyclic in\/ $\CA$ is
also Becker-con\-tra\-acyclic in\/~$\sC$;
\item two periodicity properties hold simultaneously:
\begin{enumerate}
\renewcommand{\theenumii}{\alph{enumii}}
\item in any acyclic complex in\/ $\sK$ with the terms belonging to\/
$\sF$ and the objects of cocycles belonging to\/ $\Flat$, the objects
of cocycles actually belong to\/~$\sF$;
\item in any acyclic complex in\/ $\sK$ with the terms belonging to\/
$\sC$, the objects of cocycles also belong to\/~$\sC$.
\end{enumerate}
\end{enumerate}
\end{thm}

\begin{proof}
 (2)~$\Longleftrightarrow$~(3)
 The triangulated functor $\sD^\bco(\VF)\rarrow\sD^\bco(\sF)$ is
always fully faithful by
Corollary~\ref{coderived-VF-to-F-fully-faithful}.

 (2)~$\Longleftrightarrow$~(4)
 The triangulated functors $\sD^\bco(\VF)\rarrow\sD^\bco(\sF)$
and $\sD^\bco(\sF)\rarrow\sD^\bco(\Flat)$ are adjoint by
Lemma~\ref{coderived-VF-F-Flat-adjunction}.

 (4)~$\Longleftrightarrow$~(5)
 The triangulated functor $\sD^\bco(\sF)\rarrow\sD^\bco(\Flat)$ is
an equivalence if and only if $\Hot(\sF)\cap\Ac^\bco(\Flat)=
\Ac^\bco(\sF)$ by Lemma~\ref{coderived-F-to-Flat-Verdier-quotient}.

 (6)~$\Longleftrightarrow$~(7)
 The triangulated functor $\sD^\bctr(\Cot)\rarrow\sD^\bctr(\sC)$ is
always fully faithful by
Corollary~\ref{contraderived-Cot-to-C-fully-faithful}.

 (6)~$\Longleftrightarrow$~(8)
 The triangulated functors $\sD^\bctr(\Cot)\rarrow\sD^\bctr(\sC)$
and $\sD^\bctr(\sC)\rarrow\sD^\bctr(\CA)$ are adjoint by
Lemma~\ref{contraderived-Cot-C-CA-adjunction}.

 (8)~$\Longleftrightarrow$~(9)
 The triangulated functor $\sD^\bctr(\sC)\rarrow\sD^\bctr(\CA)$ is
an equivalence if and only if $\Hot(\sC)\cap\Ac^\bctr(\CA)=
\Ac^\bctr(\sC)$ by Lemma~\ref{contraderived-C-to-CA-Verdier-quotient}.

 (4)~$\Longleftrightarrow$~(8)
 The triangulated functors $\sD^\bco(\sF)\rarrow\sD^\bco(\Flat)$
and $\sD^\bctr(\sC)\rarrow\sD^\bctr(\CA)$ are isomorphic by
Corollary~\ref{F-to-Flat-isomorphic-to-C-to-CA-cor}.

 (5)~$\Longleftrightarrow$~(9)
 One has $\Hot(\sF)\cap\Ac^\bco(\Flat)=\Ac^\bco(\sF)$ if and only if
$\Hot(\sC)\cap\Ac^\bctr(\CA)=\Ac^\bctr(\sC)$ by
Corollary~\ref{F-to-Flat-isomorphic-to-C-to-CA-cor}.

 (5)~$\Longleftrightarrow$~(10)
 Condition~(5) says that $\Hot(\sF)\cap\Ac^\bco(\Flat)=\Ac^\bco(\sF)$.
 Notice the inclusions and equality
$$
 \Ac^\bco(\sF)\subset\Ac(\sF)\subset\Com(\sF)\cap\Ac(\Flat)=
 \Com(\sF)\cap\Ac^\bco(\Flat).
$$
 Here $\Ac^\bco(\sF)\subset\Ac(\sF)$ by
Corollary~\ref{coacyclic-in-A-are-acyclic-in-A} and
$\Ac(\Flat)=\Ac^\bco(\Flat)$ by
Lemma~\ref{ac=ac-bco-equivalent-conditions}(4)
(since the cotorsion pair $(\Flat,\Cot)$ is of the flat type),
while one has $\Ac(\sF)\subset\Ac(\Flat)$ since the inclusion
functor $\sF\rarrow\Flat$ is exact.

 Hence (5)~holds if and only if $\Ac^\bco(\sF)=\Ac(\sF)$ \emph{and}
$\Ac(\sF)=\Com(\sF)\cap\Ac(\Flat)$.
 Now the equation $\Ac(\sF)=\Com(\sF)\cap\Ac(\Flat)$ is a restatement
of condition~(10a), while the equation $\Ac^\bco(\sF)=\Ac(\sF)$
is equivalent to condition~(10b) by
Theorem~\ref{cotorsion-periodity-type-equivalent-conditions}\,%
(2)\,$\Leftrightarrow$\,(4).

 (10)~$\Longrightarrow$~(9)
 Let $D^\bu$ be a complex belonging to $\Com(\sC)\cap\Ac^\bctr(\CA)$.
 By Proposition~\ref{all-contraacyclic-cotorsion-pair-prop}, we have
a hereditary complete cotorsion pair $(\Com(\sF),\>\Ac^\bctr(\sC))$
in $\Com(\sK)$.
 Let $0\rarrow D^\bu\rarrow C^\bu\rarrow F^\bu\rarrow0$ be a special
preenvelope exact sequence~\eqref{special-preenvelope-sequence} for
the object $D^\bu\in\Com(\sK)$ with respect to this cotorsion pair.
 So we have $C^\bu\in\Ac^\bctr(\sC)$ and $F^\bu\in\Com(\sF)$.
 Since the class $\sC$ is closed under cokernels of monomorphisms in
$\sK$, we also have $F^\bu\in\Com(\sC)$.
 Moreover, since $\Ac^\bctr(\sC)\subset\Ac^\bctr(\CA)$ by
Corollary~\ref{contraderived-B2-to-B1-well-defined} and the class
$\Ac^\bctr(\CA)$ is closed under cokernels of monomorphisms in
$\Com(\sK)$ by Proposition~\ref{all-contraacyclic-cotorsion-pair-prop},
we actually have $F^\bu\in\Com(\sF\cap\sC)\cap\Ac^\bctr(\CA)$.

 By Lemma~\ref{flat-type-convenient-restatement-lemma}, we have
$\Com(\Flat)\cap\Ac^\bctr(\CA)=\Ac(\Flat\cap\CA)$; so
$F^\bu\in\Ac(\Flat\cap\CA)$.
 Now condition~(10a) says that $\Com(\sF)\cap\Ac(\Flat)=\Ac(\sF)$;
hence $F^\bu\in\Ac(\sF)$.
 Condition~(10b) essentially says that $\Com(\sC)\cap\Ac(\CA)=\Ac(\sC)$
(as any acyclic complex in $\CA$ is acyclic in~$\sK$;
cf.\ Lemma~\ref{finite-coresol-dim-periodicity});
hence $F^\bu\in\Ac(\sC)$.
 It follows that the objects of cocycles of the acyclic complex $F^\bu$
in $\sK$ belong to $\sF\cap\sC$; so $F^\bu\in\Ac(\sF\cap\sC)$.
 As the exact category $\sF\cap\sC$ is split (all the admissible short
exact sequences in $\sF\cap\sC$ are split), we can conclude that
the complex $F^\bu$ is contractible in $\sF\cap\sC$.

 Hence $F^\bu\in\Ac^\bctr(\sC)$.
 As we also have $C^\bu\in\Ac^\bctr(\sC)$, in view of
Lemma~\ref{totalizations-are-co-contra-acyclic}(a) it follows that
$D^\bu\in\Ac^\bctr(\sC)$.
 Alternatively, one can observe that the short exact sequence
$0\rarrow D^\bu\rarrow C^\bu\rarrow F^\bu\rarrow0$ is termwise split;
so the complex $D^\bu$ is actually homotopy equivalent to~$C^\bu$.

 (1)~$\Longrightarrow$~(2) holds by
Theorem~\ref{flat-type-contraderived-equivalences-theorem}(b).

 (1)~$\Longrightarrow$~(8) holds by
Theorem~\ref{flat-type-contraderived-equivalences-theorem}(c).

 (10)~$\Longrightarrow$~(1)
 Let us check the conditions of
Lemmas~\ref{ac=ac-bctr-equivalent-conditions}(1)
and~\ref{ac=ac-bco-equivalent-conditions}(1) for the cotorsion pair
$(\VF,\CA)$, the class $\sF$, and the cotorsion pair $(\sF,\sC)$.
 Given a complex $F^\bu\in\sF\cap\CA\subset\Flat\cap\CA$ that is
acyclic in $\sK$, the objects of cocycles of $F^\bu$ belong to
$\sF$ if and only if they belong to $\Flat$.
 That is what condition~(10a) says.
 Thus the condition of Lemma~\ref{ac=ac-bctr-equivalent-conditions}(1)
holds for $\sF$, since it holds for~$\Flat$.
 On the other hand, the condition of
Lemma~\ref{ac=ac-bco-equivalent-conditions}(1) for the cotorsion
pair $(\sF,\sC)$ coincides with~(10b).
\end{proof}

\Section{The Flaprojective Conjecture}

 Let $R\rarrow A$ be a homomorphism of associative rings.
 The definition of a \emph{cotorsion left $R$\+module} was given in
Example~\ref{flat-cotorsion-pair-in-modules-flat-type}.
 The following lemma was already mentioned in
Example~\ref{very-flat-and-very-flaprojective-in-modules-exs}(2).

\begin{lem} \label{restriction-of-scalars-cotorsion}
 Any cotorsion $A$\+module is also cotorsion as an $R$\+module.
\end{lem}

\begin{proof}
 This is~\cite[Lemma~1.3.4(a)]{Pcosh}, \cite[Lemma~1.6(a)]{Pdomc},
or~\cite[Lemma~2.3]{PBas}.
\end{proof}

 We will say that a left $A$\+module is \emph{$R$\+cotorsion} if it is
cotorsion as an $R$\+module.
 A left $A$\+module $F$ is said to be \emph{flaprojective relative
to~$R$} (or~\emph{$A/R$\+flaprojective}) \cite[Section~1.6]{Pdomc},
\cite[Section~3]{PBas} if $\Ext_A^1(F,C)=0$ for all $R$\+cotorsion
left $A$\+modules~$C$ (cf.\ the definition of an $A/R$\+very
flaprojective $A$\+module in
Example~\ref{very-flat-and-very-flaprojective-in-modules-exs}(2)).

 Denote the class of $R$\+cotorsion left $A$\+modules by
$A\Modl^{R\dcot}\subset A\Modl$ and the class of $A/R$\+flaprojective
left $A$\+modules by $A\Modl_{A/R\dflpr}\subset A\Modl$.

 Clearly, all projective left $A$\+modules are $A/R$\+flaprojective.
 It follows from Lemma~\ref{restriction-of-scalars-cotorsion} that
all $A/R$\+flaprojective $A$\+modules are flat as $A$\+modules
(see~\cite[Remark~1.29]{Pdomc}).
 So we have $A\Modl_\proj\subset A\Modl_{A/R\dflpr}\subset
A\Modl_\flat$ and $A\Modl^\cot\subset A\Modl^{R\dcot}$ (in the notation
of Examples~\ref{projective-cotorsion-pair-is-very-flat-type-ex}
and~\ref{flat-cotorsion-pair-in-modules-flat-type}).
{\hbadness=1150\par}

\begin{prop} \label{flaprojective-cotorsion-pair-prop}
 For any homomorphism of associative rings $R\rarrow A$, the pair of
classes $(A\Modl_{A/R\dflpr},\>A\Modl^{R\dcot})$ is a hereditary
complete cotorsion pair generated by a set of objects in $A\Modl$.
 A left $A$\+module $F$ is $A/R$\+flaprojective if and only if $F$ is
a direct summand of an $A$\+module filtered by $A$\+modules of
the form $A\ot_RG$, where $G$ ranges over the flat left $R$\+modules.
\end{prop}

\begin{proof}
 This is~\cite[Theorem~3.5]{PBas} and a special case
of~\cite[Proposition~3.3]{Pctrl}.
 For the assertion that the cotorsion pair in question is hereditary,
see~\cite[Lemma~1.30]{Pdomc}.
\end{proof}

\begin{lem} \label{R-cotorsion-A-module-periodicity-lemma}
 Let $C^\bu$ be a complex of $R$\+cotorsion left $A$\+modules that is
acyclic as a complex in $A\Modl$.
 Then the $A$\+modules of cocycles of the complex $C^\bu$ are
also $R$\+cotorsion.
\end{lem}

\begin{proof}
 Apply the cotorsion periodicity theorem~\cite[Theorem~5.1(2)]{BCE}
to the ring $R$ and the complex $C^\bu$ viewed as a complex of
$R$\+modules.
\end{proof}

\begin{cor} \label{flaprojective-coacyclic=acyclic}
 In the exact category of $A/R$\+flaprojective left $A$\+modules,
the classes of acyclic and Becker-coacyclic complexes coincide,
$$
 \Ac^\bco(A\Modl_{A/R\dflpr})=\Ac(A\Modl_{A/R\dflpr}).
$$
 So one has\/ $\sD^\bco(A\Modl_{A/R\dflpr})=
\sD(A\Modl_{A/R\dflpr})$.
\end{cor}

\begin{proof}
 The assertion follows from
Proposition~\ref{flaprojective-cotorsion-pair-prop}
and Lemma~\ref{R-cotorsion-A-module-periodicity-lemma} by virtue of
Theorem~\ref{cotorsion-periodity-type-equivalent-conditions}\,%
(4)\,$\Rightarrow$\,(2) applied to the cotorsion pair
$(\sA,\sB)=(A\Modl_{A/R\dflpr},\>A\Modl^{R\dcot})$
in the category of left $A$\+modules $\sK=A\Modl$.
\end{proof}

\begin{cor} \label{flaprojective-conjecture-corollary}
 For any homomorphism of associative rings $R\rarrow A$,
the following ten conditions are equivalent:
\begin{enumerate}
\item the cotorsion pair $(A\Modl_{A/R\dflpr},\>A\Modl^{R\dcot})$ in
$A\Modl$ is of the flat type;
\item the triangulated functor\/ $\Hot(A\Modl_\proj)\rarrow
\sD(A\Modl_{A/R\dflpr})$ induced by the inclusion of additive/exact
categories $A\Modl_\proj\rarrow A\Modl_{A/R\dflpr}$ is a triangulated
equivalence;
\item the triangulated functor\/ $\Hot(A\Modl_\proj)\rarrow
\sD(A\Modl_{A/R\dflpr})$ induced by the inclusion of additive/exact
categories $A\Modl_\proj\rarrow A\Modl_{A/R\dflpr}$ is essentially
surjective;
\item the triangulated functor\/ $\sD(A\Modl_{A/R\dflpr})\rarrow
\sD(A\Modl_\flat)$ induced by the inclusion of exact categories
$A\Modl_{A/R\dflpr}\rarrow A\Modl_\flat$ is a triangulated equivalence;
\item any complex in $A\Modl_{A/R\dflpr}$ that is Becker-coacyclic
in $A\Modl_\flat$ is also Becker-coacyclic in $A\Modl_{A/R\dflpr}$;
\item the triangulated functor\/ $\sD^\bctr(A\Modl^\cot)\rarrow
\sD^\bctr(A\Modl^{R\dcot})$ induced by the inclusion of exact categories
$A\Modl^\cot\rarrow A\Modl^{R\dcot}$ is a triangulated equivalence;
\item the triangulated functor\/ $\sD^\bctr(A\Modl^\cot)\rarrow
\sD^\bctr(A\Modl^{R\dcot})$ induced by the inclusion of exact categories
$A\Modl^\cot\rarrow A\Modl^{R\dcot}$ is essentially surjective;
\item the triangulated functor\/ $\sD^\bctr(A\Modl^{R\dcot})\rarrow
\sD^\bctr(A\Modl)$ induced by the inclusion of exact/abelian categories
$A\Modl^{R\dcot}\rarrow A\Modl$ is a triangulated equivalence;
\item any complex in $A\Modl^{R\dcot}$ that is Becker-contraacyclic
in $A\Modl$ is also Becker-contraacyclic in $A\Modl^{R\dcot}$;
\item in any acyclic complex of left $A$\+modules whose terms are
$A/R$\+flaprojective and whose modules of cocycles are flat as
$A$\+modules, the modules of cocycles are actually $A/R$\+flaprojective.
\end{enumerate}
\end{cor}

\begin{proof}
 This is a corollary of
Theorem~\ref{sandwiched-equivalent-conditions-theorem}.
 Consider the three cotorsion pairs
$(\VF,\CA)\allowbreak=(A\Modl_\proj,\>A\Modl)$,
\ $(\sF,\sC)=(A\Modl_{A/R\dflpr},\>A\Modl^{R\dcot})$, and
$(\Flat,\Cot)=(A\Modl_\flat,\>A\Modl^\cot)$ in the category of left
$A$\+modules $\sK=A\Modl$.
 Then the assumptions of
Theorem~\ref{sandwiched-equivalent-conditions-theorem} are satisfied
by Examples~\ref{projective-cotorsion-pair-is-very-flat-type-ex}
and~\ref{flat-cotorsion-pair-in-modules-flat-type},
Lemma~\ref{restriction-of-scalars-cotorsion},
and Proposition~\ref{flaprojective-cotorsion-pair-prop}.

 The conditions~(1\+-10) of the corollary correspond to
the conditions~(1\+-10) of the theorem, with the following changes:
\begin{itemize}
\item Condition~(10b) of
Theorem~\ref{sandwiched-equivalent-conditions-theorem} always holds
in the situation at hand by
Lemma~\ref{R-cotorsion-A-module-periodicity-lemma}, so it is dropped.
\item One has $\sD^\bco(A\Modl_\flat)=\sD(A\Modl_\flat)$ by
Lemma~\ref{ac=ac-bco-equivalent-conditions}(4)
or~\cite[Theorem~7.18]{Pphil}, and also
$\sD^\bco(A\Modl_{A/R\dflpr})=\sD(A\Modl_{A/R\dflpr})$
by Corollary~\ref{flaprojective-coacyclic=acyclic}.
 For these reasons, all mentions of the coderived categories in
conditions~(2\+-4) are replaced with the conventional
derived categories.
\item The exact category $A\Modl_\proj$ is split, so all mentions of
its (co)derived category $\sD^\bco(A\Modl_\proj)$ or
$\sD(A\Modl_\proj)$ are replaced with the homotopy category
$\Hot(A\Modl_\proj)$.
\end{itemize}
\end{proof}

 The property of Corollary~\ref{flaprojective-conjecture-corollary}(10)
was already formulated as
Conjecture~\ref{flaprojective-conjecture} in the Introduction.
 Let us state it again with the added benefit of the whole list of
equivalent formulations.

\begin{conj}[Flaprojective Conjecture]
 All the equivalent conditions of
Corollary~\ref{flaprojective-conjecture-corollary} hold for all
homomorphisms of associative rings $R\rarrow A$.
\end{conj}

 In the next corollary, it is assumed that flat left $R$\+modules
have finite projective dimensions.
 In particular, all Noetherian commutative rings $R$ of finite Krull
dimension have this property~\cite[Corollaire~II.3.2.7]{RG}, as well as
all associative rings $R$ of the cardinality not exceeding~$\aleph_n$
for a finite integer~$n$ \,\cite[Th\'eor\`eme~7.10]{GJ}.

\begin{cor}
 Let $R\rarrow A$ be a homomorphism of associative rings.
 Assume that all flat left $R$\+modules have finite projective
dimensions.
 Then all the equivalent conditions of
Corollary~\ref{flaprojective-conjecture-corollary} hold for the ring
homomorphism $R\rarrow A$.
 In other words, Conjecture~\ref{flaprojective-conjecture} is true
for $R\rarrow A$.
\end{cor}

\begin{proof}
 The point is that the projective dimensions of modules are never
increased by transfinitely iterated extensions (in view of
the Eklof lemma, see Lemma~\ref{eklof-lemma}).
 Furthermore, if $G$ is a flat left $R$\+module of projective
dimension~$n$, then the projective dimension of the $A$\+module
$A\ot_RG$ does not exceed~$n$.
 If all flat left $R$\+modules have finite projective dimensions,
then such projective dimensions are uniformly bounded by some fixed
integer~$d$, by Lemma~\ref{finite-projective-dimensions-are-bounded}.
 In view of Proposition~\ref{flaprojective-cotorsion-pair-prop},
it follows that the projective dimensions of all $A/R$\+flaprojective
left $A$\+modules do not exceed~$d$.

 So, in the situation at hand, the cotorsion pair
$(A\Modl_{A/R\dflpr},\>A\Modl^{R\dcot})$ in $A\Modl$ is of the very
flat type in the sense of Section~\ref{very-flat-type-secn}.
 By Example~\ref{very-flat-type-is-flat-type-example}, it follows
that the cotorsion pair $(A\Modl_{A/R\dflpr},\>A\Modl^{R\dcot})$ is
of the flat type.
 Thus the condition of
Corollary~\ref{flaprojective-conjecture-corollary}(1) is satisfied
for the ring homomorphism $R\rarrow A$.
 See also the discussion in~\cite[Section~8.4]{Pdomc}.

 Let us spell out a direct proof of the assertion that
Conjecture~\ref{flaprojective-conjecture} (i.~e., the property
of Corollary~\ref{flaprojective-conjecture-corollary}(10)) holds
whenever all flat left $R$\+modules have finite projective dimensions.
 According to the argument above, there exists an integer~$d\ge0$
such that the projective dimensions of all $A/R$\+flaprojective
left $A$\+modules do not exceed~$d$.
 Let $F^\bu$ be a complex of $A/R$\+flaprojective left $A$\+modules
such that $F^\bu$ is an acyclic complex with flat $A$\+modules of
cocycles.

 Pick a contractible complex of projective left $A$\+modules $P_0^\bu$
together with a termwise surjective morphism of complexes of
$A$\+modules $P_0^\bu\rarrow F^\bu$.
 Denote by $F_1^\bu$ the kernel of the morphism of complexes
$P_0^\bu\rarrow F^\bu$, and pick a contractible complex of projective
left $A$\+modules $P_1^\bu$ together with a termwise surjective
morphism of complexes of modules $P_1^\bu\rarrow F_1^\bu$.
 Proceeding in this way, we arrive to a finite exact sequence of
complexes of $A$\+modules $0\rarrow F_d^\bu\rarrow P_{d-1}^\bu
\rarrow\dotsb\rarrow P_0^\bu\rarrow F^\bu\rarrow0$, where $P_i^\bu$
are contractible complexes of projective $A$\+modules.
 Since the projective dimensions of the terms of the complex $F^\bu$
do not exceed~$d$, the complex $F_d^\bu$ is a complex of projective
$A$\+modules.

 Notice that all the complexes $F_i^\bu$ are acyclic (since the kernel
of a termwise surjective morphism of acyclic complexes is acyclic).
 Furthermore, the functor assigning to an acyclic complex its module of
cocycles is exact.
 Denote the image of the differential $F^{n-1}\rarrow F^n$ by $G^n$,
the images of the differentials $F_i^{n-1}\rarrow F_i^n$ by $G_i^n$,
and the images of the differentials $P_i^{n-1}\rarrow P_i^n$ by~$Q_i^n$.
 Then, for every degree $n\in\boZ$, we have a finite exact sequence of
$A$\+modules $0\rarrow G_d^n\rarrow Q_{d-1}^n\rarrow\dotsb\rarrow
Q_0^n\rarrow G^n\rarrow0$.

 The $A$\+modules $G^n$ are flat by assumption, and the $A$\+modules
$Q_i^n$ are projective (since $P_i^\bu$ are contractible complexes of
projective $A$\+modules).
 It follows that the $A$\+modules $G_d^n$ are flat.
 Now $F_d^\bu$ is an acyclic complex of projective $A$\+modules with
flat $A$\+modules of cocycles, so the flat and projective periodicity
theorem~\cite[Theorem~2.5]{BG}, \cite[Remark~2.15]{Neem} is applicable,
and we arrive to the conclusion that the $A$\+modules $G_d^n$ are
actually projective.
 Thus the projective dimensions of the $A$\+modules $G^n$ do not
exceed~$d$ (in view of the same finite exact sequence above).

 Finally, we recall that $(A\Modl_{A/R\dflpr},\>A\Modl^{R\dcot})$ is
a hereditary cotorsion pair in $A\Modl$ by
Proposition~\ref{flaprojective-cotorsion-pair-prop}, so the full
subcategory of $A/R$\+flaprojective left $A$\+modules
$A\Modl_{A/R\dflpr}$ is resolving in $A\Modl$.
 The resolution dimensions of the $A$\+modules $G^n$ with respect to
$A\Modl_{A/R\dflpr}$ do not exceed the projective dimensions of $G^n$,
which do not exceed~$d$.
 For every $n\in\boZ$, the finite exact sequence
$0\rarrow G^n\rarrow F^n\rarrow F^{n+1}\rarrow\dotsb\rarrow F^{n+d-1}
\rarrow G^{n+d}\rarrow0$ is a resolution of the $A$\+module $G^{n+d}$
with $A$\+modules $F^i\in A\Modl_{A/R\dflpr}$.
 Since the resolution dimension does not depend on the choice of
a resolution~\cite[Lemma~2.1]{Zhu}, \cite[Proposition~2.3]{Sto0},
\cite[Corollary~A.5.2]{Pcosh}, we can conclude that
$G^n\in A\Modl_{A/R\dflpr}$.
 Thus the $A$\+modules of cocycles of the complex $F^\bu$ are
$A/R$\+flaprojective, as desired.
\end{proof}

 A nonaffine (quasi-compact semi-separated scheme) version of
Corollary~\ref{flaprojective-conjecture-corollary}(8)
for Noetherian commutative rings $R$ of finite Krull dimension
can be found in~\cite[Theorem~15.22]{Pdomc}.

\Section{The Relatively Cotorsion Conjecture}

 Let $R\rarrow A$ be a homomorphism of associative rings.
 Assume that the left $R$\+module $A$ is projective.
 We will say that a left $A$\+module $F$ is \emph{$R$\+projective}
if $F$ is projective as a left $R$\+module.
 In this section, we are interested in left $A$\+modules $F$ that are
simultaneously flat (i.~e., $A$\+flat) and $R$\+projective.

 We will say that a left $A$\+module $C$ is \emph{cotorsion relative
to~$R$} (or \emph{$A/R$\+cotorsion}) if $\Ext^1_A(F,C)=0$ for all
$R$\+projective flat left $A$\+modules~$F$.
 Denote the class of $R$\+projective left $A$\+modules by
$A\Modl^{R\dproj}\subset A\Modl$, the class of $R$\+projective flat
left $A$\+modules by $A\Modl_\flat^{R\dproj}=A\Modl^{R\dproj}\cap
A\Modl_\flat\subset A\Modl$, and the class of $A/R$\+cotorsion
left $A$\+modules by $A\Modl^{A/R\dcot}\subset A\Modl$.

 Clearly, all projective left $A$\+modules are flat and $R$\+projective
(since we are assuming that the left $R$\+module $A$ is projective).
 So we have $A\Modl_\proj\subset A\Modl_\flat^{R\dproj}\subset
A\Modl_\flat$ and $A\Modl^\cot\subset A\Modl^{A/R\dcot}$
(in the notation
of Examples~\ref{projective-cotorsion-pair-is-very-flat-type-ex}
and~\ref{flat-cotorsion-pair-in-modules-flat-type}).

\begin{prop} \label{relatively-cotorsion-cotorsion-pair-prop}
 Let $R\rarrow A$ be a homomorphism of associative rings such that
$A$ is projective as a left $R$\+module.
 Then the pair of classes
$(A\Modl_\flat^{R\dproj},\>A\Modl^{A/R\dcot})$ is a hereditary
complete cotorsion pair generated by a set of objects in $A\Modl$.
\end{prop}

\begin{proof}
 The key step is to prove that the class of modules
$A\Modl_\flat^{R\dproj}$ is deconstructible in $A\Modl$.
 This is provable using the Hill lemma~\cite[Theorem~7.10]{GT},
\cite[Theorem~2.1]{Sto-hill}.
 Instead of spelling out the details of the argument, we rely on
several references.
 First of all, the class of projective left $R$\+modules is
deconstructible in $R\Modl$ by Kaplansky's
theorem~\cite[Theorem~1]{Kap}, \cite[Lemma~7.12 and
Corollary~7.14]{GT}, \cite[Proposition~2.9(1)]{Sto-hill}.
 Secondly, it follows that the class of $R$\+projective left
$A$\+modules is deconstructible in $A\Modl$
by~\cite[Proposition~2.3]{Pctrl}.
 On the other hand, the class of flat left $A$\+modules is
deconstructible in $A\Modl$ by~\cite[Lemma~1 and Proposition~2]{BBE}
or~\cite[Lemma~6.17]{GT}.
 Finally, by~\cite[Proposition~2.9(2)]{Sto-hill}, the intersection of
any two deconstructible classes is deconstructible.
 Thus the class $A\Modl_\flat^{R\dproj}=A\Modl^{R\dproj}\cap
A\Modl_\flat$ is deconstructible.

 Now let $\sS$ be a set of left $A$\+modules such as
$A\Modl^{R\dproj}_\flat=\Fil(\sS)$.
 Let $(\sF,\sC)$ be the cotorsion pair generated by $\sS$ in $A\Modl$.
 Then we have $A\Modl_\flat^{R\dproj}\subset\sF$ by
Lemma~\ref{eklof-lemma}, and the class $A\Modl_\flat^{R\dproj}$ is
generating in $A\Modl$ since $A\Modl_\proj\subset
A\Modl_\flat^{R\dproj}$.
 Thus, by Theorem~\ref{loc-pres-eklof-trlifaj}
or~\ref{efficient-eklof-trlifaj}, the cotorsion pair $(\sF,\sC)$
is complete, and $\sF=A\Modl_\flat^{R\dproj}$.
 Hence $\sC=A\Modl^{A/R\dcot}$.
 It remains to observe that the class $A\Modl_\flat^{R\dproj}$ is
obviously closed under kernels of epimorphisms in $A\Modl$; so
the cotorsion pair $(A\Modl_\flat^{R\dproj},\>A\Modl^{A/R\dcot})$
is hereditary.
\end{proof}

\begin{lem} \label{R-projective-A-flat-periodicity-lemma}
 Let $F^\bu$ be a complex of $R$\+projective flat left $A$\+modules
such that the $A$\+modules of cocycles of $F^\bu$ are flat.
 Then the $A$\+modules of cocycles of $F^\bu$ are also
$R$\+projective.
\end{lem}

\begin{proof}
 Consider $F^\bu$ as a complex of $R$\+modules.
 By assumption, the terms of the complex $F^\bu$ are projective
$R$\+modules.
 Since we are assuming that $A$ is a projective (hence flat) left
$R$\+module, all flat left $A$\+modules are flat as $R$\+modules.
 So the $R$\+modules of cocycles of $F^\bu$ are flat.
 By the flat and projective periodicity theorem for
$R$\+modules (\cite[Remark~2.15]{Neem} or~\cite[Theorem~2.5]{BG}
with~\cite[Proposition~7.6]{CH}), it follows that the $R$\+modules
of cocycles of $F^\bu$ are projective.
\end{proof}

\begin{cor} \label{relatively-cotorsion-conjecture-corollary}
 Let $R\rarrow A$ be a homomorphism of associative rings such that
$A$ is projective as a left $R$\+module.
 Then the following eleven conditions are equivalent:
\begin{enumerate}
\setcounter{enumi}{-1}
\item all acyclic complexes in the exact category
$A\Modl_\flat^{R\dproj}$ are Becker-coacyclic, that is\/
$\Ac(A\Modl_\flat^{R\dproj})\subset\Ac^\bco(A\Modl_\flat^{R\dproj})$
(or equivalently, the classes of acyclic and Becker-coacyclic
complexes in $A\Modl_\flat^{R\dproj}$ coincide, that is\/
$\Ac^\bco(A\Modl_\flat^{R\dproj})=\Ac(A\Modl_\flat^{R\dproj})$);
\item the cotorsion pair $(A\Modl_\flat^{R\dproj},\>A\Modl^{A/R\dcot})$
in $A\Modl$ is of the flat type;
\item the triangulated functor\/ $\Hot(A\Modl_\proj)\rarrow
\sD^\bco(A\Modl_\flat^{R\dproj})$ induced by the inclusion of
additive/exact categories $A\Modl_\proj\rarrow A\Modl_\flat^{R\dproj}$
is a triangulated equivalence;
\item the triangulated functor\/ $\Hot(A\Modl_\proj)\rarrow
\sD^\bco(A\Modl_\flat^{R\dproj})$ induced by the inclusion of
additive/exact categories $A\Modl_\proj\rarrow A\Modl_\flat^{R\dproj}$
is essentially surjective;
\item the triangulated functor\/ $\sD^\bco(A\Modl_\flat^{R\dproj})
\rarrow\sD^\bco(A\Modl_\flat)$ induced by the inclusion of exact
categories $A\Modl_\flat^{R\dproj}\rarrow A\Modl_\flat$ is
a triangulated equivalence;
\item any complex in $A\Modl_\flat^{R\dproj}$ that is Becker-coacyclic
in $A\Modl_\flat$ is also Becker-coacyclic in $A\Modl_\flat^{R\dproj}$;
\item the triangulated functor\/ $\sD^\bctr(A\Modl^\cot)\rarrow
\sD^\bctr(A\Modl^{A/R\dcot})$ induced by the inclusion of exact
categories $A\Modl^\cot\rarrow A\Modl^{A/R\dcot}$ is a triangulated
equivalence;
\item the triangulated functor\/ $\sD^\bctr(A\Modl^\cot)\rarrow
\sD^\bctr(A\Modl^{A/R\dcot})$ induced by the inclusion of exact
categories $A\Modl^\cot\rarrow A\Modl^{A/R\dcot}$ is essentially
surjective;
\item the triangulated functor\/ $\sD^\bctr(A\Modl^{A/R\dcot})\rarrow
\sD^\bctr(A\Modl)$ induced by the inclusion of exact/abelian
categories $A\Modl^{A/R\dcot}\rarrow A\Modl$ is a triangulated
equivalence;
\item any complex in $A\Modl^{A/R\dcot}$ that is Becker-contraacyclic
in $A\Modl$ is also Becker-contraacyclic in $A\Modl^{A/R\dcot}$;
\item in any acyclic complex of $A$\+modules whose terms are
$A/R$\+cotorsion, the modules of cocycles are also $A/R$\+cotorsion.
\end{enumerate}
\end{cor}

\begin{proof}
 This is a corollary of
Theorems~\ref{cotorsion-periodity-type-equivalent-conditions}
and~\ref{sandwiched-equivalent-conditions-theorem}.
 Consider the three cotorsion pairs $(\VF,\CA)=(A\Modl_\proj,\>A\Modl)$,
\ $(\sF,\sC)=(A\Modl_\flat^{R\dproj},\>A\Modl^{A/R\dcot})$, and
$(\Flat,\Cot)=(A\Modl_\flat,\>A\Modl^\cot)$ in the category of left
$A$\+modules $\sK=A\Modl$.
 Then the assumptions of
Theorem~\ref{sandwiched-equivalent-conditions-theorem} are satisfied
by Examples~\ref{projective-cotorsion-pair-is-very-flat-type-ex}
and~\ref{flat-cotorsion-pair-in-modules-flat-type}, and
by Proposition~\ref{relatively-cotorsion-cotorsion-pair-prop}.

 The conditions~(1\+-10) of the corollary correspond to
the conditions~(1\+-10) of
Theorem~\ref{sandwiched-equivalent-conditions-theorem},
with the following changes:
\begin{itemize}
\item Condition~(10a) of
Theorem~\ref{sandwiched-equivalent-conditions-theorem} always holds
in the situation at hand by
Lemma~\ref{R-projective-A-flat-periodicity-lemma}, so it is dropped.
\item The exact category $A\Modl_\proj$ is split, so all mentions of
its coderived category $\sD^\bco(A\Modl_\proj)$ are replaced with
the homotopy category $\Hot(A\Modl_\proj)$.
\end{itemize}

 The corollary also contains an additional equivalent condition~(0),
which was not mentioned in
Theorem~\ref{sandwiched-equivalent-conditions-theorem}.
 We have (0)~$\Longleftrightarrow$~(10) by
Theorem~\ref{cotorsion-periodity-type-equivalent-conditions},
whose assumptions are satisfied for the cotorsion pair
$(\sA,\sB)=(A\Modl_\flat^{R\dproj},\>A\Modl^{A/R\dcot})$ in $\sK=A\Modl$
by Proposition~\ref{relatively-cotorsion-cotorsion-pair-prop}.
\end{proof}

 The property of
Corollary~\ref{relatively-cotorsion-conjecture-corollary}(10)
was already formulated as
Conjecture~\ref{relatively-cotorsion-conjecture} in the Introduction.
 Let us state it again with the added benefit of the whole list of
equivalent formulations.

\begin{conj}[Relatively Cotorsion Conjecture]
 Let $R\rarrow A$ be a homomorphism of associative rings making $A$
a projective left $R$\+module.
 Then all the equivalent conditions of
Corollary~\ref{relatively-cotorsion-conjecture-corollary} hold.
\end{conj}

\end{document}